\documentclass[11pt]{amsart}
\usepackage{amsmath, amssymb, amscd, mathrsfs, slashed, url,color,extsizes,xcolor,multicol,indentfirst,latexsym,bm,graphicx,subfigure,esint,float,verbatim,comment,epsfig,dsfont,amsthm,amsfonts,amsbsy,tikz-cd,amsthm,thmtools,bbm,adjustbox,enumerate,calligra,fancyhdr,accents,enumitem,etoolbox}

\usepackage[all,cmtip]{xy}

\usepackage[top=1in, bottom=1in, left=1.25in, right=1.25in]{geometry}
\usepackage[colorlinks]{hyperref}
\usepackage[toc,page]{appendix}

\DeclareMathAlphabet{\mathcalligra}{T1}{calligra}{m}{n}

\declaretheoremstyle[
headfont=\color{blue}\normalfont\bfseries,
bodyfont=\color{blue}\normalfont\itshape,
]{colored}

\usepackage[utf8]{inputenc}
\usepackage[english]{babel}

\setlist[enumerate]{
    label={(\roman*)},
}

\AtBeginEnvironment{definition}{%
    \setlist[enumerate]{
        label={(\roman*)}, 
        font=\normalfont,
        before=\normalfont
    }%
}

\newtheorem{theorem}{Theorem}[section]

\newtheorem{corollary}[theorem]{Corollary}

\newtheorem{definition}[theorem]{Definition}
\newtheorem{question}[theorem]{Question}

\newtheorem{lemma}[theorem]{Lemma}

\newtheorem{proposition}[theorem]{Proposition}

\theoremstyle{definition}
\newtheorem{example}[theorem]{Example}

\newtheorem{remark}{Remark}



\newcommand{\RR}{\mathbb{R}}

\newcommand{\ZZ}{\mathbb{Z}}

\newcommand{\CC}{\mathbb{C}}

\newcommand{\Int}{\mathrm{Int}}

\newcommand{\ZT}{\ZZ/2}

\newcommand{\Zl}{\mathcal{Z}}
\newcommand{\I}{\mathcal{I}}
\newcommand{\Sl}{\mathcal{S}}
\newcommand{\vp}{\mathfrak{p}}

\newcommand{\p}{\partial}
\title{Calabi surgery for $\ZT$ harmonic $1$-forms}

\begin{document}

\author{Jiahuang Chen, Siqi He, Dashen Yan} 

\address{AMSS}
\email{chenjiahuang@amss.ac.cn, sqhe@amss.ac.cn}	
\address{Department of Mathematics, Stanford University}
\email{yandashen19@gmail.com}

\maketitle
\begin{abstract}
We prove a $2$-valued analogue of Calabi's intrinsic harmonicity theorem and use it to introduce the Calabi surgery method, a surgery theory for $\ZT$ harmonic $1$-forms. Once the ambient metric is allowed to vary, the construction of new $\ZT$ harmonic forms can be reduced to cutting and pasting closed $2$-valued $1$-forms, matching local harmonic models, and controlling the transitivity of the resulting singular foliation. For these constructions, the Nash--Moser-type analytic deformation problem that arises in singular gluing is replaced by local model matching and a global dynamical condition on the foliation. The resulting procedure gives a flexible way to construct and modify $\ZT$ harmonic $1$-forms under weak regularity assumptions. As applications, we obtain connected-sum and local replacement theorems, blow up isolated ordinary zeros by prescribed Euclidean models, split smooth $\vec{k}$-nondegenerate branching components, and desingularize graphic singular sets in dimensions $3$ and $4$ with suitable resolution models. 
\end{abstract}

\section{Introduction}
A $\ZT$ harmonic $1$-form on a Riemannian manifold $(M,g)$ is an
$\I$-valued harmonic $1$-form $\alpha$ on $M\setminus\Sl$, where
$\I\to M\setminus\Sl$ is a flat real line bundle and
$\Sl\subset M$ is a closed subset of Hausdorff codimension at least
two.  
Locally, $\alpha$ can be viewed as a harmonic
$1$-form defined only up to sign, and this sign ambiguity is recorded
by the monodromy of $\I$.

The $\ZT$ harmonic $1$-forms were introduced by Taubes in his compactness
theory for gauge-theoretic equations, where they arise as
limiting objects
\cite{TaubesPSL,TaubesASD,TaubesZero}.  They have since appeared in a
range of problems in differential geometry and gauge theory, including
compactness and enumerative questions, generalized Seiberg--Witten
equations, calibrated geometry, and the study of singular foliations
\cite{DoanWalpuski2021,WalpuskiZhang2021,HeWentworthZhang2026,He2023Branched,Chen2025Perturbation,HeChenSalmZ2Rn}.

Much of the difficulty comes from the behavior near the singular set $\Sl$.
Suppose that $\Sl$ is a smooth codimension-two submanifold.  Since
$\I$ has nontrivial monodromy around a small meridian of $\Sl$, the
transverse expansion has half-integral powers.  In a normal complex
coordinate $z$ transverse to $\Sl$, $\alpha$ has the following expansion
\begin{equation}\label{eqn}
\alpha \sim \Re(Bz^{1/2}dz)+\mathcal{O}(|z|^{\frac32}).
\end{equation}
Following \cite{donaldsondeformation2019}, see also
\cite{Takahashi23Z2,Parker2023}, $\alpha$ is called \emph{nondegenerate} if
$B$ is nowhere vanishing along $\Sl$.

This local asymptotic model is used in existing deformation and gluing results.
The usual method is to build an approximate solution by cutting and
pasting local models, and then to deform this approximate solution to
a genuine one.  Various local models have been constructed in
\cite{TWI,TaubesWuTopological,chenHe2024,haydys2025new,yan2025construction,donaldson2025twistor},
and some of them have been used in gluing constructions
\cite{HeParker2024,salm2024construction,yan2025nondegenerate}.

From the analytic point of view, gluing and deformation for $\mathbb Z/2$ harmonic $1$-forms are already subtle in the smooth nondegenerate setting. 
Existing gluing constructions usually proceed by cutting and pasting local models to obtain a candidate object and then solving a deformation problem for the harmonic equation. 
If the singular set is kept fixed, the resulting linearized problem is not Fredholm on the natural function spaces; it has an infinite-dimensional obstruction along $\mathcal S$. 
A central point in the deformation theories of Donaldson \cite{donaldsondeformation2019} and Parker \cite{Parker2023}, see also \cite{HeWalpuskiParkerUniversal}, is that this obstruction is removed by allowing the singular set to move.  One is then led to solve simultaneously for the form and for the singular set, and the analysis is carried out by a Nash--Moser argument. 
Related analytic results for Dirac operators twisted by ramified Euclidean line bundles appear in \cite{BeraWalpuski2025}. 
This analytic framework typically requires the singular set to be smooth and the leading term in \eqref{eqn} to be nondegenerate.

In this paper we take a different route. 
We first perform the surgery at the level of $2$-valued closed $1$-forms and then use intrinsic harmonicity to choose an ambient metric for which the resulting form is harmonic. 
Thus the global analytic correction problem is replaced, in our construction, by two geometric tasks: matching the local models and controlling the transitivity of the induced singular foliation.

\subsection{Intrinsic harmonicity and Calabi surgery}
The mechanism underlying \emph{Calabi surgery} is \emph{intrinsic
harmonicity}. A $\ZT$ closed $1$-form $(v,\Sl,\I)$ is called
\emph{intrinsically harmonic} if there exists a Riemannian metric
$g'$ on $M$ for which it is $\ZT$ harmonic.  This leads to the
following basic question.

\begin{question}[Calabi's intrinsic harmonicity question for
$\ZT$ closed $1$-forms]
Let $(v,\Sl,\I)$ be a $\ZT$ closed $1$-form on a closed manifold
$M$.  Under what conditions does there exist a
Riemannian metric $g'$ on $M$ such that $(v,\Sl,\I)$ is
$\ZT$ harmonic with respect to $g'$?
\end{question}

This is the $2$-valued counterpart of a classical problem studied by
Calabi \cite{intrinsic-Calabi}.  For an ordinary closed $1$-form
$\alpha$, Calabi established an intrinsic harmonicity criterion under
a Morse assumption on its zero set, and Volkov later extended this
characterization to a more general setting
\cite{intrinsic-Calabi,intrinsic-Volkov}.  The Calabi--Volkov theorem
states that a closed $1$-form on a closed manifold is intrinsically
harmonic if and only if it is locally intrinsically harmonic and
transitive.  Here transitivity means that every point outside the
zero set lies on a loop positively transverse to the foliation
$\ker\alpha$.  Thus the criterion separates a local condition near
the zero set from a global dynamical condition on the induced
foliation. We also refer \cite{Liu2024FiniteGraphSingularSet,Zhang2013GluingTechniques} for related ideas in calibrated geometry.

This local--global structure persists in the $2$-valued setting,
although the branching set and the twisting by $\I$ introduce
additional complications.  We prove that every $\ZT$ harmonic
$1$-form on a closed manifold is transitive.  Conversely, local
intrinsic harmonicity and local $\star$-exactness near the total zero
set, together with transitivity, allow one to reconstruct a global
metric making the form harmonic.

\begin{theorem}
\label{thm:intro_intrinsic}
Let $M$ be a closed oriented manifold, and let $(v,\Sl,\I)$ be a
$\ZT$ closed $1$-form on $M$.  Suppose that $(v,\Sl,\I)$ is locally
intrinsically harmonic, locally $\star$-exact, and transitive.  Then
there exists a Riemannian metric $g'$ on $M$ such that $(v,\Sl,\I)$
is $\ZT$ harmonic with respect to $g'$.
\end{theorem}

This gives the main technique of the paper.  Starting from existing
harmonic forms and local models, one performs a cut-and-paste
construction at the level of $\ZT$ closed $1$-forms, thereby modifying
the ambient manifold or the zero or singular set.  One then verifies
that the resulting form satisfies the criterion in
Theorem~\ref{thm:intro_intrinsic}, which yields a new ambient metric
with respect to which it is harmonic.  We call this procedure
\emph{Calabi surgery}.

Thus the construction separates into two tasks: constructing and
matching suitable local models, and controlling the global behavior
of the kernel foliation.  Once these are achieved, one need not solve
a global singular non-Fredholm correction equation.

\subsection{Applications of Calabi surgery}
We now turn to several applications of Calabi surgery. Since the construction does not require solving a global singular
linearized equation, singular components left unchanged by the surgery
usually need not satisfy additional smoothness or nondegeneracy
assumptions. Stronger hypotheses enter only through the local models
being inserted.

The first application is a connected-sum theorem, which generalizes
the construction of \cite{HeParker2024} while removing any
nondegeneracy and regularity assumptions on the original singular sets.

\begin{theorem}
\label{thm:intro_connected_sum}
Let $(M_i^n,g_i)$, $i=1,2$, be closed oriented Riemannian manifolds,
and let $(v_i,\Sl_i,\I_i)$ be nonzero $\ZT$ harmonic $1$-forms on
$(M_i,g_i)$.  Then the connected sum $M=M_1\#M_2$
admits a Riemannian metric $g_M$ and a $\ZT$ harmonic $1$-form
$(v_M,\Sl_M,\I_M)$ on $(M,g_M)$.  Moreover, $(v_M,\Sl_M,\I_M)$ is
obtained by gluing the two original forms, and $\Sl_M=\Sl_1\sqcup \Sl_2 .$
\end{theorem}

The connected-sum neck is inserted in regions where the original
forms are nonvanishing, leaving the singular sets
$\Sl_1$ and $\Sl_2$ unchanged.  The main point is therefore to
preserve transitivity across the neck.  Once this is verified,
Theorem~\ref{thm:intro_intrinsic} yields the required metric, with no
regularity or nondegeneracy condition imposed on the original
singular sets.

The next two applications are based on a common local replacement
principle for components of the total zero set. Let $Y$ be such a
component for a given $\ZT$ harmonic $1$-form. A new harmonic model
near $Y$ can be substituted for the original one, provided that it
satisfies the required compatibility and exactness conditions and is
sufficiently close to the original form on the gluing region. Positive
loops there can be chosen uniformly away from $Y$ and remain positive
under a sufficiently small modification of form. The resulting closed form is
therefore transitive, and Theorem~\ref{thm:intro_intrinsic} supplies a
metric for which it is harmonic. See
Section~\ref{sec:changing_zero_sets} for details.

Applied to an isolated ordinary zero, this principle gives the
following blow-up theorem, generalizing related constructions in
\cite{yan2025construction,yan2025nondegenerate}.

\begin{theorem}
\label{thm:intro_blowup}
Let $(v,\Sl,\I)$ be a $\ZT$ harmonic $1$-form on a closed oriented
Riemannian manifold $(M^n,g)$, and let $p\in M\setminus\Sl$ be an
isolated zero of $v$.  Suppose that, in geodesic coordinates centered
at $p$,
\[
v=d\bigl(P_p(x)+\mathcal O(|x|^{d+1})\bigr),
\]
where $P_p$ is a homogeneous harmonic polynomial of degree $d\ge2$
on $\RR^n$ with $dP_p(x)\ne0$ for $x\ne0$.  Suppose also that there
is a $\ZT$ harmonic local model $(du,\Sl',\I')$ on $(\RR^n,g')$ such
that $\Sl'$ is compact, $g'$ is Euclidean outside a compact set,
\[
u(y)=P_p(y)+\mathcal O(|y|^{d-1})\quad\text{as }|y|\to\infty,
\]
and $\star_{g'}du$ is exact near infinity.  Then replacing a small
neighborhood of $p$ by a suitably scaled copy of this model produces
a Riemannian metric $g_M$ on $M$ and a $\ZT$ harmonic $1$-form
$(v_M,\Sl_M,\I_M)$ on $(M,g_M)$, with
$\Sl_M\cong\Sl\sqcup\Sl'$.
\end{theorem}

Thus, whenever an appropriate Euclidean model is available, an isolated ordinary
zero can be replaced by a small branching set whose local geometry is prescribed
by the model.

We next recall the higher-order nondegeneracy condition for smooth branching
sets. Let \((v,\Sl,\I)\) be a smooth \(\ZT\) harmonic \(1\)-form and write
\(\Sl=\Sl_1\sqcup\cdots\sqcup\Sl_m\). Near a connected component
\(\Sl_i\subset\Sl\), the local expansion has the form
\[
        v
        =
        d\Re\left(A_i(t)z^{k_i+1/2}
        + \mathcal{O}(|z|^{k_i+3/2})\right),
\]
where \(z\) is a normal complex coordinate and \(k_i\geq 1\). We say that \(v\)
is \(k_i\)-nondegenerate along \(\Sl_i\) if \(A_i\) is nowhere vanishing. If this
holds for every component, then \(v\) is called \(\vec k\)-nondegenerate, where
\(\vec k=(k_1,\ldots,k_m)\). This notion was studied in
\cite{HeSalm2026Metric}.

The isolated zero gluing theorem, combined with zonal harmonic polynomials and
the Euclidean models in Appendix~\ref{section:k-nondeg_on_Rn}, gives the following existence result for
\(\vec k\)-nondegenerate \(\ZT\) harmonic \(1\)-forms, generalizing
\cite{yan2025construction}.

\begin{corollary}\label{cor:prescribed-knondegenerate-unknots}
Let $M^3$ be a closed oriented three-manifold with $b_1(M)>0$.
For any $m\ge1$ and $\vec k=(k_1,\ldots,k_m)\in\mathbb N^m$,
there exist pairwise disjoint embedded balls
$B_1,\ldots,B_m\subset M$, a Riemannian metric $g'$, and a smooth
$\vec k$-nondegenerate $\ZT$ harmonic $1$-form
$(v',\Sl,\I')$ on $(M,g')$ such that
$\Sl=C_1\sqcup\cdots\sqcup C_m$,
where each $C_i\subset B_i$ is an unknot and $v'$ is
$k_i$-nondegenerate along $C_i$.
\end{corollary}

The same replacement principle also yields a splitting theorem for smooth higher-order branching components, provided that a certain obstruction vanishes. Given a $\vec{k}$-nondegenerate $\ZT$ harmonic $1$-form, for each connected component $S_{i}$ we define the \textit{winding class} $\mathfrak{m}(A_{i})\in H^{1}(S_{i},\mathbb{Z}/(2k_i+1))$ of the leading coefficient to be the \v{C}ech cocycle determined by local choices of the $(2k_i+1)$-th root of the section $A_{i}$. The class $\mathfrak{m}(A_{i})$ vanishes if and only if the section $A_{i}$ admits a global $(2k_i+1)$-th root. In particular, the examples constructed in Appendix \ref{section:k-nondeg_on_Rn} all have vanishing winding class; see Remark \ref{rem:winding}.

\begin{theorem}
\label{thm:intro_k-nondeg}
Let $(v,\Sl,\I)$ be a smooth $\vec k$-nondegenerate $\ZT$ harmonic
$1$-form on a closed oriented Riemannian manifold $(M^n,g)$, $n\ge3$,
with vanishing winding classes. Suppose further that $\Sl$ has trivial normal bundle in $M$.

Then there exist a Riemannian metric
$g'$ on $M$ and a smooth nondegenerate $\ZT$ harmonic $1$-form
$(v',\Sl',\I')$ on $(M,g')$. The splitting can be performed
arbitrarily close to $\Sl$, and near each component $S_i\subset\Sl$,
the new singular set consists of $2k_i-1$ copies
of $S_i$.
\end{theorem}

This resolves each higher-order branching component into ordinary
nondegenerate components.  Locally, the model
$d\Re(z^{k+1/2})$ is replaced by a product of half-power factors with
$2k-1$ nearby branch components.  The result complements the
deformation theory for $\vec k$-nondegenerate $\ZT$ harmonic $1$-forms
developed in \cite{HeSalm2026Metric}.

Finally, we consider $\ZT$ harmonic $1$-forms with graphic singular
sets in dimensions three and four.  Away from finitely many vertices,
such singular sets are smooth codimension-two submanifolds; near each
vertex, they are asymptotic to a cone.  Correspondingly, the form has
at each vertex a tangent cone given by a homogeneous
$\ZT$ harmonic model on Euclidean space.  The local structure and
deformation theory of these objects are closely related to the index
theory for graph-like singularities
\cite{HMT2025index,HeMazzeoGraphicL2}.

If the tangent cone at each vertex admits a resolution model in the
sense of Definition~\ref{def:resolution_model}, then the graphic
singular set can be desingularized at all vertices.

\begin{theorem}
\label{thm:intro_graphic_desing}
Let $(M^n,g)$ be a closed oriented Riemannian manifold with
$n\in{3,4}$, and let $(v,\Sl,\I)$ be a strongly nondegenerate
graphic $\ZT$ harmonic $1$-form on $(M,g)$. Suppose that the tangent
cone at every vertex of $\Sl$ admits a resolution model. Then there
exist a Riemannian metric $g'$ on $M$ and a nondegenerate $\ZT$
harmonic $1$-form $(v',\Sl',\I')$ on $(M,g')$ whose singular set
$\Sl'$ is a smooth codimension-two submanifold. Moreover, for every
neighborhood $\mathcal U$ of $\Sl$, the construction can be arranged
so that $\Sl'\subset\mathcal U$ and, after identifying the flat line
bundles, $v'=v$ on $M\setminus\mathcal U$.
\end{theorem}

These applications illustrate the flexibility of Calabi surgery.
Once compatible local models have been constructed, the passage from
local data to a global $\ZT$ harmonic $1$-form requires only control
of transitivity and metric reconstruction.  It therefore avoids
introducing additional smoothness or nondegeneracy hypotheses
for the purpose of solving a global analytic correction problem.

\subsection*{Acknowledgments}
We would like to thank Rafe Mazzeo, Clifford Taubes and Thomas Walpuski for helpful comments and disucssions. 

\section{Background on $\ZT$ harmonic 1-forms}\label{sec:preliminary}
In this section we collect the background and notation for $\mathbb Z/2$ harmonic
$1$-forms used throughout the paper. We first review $2$-valued differential
forms and basic conventions, and then recall the local asymptotic
models near smooth and graphic singular sets.
\subsection{Basic definitions and conventions}

Let $(M^n,g)$ be an oriented, connected, compact Riemannian manifold, possibly with nonempty boundary $\partial M$. Let $\Sl\subset \Int M$ be a closed subset of Hausdorff codimension at least two, and let $\I\to M\setminus \Sl$ be a flat real line bundle. A section of $\wedge^kT^*M\otimes \I$ is called an $\I$-valued, or $2$-valued, $k$-form.

The exterior differential, Hodge star, and covariant derivative are defined using the flat structure of $\I$. Thus, after choosing a local unit flat section $s$ and writing an $\I$-valued form as $w\otimes s$, one has $d(w\otimes s)=dw\otimes s$ and $\star_g(w\otimes s)=(\star_gw)\otimes s$; changing $s$ to $-s$ changes $w$ to $-w$.

\begin{definition}\label{def:z2-closed harmonic}
A $2$-valued closed $1$-form on $(M,g)$ is a triple $(v,\Sl,\I)$, where $v\in \Gamma(T^*M\otimes \I)$ satisfies $dv=0$ and $\int_{M\setminus\Sl}|v|^2<+\infty$. The closed set $\Sl$ is called the singular set of $v$. A $2$-valued closed $1$-form $(v,\Sl,\I)$ is called a \emph{$\mathbb{Z}/2$ closed $1$-form} if, in addition, $\int_{M\setminus\Sl}|\nabla v|^2<+\infty$. A $2$-valued, respectively $\mathbb{Z}/2$, closed $1$-form is called harmonic if $d\star_g v=0$.
\end{definition}

In the rest of the paper, all $2$-valued closed $1$-forms are assumed to have minimal singular set unless explicitly stated otherwise. By this we mean that there is no proper closed subset $\Sl'\subsetneq \Sl$, no flat real line bundle $\I'\to M\setminus\Sl'$, and no $2$-valued closed $1$-form $(v',\Sl',\I')$ such that $(\I',v')$ agrees with $(\I,v)$ over $M\setminus\Sl$ and such that, for every $x\in \Sl\setminus\Sl'$, there exists a ball $B\subset M\setminus\Sl'$ centered at $x$ with $\int_{B\setminus \Sl} |\nabla v'|^2<+\infty .$

In particular, if $(v,\Sl,\I)$ is a minimal $\mathbb{Z}/2$ closed $1$-form, then near every point of $\Sl$ either the line bundle $\I$ has nontrivial local monodromy, or no extension of the same form across that point satisfies the above local $L^2$ bound for its covariant derivative. Thus, for a minimal $\mathbb{Z}/2$ closed $1$-form, the singular set is part of the actual branching or analytic singularity.
\begin{definition}
The zero set of $v$ on $M\setminus\Sl$ is denoted by $\Zl_v$ and is called the ordinary zero set of $v$. 
\end{definition}

When no confusion can arise, we simply write $v$ instead of $(v,\Sl,\I)$.

\subsubsection{Boundedness of $|v|$}
Our definition of a $\ZT$ harmonic $1$-form is weaker than the one
used in \cite{TaubesZero}; see also \cite{Boyurecify}. In those works,
$|v|$ is assumed to extend continuously to $M$ and to vanish on the
singular set. For our purposes, the following weaker boundedness result
is sufficient. Here $\mathrm{Cap}_2$ denotes the Sobolev $2$-capacity, see \cite[Chapter 4]{evans}.

\begin{lemma}\label{lem:|v|bd}
    Let $(v,\Sl,\I)$ be a $\ZT$ harmonic $1$-form on a closed Riemannian manifold $(M,g)$. Suppose the $2$-capacity $\mathrm{Cap}_2(\Sl)=0$, then $|v|$ is bounded on $M$.
\end{lemma}
\begin{proof}
Set $u=|v|$. By Kato's inequality, $|du|\leq|\nabla v|$,
so $u\in W^{1,2}(M\setminus\Sl)$. Choose $\Lambda>0$ such that
$\mathrm{Ric}_g\geq-\Lambda g$. The Bochner formula, Kato's inequality, and the standard
regularization $u_\epsilon=(|v|^2+\epsilon^2)^{1/2}$
give $\Delta u\ge-\Lambda u$ weakly on $M\setminus\Sl$.

Since $\mathrm{Cap}_2(\Sl)=0$, the function $u$ extends to an
element of $W^{1,2}(M)$, and there exist cut-off functions
$0\leq\chi_j\leq1$, vanishing near $\Sl$, such that
$\chi_j\to1$ a.e. on $M$ and $\|d\chi_j\|_{L^2}\to0$.

For every nonnegative $\phi\in C^\infty(M)$, applying the weak
inequality to $\chi_j\phi$ and letting $j\to\infty$ gives
\[
\int_M\langle du,d\phi\rangle
\leq
\Lambda\int_Mu\phi.
\]
Thus $\Delta u\geq-\Lambda u$ weakly on $M$. Standard Moser iteration on the closed manifold $M$
then yields
$\|v\|_{L^\infty(M)}
=
\|u\|_{L^\infty(M)}
\leq
C_{M,g}\|u\|_{L^2(M)}$.
\end{proof}

The lemma applies, in particular, to the two classes of singular sets
considered in the following subsections: smooth codimension-two
submanifolds and graphic singular sets on closed manifolds. Both have
zero $2$-capacity.

\subsection{Smooth singular sets and local asymptotics}
We recall the local asymptotic expansion of a $\mathbb Z/2$ harmonic $1$-form
near a smooth singular set, following \cite{donaldsondeformation2019}.
\begin{definition}
A $2$-valued harmonic $1$-form $(v,\Sl,\I)$ is called smooth if $\Sl$ is a smooth codimension-two submanifold of $M$.
\end{definition}

Assume that $(v,\Sl,\I)$ is smooth, and write $\Sl=S_1\sqcup\cdots\sqcup S_m$ as a disjoint union of connected components. Let $U_i$ be a tubular neighborhood of $S_i$. Locally on $U_i$, choose coordinates $(t,z)$, where $t$ denotes coordinates along $S_i$ and $z=re^{i\theta}$ is a complex coordinate on the normal plane, so that $S_i=\{z=0\}$. Since the monodromy of $\I$ around the meridian $\{r=\epsilon\}$ is $-1$, the local expressions $z^{\ell+1/2}$ are naturally $\I$-valued.

By \cite{donaldsondeformation2019}, locally near $S_i$, we have the expansion
\begin{align}\label{eq:asymptotic_expansion}
v=d\Re\bigl(A_{i}(t)z^{k_i+\frac12}\bigr)+\mathcal{O}(|z|^{m_i+\frac12}).
\end{align}
The case $k_i=0$ gives the model $d\Re(z^{1/2})$, for which the norm of its covariant derivative is not $L^2$-integrable near $S_i$. Hence, if $v$ is a $\mathbb{Z}/2$ harmonic $1$-form in the sense of the preceding subsection, then $k_i\ge1$. Define $\vec k=(k_1,\cdots,k_m)$, we next introduce the concept of $\vec k$-nondegeneracy.

\begin{definition}\label{def:smooth_k_nondeg}
Let $(v,\Sl,\I)$ be a smooth $\mathbb{Z}/2$ harmonic $1$-form. We say that $v$ is $k$-nondegenerate near $S_i$ if $m_i=k$ and $A_{i}$ is nowhere vanishing on $S_i$. We say that $v$ is nondegenerate if it is $1$-nondegenerate near every component of $\Sl$. 

Let $\vec{k}=(k_i,\ldots,k_m)$ be a $m$-tuple of positive integers. $(v, \Sl, \I)$ is called $\vec{k}$-nondegenerate if for each $i$, $v$ is $k_i$-nondegenerate near $S_i$.
\end{definition}
This generalized nondegeneracy was introduced and studied in \cite{HeSalm2026Metric}. When $k_i=1$, this is the usual nondegeneracy condition introduced in \cite{donaldsondeformation2019}.

\subsection{Graphic \texorpdfstring{$\ZT$}{Z/2} harmonic $1$-forms}
\label{subsec:graphic_Z2}

In this subsection we collect the local structure of $\ZT$ harmonic $1$-forms on manifolds of dimension $n\in\{3,4\}$, follows from \cite{HeMazzeoGraphicL2,HMT2025index}. The results of this subsection are used in Section~\ref{section:desingularize_graphical}.

\subsubsection{Graphic singular sets}
\label{subsubsec:graphic_singular_sets}

\begin{definition}\label{def:graphic_singular_set}
A closed subset $\Sl$ of an oriented Riemannian manifold $(M^n,g)$, $n\in\{3,4\}$, is called a \emph{graphic submanifold} if there is a finite set $V=\{x_1,\dots,x_m\}\subset\Sl$, called the \emph{vertex set}, such that $\Sl\setminus V$ is a smooth oriented submanifold of $M$ of codimension $2$, and the following asymptotically conical condition holds at each vertex.

For each $x_i$, choose geodesic normal coordinates centered at $x_i$ and identify a neighborhood of $x_i$ with a ball $B_\epsilon(0)\subset T_{x_i}M$. There exists a closed smooth codimension $2$ submanifold $\vp_i$ of the unit sphere $S_{x_i}M\subset T_{x_i}M$ such that $\exp_{x_i}^{-1}\big((\Sl\cap B_\epsilon(x_i))\setminus\{x_i\}\big)$ is parametrized by an embedding $\Gamma_i:(0,\epsilon)\times \vp_i\longrightarrow T_{x_i}M,
$ satisfying
\[
\bigl|\tilde\nabla^k\bigl(\Gamma_i(\rho,p)-\rho p\bigr)\bigr|
=\mathcal{O}(\rho^{1+\alpha-k}),
\qquad k\ge0,\ \alpha>0,
\]
where the covariant derivatives and the norm are computed with respect to the conical metric $g_{C_i}=d\rho^2+\rho^2 g_{\vp_i}$ on $(0,\epsilon)\times\vp_i$.
\end{definition}

The embedding $\Gamma_i$ can be extended to $(0,\epsilon)\times S_{x_i}M$ and also satisfies the estimates
\[
\bigl|\tilde\nabla^k\bigl(\Gamma_i(\rho,\omega)-\rho \omega\bigr)\bigr|
=\mathcal{O}(\rho^{1+\alpha-k}),
\qquad k\ge0,\ \alpha>0,
\]
for any $\omega\in S_{x_i}M$. $\Gamma_i$ is called a \emph{straightening map}.

Each connected component of $\Sl\setminus V$ is called an \emph{edge} of $\Sl$. The cone $C(\vp_i):=\{\rho p:\rho>0,\,p\in\vp_i\}\subset T_{x_i}M$ is called the \emph{tangent cone} of $\Sl$ at $x_i$. When $n=3$, the link $\vp_i$ is an even number of points in $S^2$, and $\Sl$ near $x_i$ is a union of arcs meeting at $x_i$; when $n=4$, $\vp_i$ is a smooth closed $1$-dimensional submanifold of $S^3$.

\begin{definition}\label{def:graphic_Z2_closed_1-form}
A $\ZT$ closed $1$-form $(v,\Sl,\I)$ is called \emph{graphic} if $\Sl$ is a graphic submanifold and $\I$ is nontrivial around the edges of $\Sl$. 
\end{definition}

\subsubsection{\texorpdfstring{$\ZT$}{Z/2} eigensections}
\label{subsubsec:Z2_eigensections}

We now recall the eigensection theory on the links of graphic singularities developed in \cite{TWI,TaubesWuTopological,TaubesWuPolytopes}. Let $\vp\subset S^{n-1}$ be a closed submanifold of codimension $2$, and let $\I_\vp\to S^{n-1}\setminus\vp$ be a flat real line bundle with a Euclidean metric and monodromy $-1$ around each small meridian of $\vp$. When $n=3$ this forces $\vp$ to be a finite set of even cardinality in $S^2$; when $n=4$, $\vp$ can be any smooth link in $S^3$.

Let $\mathcal H_\vp$ be the completion of $C_c^\infty(S^{n-1}\setminus\vp;\I_\vp)$ with respect to the norm $\|f\|_{\mathcal H_\vp}^2=\int_{S^{n-1}\setminus\vp}|df|^2$ and let $\mathcal L_\vp$ be the corresponding $L^2$-completion. The flat connection on $\I_\vp$ together with the Levi-Civita connection of the round metric on $S^{n-1}$ defines the twisted Laplacian $\Delta_\vp$ acting on sections of $\I_\vp$.

\begin{definition}\label{def:Z2_eigensection}
A nonzero section $f\in\mathcal H_\vp$ is called a $\ZT$ eigensection if there is a real number $\lambda$ such that
\[
\int_{S^{n-1}}\langle df,dh\rangle=\lambda\int_{S^{n-1}}\langle f,h\rangle
\qquad\text{for every }h\in C_c^{\infty}(S^{n-1}\setminus\vp;\I_\vp).
\]
The number $\lambda$ is called a $\ZT$ eigenvalue, and the set of all such eigenvalues is denoted $\mathrm{Spec}_{\ZT}(\vp)$.
\end{definition}

By standard elliptic theory, a $\ZT$ eigensection $f$ is smooth on $S^{n-1}\setminus\vp$ and satisfies $\Delta_\vp f+\lambda f=0$ there. The eigenvalues form a discrete sequence
\(0<\lambda_1(\vp)\le \lambda_2(\vp)\le\cdots\) tending to infinity, and the
corresponding eigensections form an \(L^2\)-orthonormal basis of \({\mathcal L}_{\vp}\). Moreover, the norm of every eigensection extends H\"older continuously across $\vp$ and vanishes on $\vp$.

Let $C(\vp)=\{\rho p:\rho\ge0,\ p\in\vp\}\subset\RR^n$ be the cone over $\vp$, and let $\I(\vp)$ be the pullback of $\I_\vp$ along the projection $\RR^n\setminus C(\vp)\to S^{n-1}\setminus\vp$. If $f$ is a $\ZT$ eigensection with eigenvalue $\lambda$, define the positive homogeneous exponent
\begin{equation}\label{eq:homogeneous_exponent}
\mu_n(\lambda):=-\frac{n-2}{2}+\sqrt{\lambda+\frac{(n-2)^2}{4}}.
\end{equation}
Then $\rho^{\mu_n(\lambda)}f(\omega)$ is an $\I(\vp)$-valued homogeneous harmonic function on $\RR^n\setminus C(\vp)$. The \emph{positive indicial set} of $(\vp,\I_\vp)$ is defined to be
\begin{equation}\label{eq:indicial_set}
\mathcal D_\vp:=\{\mu_n(\lambda):\lambda\in\mathrm{Spec}_{\ZT}(\vp)\}.
\end{equation}

We also need the local form of eigensections near $\vp$. Let $P_i$ be a connected component of $\vp$. Choose a parametrization of $P_i$ by $t$, a framing of its normal bundle, and a complex normal coordinate $z=re^{i\theta}$ with $P_i=\{z=0\}$. Then each $\ZT$ eigensection has an asymptotic expansion
\begin{align}
f(t,z)=\Re\big(a_{m_i}(t)z^{m_i+\frac12}\big)+\mathcal{O}(r^{m_i+\frac{3}{2}}),\label{eq:eigensection_expansion}
\end{align}
where $m_i\ge0$ is an integer; when we need to emphasize the dependence on $f$ we write $m_i(f)$. For $n=3$, $P_i$ is a point, and $a_{m_i}$ is a complex constant; for $n=4$, $a_{m_i}(t)$ is a smooth function on the knot $P_i$, whose smoothness follows from \cite{donaldsondeformation2019}.

\begin{definition}\label{def:critical_Z2_eigensection}
A $\ZT$ eigensection $f$ and the corresponding eigenvalue are called \emph{critical} if in the expansion \eqref{eq:eigensection_expansion}, $m_i(f)\ge1$ along every connected component of $\vp$. 
\end{definition}

\subsubsection{Asymptotic expansions and nondegeneracy}
\label{subsubsec:expansions_and_nondegeneracy}
We now briefly introduce the local asymptotic and    structure of a graphic $\ZT$ harmonic $1$-form \cite{HMT2025index,HeMazzeoGraphicL2}. 

A graphic $\ZT$ harmonic 1-form $(v,\Sl,\I)$ is locally exact and coexact near the singular set.
\begin{proposition}\label{prop:local_exact_graphic}
Let $(v,\Sl,\I)$ be a graphic $\ZT$ harmonic $1$-form on a closed
Riemannian manifold $(M^n,g)$ with $n=3,4$. Then there exists a sufficiently small
neighborhood $U_{\Sl}$ of $\Sl$, and writing $U_{\Sl}^{\circ}:=U_{\Sl}\setminus\Sl$,
there exist $f\in \Gamma(U_{\Sl}^{\circ};\I)$ and
$\beta\in \Omega^{n-2}(U_{\Sl}^{\circ};\I)$ such that $v=df$ and $\star_g v=d\beta$
on $U_{\Sl}^{\circ}$.
\end{proposition}
\begin{proof}
By Lemma~\ref{lem:|v|bd}, the norm $|v|$ is bounded near $\Sl$.
Choose pairwise disjoint sufficiently small vertex balls $B^i\ni x_i$ and
pairwise disjoint thin tubular neighborhoods $U_e$ of the portions of
the edges outside the smaller vertex balls. We choose them so that
$U_{\Sl}^{\circ}$ is covered by the sets $(B^i)^\circ:=B^i\setminus\Sl$ and
$U_e^\circ:=U_e\setminus\Sl$, and every nonempty intersection $(B^i)^\circ\cap U_e^\circ$ is a punctured disk bundle over a
collar in the corresponding edge $e$ near $x_i$.

We first prove that $v$ is exact. Fix a vertex $x_i$, and let
$L_i=S_{x_i}M\setminus\vp_i$. The straightening map identifies
$(B^i)^\circ$ with $(0,\epsilon)\times L_i$. Let
$\hat{L}_i\to L_i$ be the double cover determined by $\I$, and
pull $v$ back to an ordinary closed $1$-form $\hat{v}$ on
$(0,\epsilon)\times\hat L_i$.

For every smooth $1$-cycle $\gamma\subset\hat{L}_i$, let
$\gamma_\rho=\{\rho\}\times\gamma$. Since $d\hat{v}=0$, the period
$\int_{\gamma_\rho}\hat{v}$ is independent of $\rho$. On the
other hand, the boundedness of $|v|$ and the conic form of the metric
give
$|\int_{\gamma_\rho}\hat{v}|
\leq C\mathrm{Length}(\gamma_\rho)=\mathcal O(\rho)$.
Letting $\rho\to0$ shows that every period of $\hat{v}$ vanishes.
Consequently, $\hat{v}=dF_i$ on
$(0,\epsilon)\times\hat{L}_i$.

Let $\sigma$ denote the deck involution on $\hat{L}_i\to L_i$. Replacing $F_i$ by
$(F_i-\sigma^*F_i)/2$, we may assume that $\sigma^*F_i=-F_i$.
Thus $F_i$ descends to an $\I$-valued function $f_i$ on
$(B^i)^\circ$ satisfying $v=df_i$.

Now fix an edge tube $U_e$, and let
$\pi:\hat{U}_e^\circ\to U_e^\circ$ be the double cover determined
by $\I$, with deck involution $\sigma$. After retracting radially onto
the normal circle bundle, $\sigma$ acts by a half-rotation on each
circle fiber and is therefore homotopic to the identity through
fiberwise rotations. Hence $\sigma^*=\mathrm{id}$ on
$H^*(\hat{U}_e^\circ;\RR)$. Since
$H^*(U_e^\circ;\I)$ identifies with the $(-1)$-eigenspace of
$\sigma^*$, it follows that $H^k(U_e^\circ;\I)=0$ for every $k$.
In particular, $v=df_e$ for some
$f_e\in\Gamma(U_e^\circ;\I)$.

The same argument gives
$H^\ast((B_i\cap U_e)\setminus\Sl;\I)=0$.
Since $d(f_i-f_e)=0$ on the overlap, $f_i-f_e$ is a flat
$\I$-valued section. Since $H^0((B_i\cap U_e)\setminus\Sl;\I)=0$, it must vanish. Thus the local primitives glue
to a section $f\in\Gamma(U_{\Sl}^\circ;\I)$ satisfying $v=df$.

It remains to prove that $\star_gv$ is exact. Set
$\omega=\star_gv$, so $d\omega=0$ and
$\omega\in\Omega^{n-1}(U_{\Sl}^{\circ};\I)$.

Let $A=\cup_{x_i\in V}(B^i)^\circ$ and $B=\cup_e U_e^\circ$. The Mayer--Vietoris sequence contains the following exact sequence:
\[
H^{n-2}(A\cap B;\I)
\longrightarrow
H^{n-1}(U_{\Sl}^{\circ};\I)
\longrightarrow
H^{n-1}(A;\I)\oplus H^{n-1}(B;\I).
\]
We next show that both outer groups vanish.

For the left-hand group, it suffices to note that
$H^{n-2}((B^i\cap U_e)\setminus\Sl;\I)=0$, as proved above.
For the right-hand group, first observe that
$H^{n-1}(B;\I)\cong\oplus_e H^{n-1}(U_e^\circ;\I)=0$.
It remains to show that
$H^{n-1}(A;\I)\cong\oplus_{x_i\in V}
H^{n-1}((B^i)^\circ;\I)=0$.

Let $N(\vp_i)$ be a sufficiently small tubular neighborhood of
$\vp_i$ in $S_{x_i}M$, and set
$K_i=S_{x_i}M\setminus(\mathrm{Int}N(\vp_i))$. Then the punctured
vertex neighborhood $(B^i)^\circ$ deformation retracts onto $K_i$.
Let $\hat K_i\to K_i$ be the double cover determined by $\I$.
The manifold $\hat K_i$ is a compact orientable $(n-1)$-manifold
with boundary, and hence $H^{n-1}(\hat K_i;\RR)=0$.

Since $H^{n-1}(K_i;\I)$ identifies with the $(-1)$-eigenspace of the
deck involution on $H^{n-1}(\hat K_i;\RR)$, it follows that
$H^{n-1}(K_i;\I)=0$. Consequently,
$H^{n-1}((B^i)^\circ;\I)=0$, and therefore
$H^{n-1}(A;\I)=0$.

It follows that
$H^{n-1}(U_{\Sl}^{\circ};\I)=0$.
Since $\omega=\star_gv$ is a closed $\I$-valued $(n-1)$-form, its
cohomology class vanishes. Thus there exists
$\beta\in\Omega^{n-2}(U_{\Sl}^{\circ};\I)$ such that
$\star_gv=d\beta$.
\end{proof}

We call the function $u$ of Lemma~\ref{prop:local_exact_graphic} the \emph{local potential} of $v$ near $q$. The asymptotic expansions of a graphic $\ZT$ harmonic 1-form near $\Sl$, as proved in \cite{HMT2025index,HeMazzeoGraphicL2}, are described below.

Away from the vertices, each edge is a smooth codimension $2$ submanifold, and the leading local asymptotics coincide with the submanifold case. In coordinates $(t,r,\theta)$, where $t$ is a coordinate along the edge, and $z=re^{i\theta}$ is the normal coordinate, then similar to \eqref{eq:asymptotic_expansion}, we have
\begin{align}\label{eq:edge_rough_expansion}
v=d\Re\bigl(A_mz^{m+\frac12}\bigr)+\mathcal{O}(|z|^{m+\frac12}).    
\end{align}

As the edge approaches a vertex, the coefficients $A$ themselves admit asymptotic expansions in the vertex defining function $\rho$.

Near a vertex $x_i$, use geodesic normal coordinates and pull back by
the straightening map $\Gamma_i$. We then use polar coordinates
$(\rho,\omega)\in(0,\epsilon)_\rho\times S_i^{n-1}$. Let
$\I_{\vp_i}\to S_i^{n-1}\setminus\vp_i$ be the flat real line bundle
induced by $\I$, and set
$\mathcal D_i:=\mathcal D_{\vp_i}$; see \eqref{eq:indicial_set}.
Then $\Gamma_i^\ast v=dU$, and there exist $\mu_i\in\mathcal D_i$, a critical eigensection $f_i$,
and $\epsilon_i>0$ such that
\begin{align}\label{eq:vertex_rough_expansion}
U(\rho,\omega)
=
\rho^{\mu_i}f_i(\omega)+E(\rho,\omega),
\end{align}
where
\begin{align}\label{eq:relative_estimate_error_vertex}
\sup_{\omega\in K}
\left|
(\rho\partial_\rho)^a\nabla_\omega^\beta E(\rho,\omega)
\right|
\le
C_{K,a,\beta}\rho^{\mu_i+\epsilon_i},
\qquad a+|\beta|\le2,
\end{align}
for every $K\Subset S_i^{n-1}\setminus\vp_i$. The homogeneous
$\ZT$ harmonic $1$-form $\bigl(d(\rho^{\mu_i}f_i),C(\vp_i),\I(\vp_i)\bigr)$
on $\RR^n$ is called the \emph{tangent cone} of $(v,\Sl,\I)$ at
the vertex $x_i$.

\begin{definition}\label{def:nondegenerate_tangent_cone}
The tangent cone $(d(\rho^{\mu_i}f_i),C(\vp_i),\I(\vp_i))$ at $x_i$ is called nondegenerate if, for every connected component $P_j\subset\vp_i$, the leading exponent in the expansion \eqref{eq:eigensection_expansion} of $f_i$ near $P_j$ satisfies $m_{P_j}(f_i)=1$ with nonvanishing coefficient. Equivalently,
\[
f_i(t,z_j)=\Re\bigl(a_{i,P_j}(t)z_j^{3/2}\bigr)+\mathcal{O}(|z_j|^{5/2}),
\qquad t\in P_j,\quad a_{i,P_j}(t)\ne 0.
\]
It is called \textbf{strongly nondegenerate} if, on $S_i^{n-1}$, one has $f_i(\omega)=0$ and $df_i(\omega)=0$ if and only if $\omega\in\vp_i$.
\end{definition}
Strong nondegeneracy rules out ordinary zeros of the tangent cone away from $C(\vp_i)$.

\begin{definition}\label{def:strong_nondeg_graphic}
A graphic $\ZT$ harmonic $1$-form $(v,\Sl,\I)$ is called \textbf{strongly nondegenerate} if it is nondegenerate along every edge of $\Sl$ and, at every vertex $x_i$, its tangent cone $d(\rho^{\mu_i}f_i)$ is strongly nondegenerate.
\end{definition}

Suppose that $(v,\Sl,\I)$ is strongly nondegenerate. For later use, we record a refinement of the vertex expansion near the
incident edges in the strongly nondegenerate case. Fix a vertex $x_i$
and a connected component $P\subset \vp_i$. Choose local coordinates
$(t,z)$ on $S_i^{n-1}$ near $P$, where $z$ is a complex normal
coordinate and $P=\{z=0\}$. Using the notation in
Definition~\ref{def:nondegenerate_tangent_cone}, write
\[
f_i(t,z)
=
\Re\bigl(a_{i,P}(t)z^{3/2}\bigr)
+
\mathcal O(|z|^{5/2}).
\]
Let $\zeta=\rho z$ be the corresponding physical transverse coordinate
in the straightened conic coordinates. The joint vertex--edge expansion
takes the form
\begin{equation}\label{eq:vertex_edge_expansion}
U
=
\Re\left(
\bigl(a_{i,P}(t)\rho^{\mu_i-\frac32}
+b_{i,P}(\rho,t)\bigr)\zeta^{3/2}
\right)
+
E_{i,P}(\rho,t,\zeta),
\end{equation}
where, for every $0<\epsilon<\epsilon_i$,
\begin{equation}\label{eq:vertex_edge_error_estimate}
\bigl|\widetilde\nabla^k b_{i,P}(\rho,t)\bigr|
\le
C_{k,\epsilon}
\rho^{\mu_i-\frac32+\epsilon_i-\epsilon-k},
\qquad k=0,1,2.
\end{equation}
Moreover,
$
|E_{i,P}|
\le
C\rho^{\mu_i-\frac52}|\zeta|^{5/2}$ and $
|dE_{i,P}|
\le
C\left(
\rho^{\mu_i-\frac52}|\zeta|^{3/2}
+
\rho^{\mu_i-\frac72}|\zeta|^{5/2}
\right)$.

Here $\widetilde\nabla$ denotes the covariant derivative associated with
the conic metric $d\rho^2+\rho^2g_P$ on $C(P)$, while
$dE_{i,P}$ is the full differential in the variables
$(\rho,t,\zeta)$, with norm computed in the straightened conic metric.

\section{Transitivity and gluing}\label{section:gluing}
In this section we prove the main gluing result of the paper. Starting from compatible local models for $2$-valued closed $1$-forms, we construct a global closed form and then choose a metric for which it becomes $\ZT$ harmonic. Following the strategy of \cite{intrinsic-Calabi,intrinsic-Volkov}, the proof proceeds by constructing a closed $2$-valued $(n-1)$-form as a candidate for the Hodge dual and then realizing it by a suitable choice of metric. Along the way we introduce the intrinsic and dynamical conditions needed for the construction, and prove that every $\ZT$ harmonic $1$-form on a closed manifold satisfies the required transitivity property.

\subsection{Positive loops and intrinsic harmonicity}
Let $(v,\Sl,\I)$ be a $2$-valued closed $1$-form on $(M,g)$.
A smooth loop $\gamma$ in $M$ is called \emph{$v$-positive} if $\langle v,\dot{\gamma}\rangle$ is a nonvanishing section of $\I\big|_{\gamma}$. Along such a loop, $\I\big|_\gamma$ must be a trivial bundle.

\begin{definition}\label{def:transitive}
The $\ZT$ 1-form $(v, \Sl, \I)$ is said to be \emph{transitive} if for every
$x \in M $ with $|v(x)|\ne 0$, there exists a smooth, simple, $v$-positive loop
$\gamma: S^1 \to M$ passing through $x$.
\end{definition}

As $dv=0$, $\ker v$ induces a codimensional $1$ singular foliation $\ker v$.
A leaf of $\ker v\big|_{M\setminus(\Sl\cup\Zl_v)}$ is called a \emph{regular leaf} of $\ker v$.
Note that a loop $\gamma$ is $v$-positive if and only if it is transverse to the regular leaves of $\ker v$.

\begin{lemma}\label{lem:modify_positive_loop_in_one_leaf}
    Let $\gamma$ be a simple $v$-positive loop in $M$ passing through a point $x\in M$, and let $\ell_x$ be the leaf of $\ker v$ containing $x$. Then, for every $y\in \ell_x\setminus\gamma$, there exists a simple $v$-positive loop $\gamma'$, obtained from $\gamma$ by a modification, that passes through $y$. 
\end{lemma}
\begin{proof}
    Let $C$ be the connected component of $\ell_x\setminus\gamma$ that contains $y$. Let $z$ be a boundary point of $C$; then $z\in\gamma$.

   Choose a smooth arc $\eta:[0,1]\to\ell_x$ such that $\eta(0)=z$ and $\eta(1)=y$, and whose interior is contained in $C$. In particular, the interior of $\eta$ is disjoint from $\gamma$.

    Cover $\eta([0,1])$ by finitely many flat charts for the foliation $\ker v$, and let $U$ be a neighborhood of $\eta$ contained in their union. We may choose $U$ to be homeomorphic to a product $V\times(-\epsilon,\epsilon)$, where $V=U\cap \ell_x$ and $\epsilon>0$ is sufficiently small. Each slice $V\times{t}$, with $t\in(-\epsilon,\epsilon)$, is contained in a regular leaf of $\ker v$.

    Over $U$, the line bundle $\I\big|_U$ is trivial, and we can write $v\big|_U=dt\otimes s$, for $t$ the coordinate on $(-\epsilon,\epsilon)$ and $s$ a unit length section of $\I\big|_U$.

    Identify $z$ with $(z,0)\in V\times\{0\}$. Let $(-\delta,\delta)\subset S^1$ be an interval such that $\gamma\big|_{(-\delta,\delta)}$ is the connected component of $\gamma\cap U$ containing $z$.
    
    By definition we can assume $\langle dt,\dot{\gamma}(\theta)\rangle>0$ for $\theta\in(-\delta,\delta)$. Define a loop $\gamma'$ by 
    \[\gamma'(\theta)=\begin{cases}\gamma(\theta),&\theta\notin (-\delta,\delta)\\
        (\eta(1+\frac{\theta}{\delta}),\frac{\theta\epsilon}{\delta}),&-\delta<\theta\le 0\\
        (\eta(1-\frac{\theta}{\delta}),\frac{\theta\epsilon}{\delta}),&0<\theta<\delta
    \end{cases}\]
    After a small smoothing near $y$, we obtain a smooth loop, still denoted by $\gamma'$, such that $\langle dt,\dot{\gamma}'(\theta)\rangle>0$ for $-\delta<\theta<\delta$. Hence $\gamma'$ is the desired loop.
\end{proof}

\begin{definition}\label{def:locally_intrinsically_harmonic}
Let $(M,g)$ be an oriented, connected Riemannian manifold. Let $(v, \Sl, \I)$ be a $2$-valued closed $1$-form on $M$ with $\Zl_v$ the ordinary zero set. We say $v$ is \emph{locally intrinsically harmonic} if there exists a neighborhood $U$ of the closed subset $\Sl \cup \Zl_v$, and a Riemannian metric $g_U$ on $U$, such that $v|_U$ is harmonic with respect to $g_U$. If, in addition, $\star_{g_U} v|_U$ is exact for such $U$ and $g_U$, then $v$ is called \emph{locally $\star$-exact}.
\end{definition}

\begin{definition}\label{def:intrinsic_harmonic}
    A $2$-valued closed $1$-form $(v,\Sl,\I)$ on $(M,g)$ is called \emph{intrinsically harmonic} if there exists a metric $g'$ on $M$ such that $(v,\Sl,\I)$ is harmonic with respect to $g'$.
\end{definition}

\begin{remark}
A single-valued closed $1$-form $(v, \varnothing, \underline{\mathbb{R}})$ with Morse zeros of index between $1$ and $n-1$ is always \emph{locally} intrinsically harmonic, but not necessarily intrinsically harmonic. To see this, let $p$ be a zero of $v$. By the Morse lemma, there exists a local coordinate system $(x_1, \dots, x_n)$ centered at $p$ such that
\[
v = d\bigl( \lambda_1 x_1^2 + \cdots + \lambda_k x_k^2 - \lambda_{k+1} x_{k+1}^2 - \cdots - \lambda_n x_n^2 \bigr),
\]
for some $1 \le k \le n-1$, where the coefficients $\lambda_i > 0$ are chosen so that $\lambda_{1}+\cdots+\lambda_{k}=\lambda_{k+1}+\cdots+\lambda_{n}$. Then in these coordinates, $v$ is harmonic with respect to the Euclidean metric.
\end{remark}

\subsection{Transitivity of $\ZT$ harmonic 1-forms}
In this subsection we prove that every \(\ZT\) harmonic 1-form on a closed manifold satisfies the transitivity condition needed in the gluing theorem.

\begin{theorem}\label{thm:harmonic_implies_transitivity}
Let $(v,\Sl,\I)$ be a $\ZT$ harmonic $1$-form on a closed Riemannian manifold $(M,g)$. Then $v$ is transitive on $M$.
\end{theorem}

We begin with some preparatory notation and lemmas for the proof of Theorem~\ref{thm:harmonic_implies_transitivity}.

First, note that the set of unit-length vectors of $\I\to M\setminus\Sl$, denoted by $\hat{M}_c$, is a $2$-fold cover of $M\setminus\Sl$. We pullback the metric $g$ to this open manifold and denote it by $\hat{g}$. $(\hat{M}_c,\hat{g})$ is incomplete. If $M$ is closed, this incomplete manifold has finite volume. $v$ lifts to a single-valued harmonic $1$-form $\hat{v}$ on $(\hat{M}_c,\hat{g})$.

There is a natural topology on the space $\hat{M}_c\sqcup \Sl=:\hat{M}$, under which the projection $\pi:\hat{M}_c\sqcup \Sl\to M$ is continuous. We can also define a natural metric $\mathrm{dist}$ and measure $\mathrm{vol}_{\hat{g}}$ on $\hat{M}$ in a clear way. Thus $\hat{M}$ is a metric measure space. Locally, the closed set $\Zl_v$ is the zero set of an ordinary harmonic $1$-form, hence it has Hausdorff codimension at least $2$ (See \cite[Lemma 3]{intrinsic-Volkov}). So as for $\Sl$ by definition. Therefore, let $A:=\Sl\cup\Zl_v$ and $\hat{A}:=\Sl\cup\pi^{-1}(\Zl_v)\subset \hat{M}$, we have $\dim_{\mathcal{H}}A,~\dim_{\mathcal{H}}\hat{A}\le n-2$.

Consider the dual vector field $\hat{V}$ of $\hat{v}$ on $\hat{M}_c$, with respect to $\hat{g}$. By definition of $\ZT$ harmonic $1$-forms, $\hat{V}$ satisfies $|\nabla \hat{V}|\in L^2(\hat{M}_c)$ and $\mathrm{div}_{\hat{g}}\hat{V}=0$ on $\hat{M}_c$. 

\begin{lemma}\label{lem:hatV_Lq}
    There exists $q>2$ such that $|\hat{V}|$ extends to an $L^q$-integrable function on $\hat{M}$.
\end{lemma}
\begin{proof}
    By definition, $|\hat{V}|=|\hat{v}|$. Note that $|v|$ is a well-defined function on $M\setminus\Sl$. It suffices to show that $|v|$ extends to an $L^q$-integrable function on $M$.

    For any $\sigma<2$, we have $|v|\in W^{1,\sigma}(M\setminus\Sl)$.
    Since $\mathrm{codim}_{\mathcal{H}}A\ge 2$, $\mathrm{Cap}_{\sigma}(A)=0$ (see \cite[Theorem 4.16]{evans}).  
    Therefore, $|v|\in W^{1,\sigma}(M)$.

    Choose $\sigma$ such that $\frac{2n}{n+2}<\sigma<2$. The Sobolev embedding then gives  $|v|\in L^q(M)$, with $q=\frac{n\sigma}{n-\sigma}>2$. This proves the lemma. 
\end{proof}

We define a measure preserving \emph{local flow} (\cite{Avoidance}) $\Phi_t$ on $\hat{M}\setminus \hat{A}$ with \emph{ceiling functions $-\infty\le S_-<0<S_+\le+\infty$}, generated by the vector field $\hat{V}$, satisfying:
\begin{enumerate}
    \item For $x\in\hat{M}\setminus \hat{A}$, $\Phi_t(x)\in \hat{M}\setminus \hat{A}$ is defined for $t\in (S_-(x),S_+(x))$;
    \item $S_\pm\big(\Phi_t(x)\big)=S_\pm(x)-t$;
    \item $\Phi_0(x)=x$, and $\Phi_{t_1}\circ\Phi_{t_2}(x)=\Phi_{t_1+t_2}(x)$, whenever $t_1,t_1+t_2\in (S_-(x),S_+(x))$;
    \item $\frac{d}{dt}\big|_{t=0}\Phi_t(x)=\hat{V}(x)$;
    \item if $x\in\hat{M}\setminus \hat{A}$ and $S_+(x)<\infty$, then $\lim_{t\to S_+(x)}\mathrm{dist}(\Phi_t(x),\hat{A})=0$, similar for $S_-(x)$.
\end{enumerate}

We now prove an analogue of \cite[Theorem 1]{Avoidance}. In the following proof, we replace the Minkowski codimension condition there with our Hausdorff codimension condition.
\begin{lemma}\label{lem:measure_zero}
    Both ceiling functions $S_+$ and $S_-$ are infinite almost everywhere on $\hat{M}$. That is, there is a measure zero set $K\subset \hat{M}\setminus \hat{A}$, such that $\Phi_t:\hat{M}\setminus(K\cup \hat{A})\to \hat{M}\setminus(K\cup \hat{A})$ is defined for any $t\in\RR$.
\end{lemma}
\begin{proof}
    We prove the statement for $S_+$. The proof for $S_-$ follows by reversing time. 
    
    By Lemma~\ref{lem:hatV_Lq}, choose $q>2$ such that $\hat{V}\in L^q(\hat{M})$, and set $q'=\frac{q}{q-1}$. Then $1<q'<2$. 
    
    Since $\mathrm{codim}_{\mathcal{H}}(\hat{A})\ge 2$, $\mathrm{Cap}_{q'}(\hat{A})=0$. For $r>0$ and $T>0$, define 
    \[ E_{r,T} = \bigl\{ x\in\hat{M}\setminus \hat{A}: \mathrm{dist}(x,\hat{A})\ge r,\; S_+(x)\le T \bigr\}. \] 
    
    It is enough to prove $|E_{r,T}|_{\hat{g}}=0$ for all $r,T>0$, since 
    $\{x\in\hat{M}\setminus \hat{A}:S_+(x)<+\infty\} = \cup_{m,N\ge1}E_{1/m,N}.$
    Let $U_r=\{y\in\hat{M}:\mathrm{dist}(y,\hat{A})<r/2\}$.  
    Since $\mathrm{Cap}_{q'}(\hat{A})=0$, for every $\varepsilon>0$ we can choose a cutoff function 
    $\varphi\in C_c^0(U_r)\cap C^\infty(\hat{M}_c)\cap W^{1,q'}(\hat{M}_c)$  
    such that 
     $0\le\varphi\le1$, $\varphi\equiv1$ in a neighborhood of $\hat{A}$, 
    and $\|\nabla\varphi\|_{L^{q'}(\hat{M}_c)}<\varepsilon$. 
     
    In particular, $\varphi(x)=0$ for every $x\in E_{r,T}$. For each $x\in E_{r,T}$, property (v) of the local flow implies that $\Phi_t(x)$ approaches $\hat{A}$ as $t\to S_+(x)$. Since $\varphi\equiv1$ near $\hat{A}$, there exists a time $\tau(x)\in(0,S_+(x))$ 
    such that $\varphi(\Phi_{\tau(x)}(x))=1$.  
   
    Therefore 
    $1 = |\varphi(\Phi_{\tau(x)}(x))-\varphi(x)| \le \int_0^{\tau(x)} |\nabla\varphi|\,|\hat{V}|(\Phi_s(x))\,ds$.  
    Integrating over $E_{r,T}$ gives 
    \begin{align*}
        |E_{r,T}|_{\hat{g}} &\le \int_{E_{r,T}}\int_0^{\tau(x)} |\nabla\varphi|\,|\hat{V}|(\Phi_s(x))\,ds\,d\mathrm{vol}_{\hat{g}}(x) \\ 
    &= \int_0^T \int_{E_{r,T}\cap\{\tau>s\}} |\nabla\varphi|\,|\hat{V}|(\Phi_s(x))\,d\mathrm{vol}_{\hat{g}}(x)\,ds . 
    \end{align*} 
    For each fixed $s$, the map $\Phi_s$ is measure preserving on its domain. Hence 
    \[ |E_{r,T}|_{\hat{g}} \le T\int_{\hat{M}_c} |\nabla\varphi|\,|\hat{V}|\,d\mathrm{vol}_{\hat{g}}. \] 
    By H\"older's inequality, 
     $\int_{\hat{M}} |\nabla\varphi|\,|\hat{V}|\,d\mathrm{vol}_{\hat{g}} \le \|\hat{V}\|_{L^{q}(\hat{M}_c)} \|\nabla\varphi\|_{L^{q'}(\hat{M}_c)} \le \|\hat{V}\|_{L^{q}(\hat{M})}\,\varepsilon$.  
    Letting $\varepsilon\to0$ yields $|E_{r,T}|_{\hat{g}}=0$. Thus $S_+=+\infty$ almost everywhere. 
\end{proof}

Recall the following classical result:
\begin{lemma}[Poincar\'e recurrence theorem]\label{lem:Poincare_recurrence}
    Let $(\Omega,\Sigma,\mu)$ be a probability space with finite measure, and $\{\Phi_t\}_{t\ge 0}$ be a measure preserving flow on it. Assume $U\in\Sigma$ has finite measure, then for almost all $x\in U$, there are integers $0<n_1<n_2<\dots$ such that $\Phi_{n_i}(x)\in U$.
\end{lemma}
Now we can apply Lemma~\ref{lem:Poincare_recurrence} to $\Phi_t$ and $\hat{M}\setminus(K\cup \hat{A})$, to prove Theorem \ref{thm:harmonic_implies_transitivity}.

\begin{proof}[Proof of Theorem \ref{thm:harmonic_implies_transitivity}]
Let $x_0\in M\setminus(\Sl\cup\Zl_v)$ and let $x_0^+\in \pi^{-1}(x_0)\subset \hat{M}\setminus \hat{A}$.  Choose a flat chart $V_\epsilon$ of $\ker\hat{v}$ centered at $x^+_0$, and a diffeomorphism \[\psi_\epsilon:V_\epsilon\to\{(x_1,\dots,x_n)\in\RR^n:|x_i|\le \epsilon\},\]
    such that $\psi_\epsilon(x^+_0)=0$ and $\pm\psi_\epsilon^\ast dx_n=\hat{v}\big|_{V_\epsilon}$. Since $K$ is measure zero in $\hat{M}$, $V_\epsilon\setminus K$ has positive measure.

    Then by Lemma \ref{lem:Poincare_recurrence}, there exists $y\in V_\epsilon\setminus K$, such that $\psi_\epsilon(y)$ has positive $n$-th coordinate, and for some large $T\gg0$, we have $\Phi_T(y)\in V_\epsilon$ with $\psi_\epsilon\circ\Phi_T(y)$ having negative $n$-th coordinate. We can connect $\Phi_T(y)$, $x^+_0$ and $y$ by a $\hat{v}$-positive arc, and extend the arc $\Phi_t(y)\,(0\le t\le T)$ to form a $\hat{v}$-positive simple loop $\hat{\gamma}$ passing through $x^+_0$. Then $\pi\circ \hat{\gamma}=:\gamma$ is a $v$-positive loop passing through $x_0$. By a small perturbation, it can be made simple. 
\end{proof}

\subsection{The gluing theorem}
We now formulate the gluing result that will be used in the surgery constructions below.
\begin{definition}\label{def:Gluing_compatible_pair}
Let $M_1$ and $M_2$ be oriented Riemannian manifolds, each with a boundary component $W_i\subset \partial M_i$ diffeomorphic to a closed, oriented and connected manifold $X$.  Let $(v_i, \Sl_i, \I_i)$ be $\ZT$ closed $1$-forms on $M_i~(i=1,2)$.

We say this pair of $\ZT$ closed $1$-forms are \emph{gluing compatible along $X$} if there exist collar neighborhoods $N_i$ of $W_i$ and embeddings $\varphi_i : N_i \to X \times (-1,1)$ such that:
\begin{enumerate}
      \item $\varphi_1(N_1)=X\times[\frac{1}{2},1)$ and $\varphi_2(N_2)=X\times(-1,-\frac{1}{2}]$;
      \item there exists a $2$-valued closed $1$-form $(v_X,\Sl_X,\I_X)$ on $X\times(-1,1)$, such that \[\varphi_i^\ast(v_X,\Sl_X, \I_X)=(v_i,\Sl_i,\I_i)\big|_{N_i}\text{ for }i=1,2.\]
\end{enumerate}
\end{definition}

Let $M$ be the oriented manifold obtained by gluing $M_1$, $X\times(-1,1)$ and $M_2$ along $N_1$ and $N_2$, via $\varphi_1$ and $\varphi_2$, respectively. Then there is a uniquely determined $2$-valued closed $1$-form $(v_M, \Sl_M, \I_M)$ on $M$ such that  
$(v_M, \Sl_M, \I_M)|_{M_i} = (v_i, \Sl_i, \I_i)$ for $i = 1, 2$.

Now let $M_i~(1\le i\le k)$ be a finite collection of oriented manifolds, and let $(v_i,\Sl_i,\I_i)$ be a $\ZT$ closed $1$-form on $M_i$. Suppose that for every component $W$ of $\partial M_i$, there exists a $M_j$ such that $(v_i,\Sl_i,\I_i)$ and $(v_j,\Sl_j,\I_j)$ are gluing compatible along $W$. 

Assume further that this collection of manifolds  can be glued along all such pairs to yield a \emph{closed} manifold $M$ equipped with a globally defined $\ZT$ $1$-form $(v_M,\Sl_M,\I_M)$ on it. For brevity, we continue to denote the gluing region by $X\times(-1,1)$, where $X\times\{-1,1\}$ corresponds to the disjoint union of $\partial M_i~(1\le i\le k)$. Let $N_i\subset M_i$ be the collar neighborhood of $\partial M_i$. 

Before stating the gluing theorem, we recall that we could pair the $\I$-factors in wedge products. This gives the natural pairing
\[
\wedge:\Gamma(\wedge^kT^*M\otimes\I)\otimes
\Gamma(\wedge^jT^*M\otimes\I)
\longrightarrow
\Gamma(\wedge^{k+j}T^*M|_{M\setminus\Sl}).
\]
In particular, for an \(\I_M\)-valued $1$-form \(v_M\) and an \(\I_M\)-valued \((n-1)\)-form \(\omega\), the expression \(v_M\wedge\omega>0\) means positivity as an ordinary top-degree form with respect to the given orientation of \(M\).

Now we can prove the following gluing theorem:
\begin{theorem}\label{thm:gluing}
Let $M_i$, $(v_i,\Sl_i,\I_i)$ $(1\le i\le k)$ and $X$ be as above. Assume there exists an integer $m$ with $1 \le m \le k$ such that, for each $1 \le i \le m$, there is a Riemannian metric $g_i$ on $M_i$ for which $v_i$ is harmonic, and for each $m < j \le k$, we have $\Sl_j = \varnothing$ and $v_j$ is nowhere vanishing.

Suppose further that this collection of $\ZT$ $1$-forms $(v_i, \Sl_i, \I_i)~(1\le i\le k)$ satisfy:
\begin{enumerate}
    \item $(\Sl_i \cup \Zl_{v_i}) \cap N_i = \varnothing$ and $\Sl_X=\Zl_{v_X}=\varnothing$;
    \item For $1\le i\le m$, $v_i$ is $\star$-exact on $N_i$, i.e., $\star_{g_i}v_i\big|_{N_i}=d\alpha_i$ for some $\I_i$-valued $(n-2)$-form $\alpha_i$ on $N_i$;
    \item For any point $x\in \big(X\times(-1,1)\big)\cup\big( \cup_{m<j\le k}M_j\big)$, there exists a $v_M$-positive loop $\gamma$ passing through $x$.
\end{enumerate}
Then there exists a metric $g_M$ on $M$, such that $v_M$ is harmonic with respect to $g_M$.
\end{theorem}
\begin{proof}
    We first construct an $\I_M$-valued closed $(n-1)$-form $\omega$ on $M$, such that $v_M\wedge\omega>0$ on the open set $M\setminus\big(\Sl_M\cup\Zl_{v_M}\big)$. $\omega$ serves as a candidate for the Hodge star dual of $v_M$. We then use $\omega$ to construct a Riemannian metric $g_M$ on $M$ such that $\star_{g_M} v_M = \omega$. 

    Since $\star_{g_i}v_i~(1\le i\le m)$ are exact on $N_i~(1\le i\le m)$, we can glue the local primitives using a cut-off function $\chi$ on $(\frac{1}{2},1)$, with $\chi(t)\equiv 1$ for $\frac{1}{2}<\frac{5}{8}<t<1$. The resulting $\I_M$-valued closed $(n-1)$-form $\omega_1$ is:
    \[\omega_1=\begin{cases}
        \star_{g_i}v_i,\quad x\in M_i\setminus N_i;\\
        d\big(\chi\alpha_i\big),\quad x\in N_i,
    \end{cases}\]
    where $1 \le i \le m$. Here, each $N_i$ is a collar neighborhood of $\partial M_i$ diffeomorphic to $\partial M_i \times [\tfrac{1}{2}, 1)$, and for $y=(y',t) \in \partial M_i \times [\frac{1}{2}, 1)$, we set $\chi(y):=\chi(t)$. 

    Next, by the transitivity assumption (iii) on $v_M$, for any point $x\in \big(X\times(-1,1)\big)\cup\big( \cup_{m<j\le k}M_j\big)$, there exists a $v_M$-positive loop $\gamma$ passing through $x$. Choose a sufficiently small tubular neighborhood $U_x$ of $\gamma$, such that $\I_M\big|_{U_x}$ is trivial. We can write $v_M\big|_{U_x}=w\otimes s$ on $U_x$, where $w$ is a single-valued closed $1$-form on $U_x$ and $s$ is a unit-length section of $\I_M\big|_{U_x}$. We then have $\langle v_M,\dot{\gamma}\rangle=\langle w,\dot{\gamma}\rangle s$ on $\gamma$, and $\langle w,\dot{\gamma}\rangle\ne 0$ everywhere on $\gamma$. Without loss of generality, we may assume $\langle w,\dot{\gamma}\rangle> 0$.

    For sufficiently small $U_x$, there exists a diffeomorphism $h : U_x \to S^1 \times D^{n-1}$, where $D^{n-1}$ is the $(n{-}1)$-dimensional unit disk, such that $h^{-1}(\{\theta\} \times D^{n-1})$ lies in a leaf of $\ker v_M$ for every $\theta \in S^1$. Let $\mathrm{vol}_{D^{n-1}}$ be the Euclidean volume form on $D^{n-1}$, 
    and let $\rho$ be a smooth cutoff function on $\RR^{n-1}$ with compact support in $D^{n-1}$. Then $\eta_x = h^* (\rho \cdot \mathrm{vol}_{D^{n-1}})$ is a smooth, closed $(n-1)$-form compactly supported in $U_x$, satisfying $w \wedge \eta_x \ge 0$ with $\rho \ge 0$ and $\rho > 0$ near $0 \in D^{n-1}$.
    
    Thus $\omega_x = \eta_x \otimes s$ is an $\I_M$-valued closed $(n{-}1)$-form on $U_x$, which extends by zero to $M$.  We then have $v_M \wedge \omega_x \geq 0$ on $M$, and $v_M \wedge \omega_x > 0$ in a neighborhood $V_x\subset U_x$ of $x$.

     We can now choose finitely many points $x_1,\dots, x_\ell$ in $\big(X\times(-1,1)\big)\cup\big( \cup_{m<j\le k}M_j\big)$, such that the open sets $V_{x_i}$ $(1 \le i \le \ell)$ form an open cover of this region.  
    Consequently, $v_M\wedge \sum_{1\le i\le \ell}\omega_{x_i}>0$ on $\big(X\times(-1,1)\big)\cup\big( \cup_{m<j\le k}M_j\big)$. Let $V=\cup_{1\le i\le \ell}V_{x_i}$. 
    
    Set $\omega = \omega_1 + \Lambda\omega_2$, where $\omega_2=\sum_{i=1}^\ell \omega_{x_i}$. Then $\omega$ is an $\I_M$-valued closed $(n-1)$-form on $M \setminus \Sl_M$. By definition, outside the region $\big(X\times(-1,1)\big)\cup\big( \cup_{m<j\le k}M_j\big)$, we have $\omega_1 = \star_{g_i} v_i$.
    It follows that $v_M \wedge \omega_1 > 0$ on the complement of $\Sl_M \cup \Zl_{v_M} \cup \big(X\times(-1,1)\big)\cup\big( \cup_{m<j\le k}M_j\big)$.
    On the region $\big(X\times(-1,1)\big)\cup\big( \cup_{m<j\le k}M_j\big)$, which is contained in $V$, we have $v_M \wedge \omega_2 > 0$ on $V$.
    Thus, choosing $\Lambda > 0$ sufficiently large, we then have $v_M \wedge \omega > 0$ on $M\setminus (\Sl_M\cup \Zl_{v_M})$.
    
    Using the $(n{-}1)$-form $\omega$, we can construct a Riemannian metric $g_M$ on $M$ such that $\star_{g_M} v_M = \omega$. For details of this construction, see \cite{intrinsic-Calabi,intrinsic-Volkov}.
\end{proof}

We call $M_i~(1\le i\le m)$ in Theorem~\ref{thm:gluing} the \emph{harmonic pieces} and $M_j~(m<j\le k)$ the \emph{transition pieces}. 

The proof of Theorem \ref{thm:gluing} follows closely the argument in \cite{intrinsic-Calabi,intrinsic-Volkov}, and also applies to the following intrinsic characterization of $2$-valued harmonic $1$-forms.

\begin{theorem}\label{thm:characterize}
    Let $M$ be a closed oriented manifold, and $(v,\Sl,\I)$ be a $2$-valued closed $1$-form on $M$. Suppose $(v,\Sl,\I)$ is locally intrinsically harmonic, locally $\star$-exact and transitive, then it is intrinsically harmonic.
\end{theorem}
\begin{proof}
    By smooth Urysohn lemma and Sard's lemma, for each connected component of $\Sl\cup\Zl_v$, one can choose a neighborhood $U$ with smooth boundary, admitting a metric $g$ for which $v\big|_U$ is harmonic, and such that $\star_gv$ is exact on $U$.

    Let $\{U_1, \dots, U_m\}$ be the resulting finite collection of such neighborhoods, and set $U_{m+1} = M \setminus \cup_{j=1}^m\overline{U_j}$. Define $M_i = \overline{U_i}$ for $1 \le i \le m+1$, each equipped with the restricted triple $(v_i, \Sl_i, \I_i) := (v, \Sl, \I)|_{M_i}$. Then the collection $\{M_i\}$ satisfies the hypotheses of Theorem~\ref{thm:gluing}, and the intrinsic harmonicity of $(v, \Sl, \I)$ follows directly from that result.
\end{proof}

\section{Applications}
In this section we apply the gluing theorem \ref{thm:gluing} to several surgery constructions for $\ZT$ harmonic $1$-forms. We first prove a connected-sum result and then formulate a local replacement principle for modifying a chosen component of the total zero set. This replacement principle is then used to blow up isolated ordinary zeros via Euclidean models, introduce prescribed degenerate isolated zeros in dimension three, and split smooth $\vec k$-nondegenerate singular components into nondegenerate ones.

\subsection{Connected sum}\label{subsection:connected_sum}
In this subsection, we prove Theorem \ref{thm:intro_connected_sum}.
Let $(M_i,g_i)~(i=1,2)$ be two closed Riemannian manifolds with $\ZT$ harmonic $1$-forms $(v_i,\Sl_i,\I_i)$. Choose a point $x_i\in M_i$ and an open ball $B^i$ centered at $x_i$, such that $v_i\big|_{B^i}$ is nonvanishing. Let $M:=M_1\#M_2$ denote the connected sum formed by gluing along $B^i\setminus (B^i)'$, where $(B^i)'$ is a smaller open ball centered at $x_i$.

As an application of Theorem \ref{thm:gluing}, we construct a $\ZT$ closed $1$-form
$(v_M,\Sl_M,\I_M)$ and a metric $g_M$ on $M$ such that $v_M$ is harmonic with respect to $g_M$.
Moreover, $v_M$ can be chosen to coincide with $v_i$ on $M_i\setminus B_i$.

Our strategy is to first construct a closed $1$-form on the neck $S^{n-1}\times(-1,1)$ that agrees with
$v_i$ near the boundary. Then, we show that the restrictions of
$(v_i, \Sl_i, \I_i)$ to $M_i \setminus B^i$ are \emph{gluing-compatible} with this closed $1$-form
on the neck, thereby obtaining the $2$-valued closed $1$-form $v_M$.
Finally, we verify the transitivity condition (iii) in Theorem~\ref{thm:gluing}.

\begin{proof}[Proof of Theorem~\ref{thm:intro_connected_sum}]
We begin by constructing a closed $1$-form on $S^{n-1}\times(-1,1)$.
To this end, we embed $S^{n-1}\times(-1,1)$ into $\RR^{n+1}$ as follows.

Let $(x_1,\dots, x_n,x_{n+1})$ be coordinates of $\RR^{n+1}$. In $\RR^{n}$, consider polar coordinates $(r,\sigma)$, where $r$ is the radial coordinate and $\sigma \in S^{n-1}$ represents spherical coordinates on the unit sphere.

Consider the following hypersurface in $\RR^{n+1}$: 
\[W=\{(r,\sigma,x_{n+1})\in\RR^{n+1}:-1/2\le x_{n+1}\le 1/2, r=f(x_{n+1})\},\]
where $f:[-1/2,1/2]\to\RR$ is a smooth function, satisfying:
\begin{enumerate}
    \item $f(\pm 1/2)=1/2$, $f(0)=1/4$;
    \item $f'(t)=0$ if and only if $t=0$; and $f''(0)>0$;
    \item $f^{(k)}(\pm 1/2)=\pm\infty$ for any $k\ge 1$.
\end{enumerate}

We define an embedding $\Psi:S^{n-1}\times (-1,1)\to \RR^{n+1}$, such that
\[\Psi(\sigma,t)=\begin{cases}
    (f(t),\sigma,t), & -\frac{1}{2}\le t\le \frac{1}{2};\\
    (t,\sigma,\mathrm{sign}(t) \frac{1}{2}), &\frac{1}{2}<|t|<1.
\end{cases}\]
We obtain a neck connecting unit balls in $\RR^n\times\{-1/2\}$ and $\RR^n\times\{1/2\}$. Now the restriction of the coordinate function $X_n$ on $\mathrm{im}\Psi$ is a Morse function, with exactly two critical points $(0,\dots,\pm1/4,0)\in\RR^{n+1}$. 

The pullback $\Psi^\ast dX_n$ is the desired closed $1$-form. We can choose a diffeomorphism $\phi_1:B^1\to B_1(0)\times\{1/2\}\subset \RR^n\times\{1/2\}$, sending $(B^1)'$ to the $1/2$-ball. And choose a trivialization of $\I_1$ on $B_1$, so that $(\phi_1)_\ast v_1=dX_n$. Similar for $\I_2\big|_{B_2}$ and $v_2\big|_{B_2}$, we have a diffeomorphism $\phi_2$ from $B^2$ to unit ball on $\RR^{n}\times\{-1/2\}$, pulling back $dX_n$ to $v_2$.

Now we can glue $S^{n-1}\times(-1,1)$ with $M_i\setminus (B^i)'$ by $\Psi$ and $\phi_i$ along the annular region $B_i\setminus (B^i)'~(i=1,2)$, and obtain a global $2$-valued closed $1$-form $(v_M,\Sl_M,\I_M)$ on $M=M_1\# M_2$, where $(v_M,\Sl_M,\I_M)\big|_{S^n\times(-1,1)}=(\Psi^\ast dX_n,\varnothing,\underline{\RR})$.

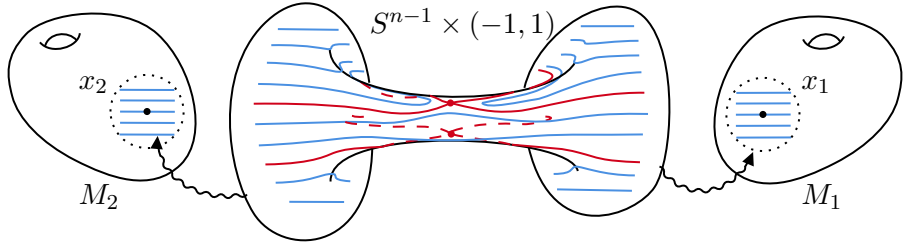
\begin{figure}[!h]
    \centering

\tikzset{every picture/.style={line width=0.75pt}} 

\begin{tikzpicture}[x=0.55pt,y=0.55pt,yscale=-1,xscale=1]

\draw [color={rgb, 255:red, 0; green, 0; blue, 0 }  ,draw opacity=1 ][fill={rgb, 255:red, 255; green, 255; blue, 255 }  ,fill opacity=1 ][line width=0.75]    (465,120) .. controls (467.41,120.59) and (468.44,122.01) .. (468.08,124.26) .. controls (467.92,126.58) and (469.09,127.8) .. (471.58,127.91) .. controls (473.78,127.58) and (475.07,128.59) .. (475.44,130.93) .. controls (476.09,133.29) and (477.47,134.09) .. (479.6,133.32) .. controls (481.76,132.45) and (483.33,133.06) .. (484.32,135.17) .. controls (485.59,137.22) and (487.23,137.59) .. (489.24,136.26) .. controls (491.02,134.77) and (492.7,134.87) .. (494.28,136.57) .. controls (496.12,138.14) and (497.81,137.98) .. (499.34,136.1) .. controls (500.61,134.13) and (502.17,133.74) .. (504.01,134.94) .. controls (506.3,135.87) and (507.82,135.25) .. (508.57,133.08) .. controls (509.32,130.78) and (510.78,129.92) .. (512.97,130.5) .. controls (515.3,130.86) and (516.59,129.86) .. (516.84,127.49) .. controls (516.89,125.15) and (518.09,123.94) .. (520.46,123.85) .. controls (522.91,123.5) and (523.93,122.2) .. (523.52,119.95) -- (524.05,119.19) -- (527.91,112.42) ;
\draw [shift={(529,110)}, rotate = 112.99] [fill={rgb, 255:red, 0; green, 0; blue, 0 }  ,fill opacity=1 ][line width=0.08]  [draw opacity=0] (8.93,-4.29) -- (0,0) -- (8.93,4.29) -- cycle    ;
\draw [color={rgb, 255:red, 0; green, 0; blue, 0 }  ,draw opacity=1 ][fill={rgb, 255:red, 255; green, 255; blue, 255 }  ,fill opacity=1 ][line width=0.75]    (183,139) .. controls (181.49,140.95) and (179.74,141.21) .. (177.75,139.8) .. controls (175.91,138.31) and (174.21,138.44) .. (172.66,140.21) .. controls (171.02,141.92) and (169.38,141.93) .. (167.75,140.24) .. controls (166.28,138.5) and (164.7,138.39) .. (163.02,139.92) .. controls (160.94,141.33) and (159.28,141.08) .. (158.05,139.17) .. controls (156.71,137.17) and (155,136.75) .. (152.92,137.92) .. controls (151.02,139.07) and (149.55,138.56) .. (148.5,136.4) .. controls (147.65,134.24) and (146.03,133.5) .. (143.64,134.19) .. controls (141.39,134.84) and (140.02,134.04) .. (139.53,131.77) .. controls (139.24,129.54) and (137.9,128.54) .. (135.51,128.78) .. controls (133.05,128.83) and (131.86,127.72) .. (131.95,125.46) .. controls (132.09,123.12) and (131,121.82) .. (128.68,121.56) .. controls (126.33,121.08) and (125.44,119.68) .. (126.01,117.37) .. controls (126.8,115.22) and (126.08,113.63) .. (123.86,112.62) -- (122.94,109.66) -- (121.89,101.87) ;
\draw [shift={(122,99)}, rotate = 94.76] [fill={rgb, 255:red, 0; green, 0; blue, 0 }  ,fill opacity=1 ][line width=0.08]  [draw opacity=0] (8.93,-4.29) -- (0,0) -- (8.93,4.29) -- cycle    ;
\draw [color={rgb, 255:red, 208; green, 2; blue, 27 }  ,draw opacity=1 ] [dash pattern={on 4.5pt off 4.5pt}]  (322.67,96) .. controls (342.67,88.5) and (367.17,91) .. (373.17,91) .. controls (379.17,91) and (405.17,85.5) .. (380.17,85.5) ;
\draw [color={rgb, 255:red, 74; green, 144; blue, 226 }  ,draw opacity=1 ]   (554.67,84.39) -- (517.24,84.52) ;
\draw [color={rgb, 255:red, 74; green, 144; blue, 226 }  ,draw opacity=1 ]   (254.67,58.5) .. controls (268.67,69) and (296.67,70) .. (308.67,74.5) ;
\draw [color={rgb, 255:red, 74; green, 144; blue, 226 }  ,draw opacity=1 ]   (343.67,76.5) .. controls (352.17,74.5) and (366.17,73) .. (378.67,69) .. controls (391.17,65) and (410.17,53) .. (387.67,51.5) ;
\draw [color={rgb, 255:red, 208; green, 2; blue, 27 }  ,draw opacity=1 ] [dash pattern={on 4.5pt off 4.5pt}]  (260.67,62) .. controls (274.67,72.5) and (285.67,74.5) .. (308.17,71.5) ;
\draw [color={rgb, 255:red, 74; green, 144; blue, 226 }  ,draw opacity=1 ]   (96.31,82.39) -- (133.74,82.52) ;
\draw   (32,29.5) .. controls (46.33,14.33) and (97.83,7.33) .. (122,29.5) .. controls (146.17,51.67) and (157.67,99.17) .. (137.67,118.67) .. controls (117.67,138.17) and (89.17,127.67) .. (65.67,115.67) .. controls (42.17,103.67) and (26.67,87.17) .. (22.67,72.17) .. controls (18.67,57.17) and (17.67,44.67) .. (32,29.5) -- cycle ;
\draw    (41.55,35.05) .. controls (49.12,44.1) and (68,38.39) .. (68.79,28.81) ;
\draw    (44.96,37.94) .. controls (47.32,28.22) and (61.02,26.92) .. (66.72,33.79) ;
\draw  [fill={rgb, 255:red, 0; green, 0; blue, 0 }  ,fill opacity=1 ] (113.02,82.33) .. controls (113.02,81.38) and (113.79,80.6) .. (114.75,80.6) .. controls (115.71,80.6) and (116.48,81.38) .. (116.48,82.33) .. controls (116.48,83.29) and (115.71,84.07) .. (114.75,84.07) .. controls (113.79,84.07) and (113.02,83.29) .. (113.02,82.33) -- cycle ;
\draw [color={rgb, 255:red, 74; green, 144; blue, 226 }  ,draw opacity=1 ]   (96,67.5) -- (133.44,67.63) ;
\draw [color={rgb, 255:red, 74; green, 144; blue, 226 }  ,draw opacity=1 ]   (96.31,74.75) -- (133.74,74.88) ;
\draw [color={rgb, 255:red, 74; green, 144; blue, 226 }  ,draw opacity=1 ]   (96.31,90.02) -- (133.74,90.15) ;
\draw [color={rgb, 255:red, 74; green, 144; blue, 226 }  ,draw opacity=1 ]   (96,98.04) -- (133.44,98.17) ;
\draw  [dash pattern={on 0.84pt off 2.51pt}] (89.75,82.33) .. controls (89.75,68.53) and (100.94,57.33) .. (114.75,57.33) .. controls (128.56,57.33) and (139.75,68.53) .. (139.75,82.33) .. controls (139.75,96.14) and (128.56,107.33) .. (114.75,107.33) .. controls (100.94,107.33) and (89.75,96.14) .. (89.75,82.33) -- cycle ;
\draw   (618.48,31.5) .. controls (604.15,16.33) and (552.65,9.33) .. (528.48,31.5) .. controls (504.32,53.67) and (492.82,101.17) .. (512.82,120.67) .. controls (532.82,140.17) and (561.32,129.67) .. (584.82,117.67) .. controls (608.32,105.67) and (623.82,89.17) .. (627.82,74.17) .. controls (631.82,59.17) and (632.82,46.67) .. (618.48,31.5) -- cycle ;
\draw    (591.09,31.02) .. controls (586.49,41.89) and (566.77,41.97) .. (563.21,33.03) ;
\draw    (588.67,34.79) .. controls (583.57,26.19) and (570.09,28.95) .. (566.65,37.2) ;
\draw  [fill={rgb, 255:red, 0; green, 0; blue, 0 }  ,fill opacity=1 ] (537.47,84.33) .. controls (537.47,83.38) and (536.69,82.6) .. (535.73,82.6) .. controls (534.78,82.6) and (534,83.38) .. (534,84.33) .. controls (534,85.29) and (534.78,86.07) .. (535.73,86.07) .. controls (536.69,86.07) and (537.47,85.29) .. (537.47,84.33) -- cycle ;
\draw [color={rgb, 255:red, 74; green, 144; blue, 226 }  ,draw opacity=1 ]   (554.98,69.5) -- (517.55,69.63) ;
\draw [color={rgb, 255:red, 74; green, 144; blue, 226 }  ,draw opacity=1 ]   (554.67,76.75) -- (517.24,76.88) ;
\draw [color={rgb, 255:red, 74; green, 144; blue, 226 }  ,draw opacity=1 ]   (554.67,92.02) -- (517.24,92.15) ;
\draw [color={rgb, 255:red, 74; green, 144; blue, 226 }  ,draw opacity=1 ]   (554.98,100.04) -- (517.55,100.17) ;
\draw  [dash pattern={on 0.84pt off 2.51pt}] (560.73,84.33) .. controls (560.73,70.53) and (549.54,59.33) .. (535.73,59.33) .. controls (521.93,59.33) and (510.73,70.53) .. (510.73,84.33) .. controls (510.73,98.14) and (521.93,109.33) .. (535.73,109.33) .. controls (549.54,109.33) and (560.73,98.14) .. (560.73,84.33) -- cycle ;
\draw    (239.67,39.67) .. controls (239.67,79.67) and (387.17,83) .. (407.17,49) ;
\draw    (240.17,130.17) .. controls (239.17,90.17) and (404.17,98) .. (409.17,123) ;
\draw    (268.67,106.67) .. controls (267.67,180.67) and (206.67,182) .. (187.67,152) .. controls (168.67,122) and (161.17,56.5) .. (194.17,23) .. controls (227.17,-10.5) and (260.67,18.17) .. (266.67,62.67) ;
\draw    (375.67,107.17) .. controls (392.67,176.17) and (449.17,153) .. (461.67,134) .. controls (474.17,115) and (478.17,43.5) .. (455.67,19) .. controls (433.17,-5.5) and (381.67,27.67) .. (372.67,66.17) ;
\draw [color={rgb, 255:red, 74; green, 144; blue, 226 }  ,draw opacity=1 ]   (204.17,26) -- (245.67,26.5) ;
\draw [color={rgb, 255:red, 74; green, 144; blue, 226 }  ,draw opacity=1 ]   (212,144.04) -- (249.44,144.17) ;
\draw [color={rgb, 255:red, 74; green, 144; blue, 226 }  ,draw opacity=1 ]   (196.67,49) .. controls (228.67,49.5) and (243.17,56.5) .. (247.17,55) .. controls (251.17,53.5) and (238.67,47.5) .. (255.17,48.5) ;
\draw [color={rgb, 255:red, 74; green, 144; blue, 226 }  ,draw opacity=1 ]   (190.67,63.5) .. controls (234.17,64) and (253.67,70) .. (267.67,73) .. controls (281.67,76) and (306.92,81.25) .. (308.67,74.5) ;
\draw [color={rgb, 255:red, 74; green, 144; blue, 226 }  ,draw opacity=1 ]   (254.67,58.5) .. controls (253.67,55.5) and (250.67,54) .. (263.17,55) ;
\draw [color={rgb, 255:red, 208; green, 2; blue, 27 }  ,draw opacity=1 ]   (187.67,77) .. controls (229.17,76.5) and (256.67,79) .. (272.67,82) .. controls (288.67,85) and (305.67,84.5) .. (322.17,76.5) ;
\draw  [color={rgb, 255:red, 208; green, 2; blue, 27 }  ,draw opacity=1 ][fill={rgb, 255:red, 208; green, 2; blue, 27 }  ,fill opacity=1 ] (320.43,76.5) .. controls (320.43,75.54) and (321.21,74.77) .. (322.17,74.77) .. controls (323.12,74.77) and (323.9,75.54) .. (323.9,76.5) .. controls (323.9,77.46) and (323.12,78.23) .. (322.17,78.23) .. controls (321.21,78.23) and (320.43,77.46) .. (320.43,76.5) -- cycle ;
\draw [color={rgb, 255:red, 208; green, 2; blue, 27 }  ,draw opacity=1 ]   (322.17,76.5) .. controls (348.17,83) and (355.17,83) .. (368.17,82) .. controls (381.17,81) and (417.67,69) .. (454.67,69.5) ;
\draw [color={rgb, 255:red, 208; green, 2; blue, 27 }  ,draw opacity=1 ]   (322.17,76.5) .. controls (333.67,74.23) and (346.67,74.5) .. (365.67,69) ;
\draw [color={rgb, 255:red, 208; green, 2; blue, 27 }  ,draw opacity=1 ]   (376.17,66.5) .. controls (383.67,64) and (403.67,57.5) .. (382.17,54) ;
\draw [color={rgb, 255:red, 208; green, 2; blue, 27 }  ,draw opacity=1 ]   (322.17,76.5) .. controls (313.67,75) and (318.17,75) .. (308.17,71.5) ;
\draw [color={rgb, 255:red, 74; green, 144; blue, 226 }  ,draw opacity=1 ]   (451.17,44) .. controls (442.92,44) and (425.89,44.65) .. (421.38,45.53) .. controls (416.87,46.41) and (409.99,50.06) .. (407.17,49) .. controls (404.34,47.94) and (414.17,42.5) .. (397.67,43.5) ;
\draw [color={rgb, 255:red, 74; green, 144; blue, 226 }  ,draw opacity=1 ]   (452.67,55.5) .. controls (417.67,56.5) and (391.17,72.5) .. (380.67,74) .. controls (370.17,75.5) and (347.67,81) .. (343.67,76.5) ;
\draw [color={rgb, 255:red, 74; green, 144; blue, 226 }  ,draw opacity=1 ]   (199.67,37) .. controls (227.67,36.5) and (236.67,40.17) .. (239.67,39.67) .. controls (242.67,39.17) and (236.67,35.5) .. (253.17,36.5) ;
\draw [color={rgb, 255:red, 74; green, 144; blue, 226 }  ,draw opacity=1 ]   (450.67,34) .. controls (435.17,34) and (422.17,36) .. (416.17,36) .. controls (410.17,36) and (419.67,33) .. (404.17,33.5) ;
\draw [color={rgb, 255:red, 74; green, 144; blue, 226 }  ,draw opacity=1 ]   (188.67,91) .. controls (238.67,90) and (252.67,88.5) .. (272.67,90) .. controls (292.67,91.5) and (309.67,83) .. (322.67,83.5) .. controls (335.67,84) and (371.67,91.5) .. (394.17,88.5) .. controls (416.67,85.5) and (420.17,83.5) .. (457.17,84) ;
\draw [color={rgb, 255:red, 74; green, 144; blue, 226 }  ,draw opacity=1 ]   (411.23,24.04) -- (447.67,24) ;
\draw  [color={rgb, 255:red, 208; green, 2; blue, 27 }  ,draw opacity=1 ][fill={rgb, 255:red, 208; green, 2; blue, 27 }  ,fill opacity=1 ] (324.4,97.73) .. controls (324.4,96.78) and (323.62,96) .. (322.67,96) .. controls (321.71,96) and (320.93,96.78) .. (320.93,97.73) .. controls (320.93,98.69) and (321.71,99.47) .. (322.67,99.47) .. controls (323.62,99.47) and (324.4,98.69) .. (324.4,97.73) -- cycle ;
\draw [color={rgb, 255:red, 208; green, 2; blue, 27 }  ,draw opacity=1 ] [dash pattern={on 4.5pt off 4.5pt}]  (322.67,97.73) .. controls (312.67,90) and (291.17,93.5) .. (278.67,92.5) .. controls (266.17,91.5) and (243.17,91) .. (250.67,87) .. controls (258.17,83) and (251.17,87.5) .. (267.17,85.5) ;
\draw [color={rgb, 255:red, 208; green, 2; blue, 27 }  ,draw opacity=1 ] [dash pattern={on 4.5pt off 4.5pt}]  (322.67,97.73) .. controls (335.17,101) and (346.67,101) .. (364.17,105) ;
\draw [color={rgb, 255:red, 208; green, 2; blue, 27 }  ,draw opacity=1 ]   (364.17,105) .. controls (396.67,108) and (396.67,107) .. (410.17,110.5) .. controls (423.67,114) and (424.17,116) .. (451.67,117) ;
\draw [color={rgb, 255:red, 208; green, 2; blue, 27 }  ,draw opacity=1 ] [dash pattern={on 4.5pt off 4.5pt}]  (275.67,106.5) .. controls (285.67,102.5) and (306.67,102) .. (322.67,97.73) ;
\draw [color={rgb, 255:red, 208; green, 2; blue, 27 }  ,draw opacity=1 ]   (195.67,118.5) .. controls (231.17,118) and (238.67,111.5) .. (244.17,110) .. controls (249.67,108.5) and (251.67,107.5) .. (275.67,106.5) ;
\draw [color={rgb, 255:red, 74; green, 144; blue, 226 }  ,draw opacity=1 ]   (191.17,105.5) .. controls (236.67,103.5) and (252.17,98.5) .. (272.17,100) .. controls (292.17,101.5) and (306.67,102) .. (322.17,102) .. controls (337.67,102) and (358.67,100.5) .. (380.67,99) .. controls (402.67,97.5) and (423.17,98.5) .. (453.67,98) ;
\draw [color={rgb, 255:red, 74; green, 144; blue, 226 }  ,draw opacity=1 ]   (202.67,131) .. controls (227.67,130.5) and (233.67,125) .. (240.67,125) .. controls (247.67,125) and (237.17,129) .. (253.67,130) ;
\draw [color={rgb, 255:red, 74; green, 144; blue, 226 }  ,draw opacity=1 ]   (392.67,125) .. controls (407.67,125) and (403.67,120) .. (408.17,120) .. controls (412.67,120) and (426.17,126) .. (448.17,126) ;
\draw [color={rgb, 255:red, 74; green, 144; blue, 226 }  ,draw opacity=1 ]   (395,135.54) -- (446.17,136.5) ;

\draw (66,130) node [anchor=north west][inner sep=0.75pt]    {$M_{2}$};
\draw (66,56) node [anchor=north west][inner sep=0.75pt]    {$x_{2}$};
\draw (560,130) node [anchor=north west][inner sep=0.75pt]    {$M_{1}$};
\draw (560,56) node [anchor=north west][inner sep=0.75pt]    {$x_{1}$};
\draw (265,10) node [anchor=north west][inner sep=0.75pt]    {$S^{n-1} \times ( -1,1)$};

\end{tikzpicture}

    \caption{Connected sum of $\ZT$ harmonic $1$-forms.}
    \label{fig:connected_sum}
\end{figure}

Next we verify that $v_M$ is transitive on $M$. 

First, since each $v_i$ is a $\ZT$ harmonic $1$-form on $(M_i, g_i)$, Theorem~\ref{thm:harmonic_implies_transitivity} implies that they are transitive.
We may choose $B^i$ to be contained in a flat chart $U_i$ of $\ker v_i$.  
By Theorem~\ref{thm:harmonic_implies_transitivity}, every point  
$y \in M_i \setminus \Sl_i$ is passed by a $v_i$-positive loop.  
In particular, any $y \in U_i \setminus B^i$ lies on such a loop $\gamma$,  
which can be further modified to avoid $B^i$.

\begin{figure}[!h]
    \centering

\tikzset{every picture/.style={line width=0.75pt}} 

\begin{tikzpicture}[x=0.5pt,y=0.5pt,yscale=-1,xscale=1]

\draw   (100.83,80.5) -- (200.33,80.5) -- (200.33,180) -- (100.83,180) -- cycle ;
\draw [color={rgb, 255:red, 74; green, 144; blue, 226 }  ,draw opacity=1 ]   (200,130) -- (100,130) ;
\draw  [fill={rgb, 255:red, 0; green, 0; blue, 0 }  ,fill opacity=1 ] (152.47,130.33) .. controls (152.47,129.38) and (151.69,128.6) .. (150.73,128.6) .. controls (149.78,128.6) and (149,129.38) .. (149,130.33) .. controls (149,131.29) and (149.78,132.07) .. (150.73,132.07) .. controls (151.69,132.07) and (152.47,131.29) .. (152.47,130.33) -- cycle ;
\draw [color={rgb, 255:red, 74; green, 144; blue, 226 }  ,draw opacity=1 ]   (200,120) -- (100,120) ;
\draw [color={rgb, 255:red, 74; green, 144; blue, 226 }  ,draw opacity=1 ]   (200,140) -- (100,140) ;
\draw [color={rgb, 255:red, 74; green, 144; blue, 226 }  ,draw opacity=1 ]   (200,150) -- (100,150) ;
\draw  [dash pattern={on 0.84pt off 2.51pt}] (175.73,130.33) .. controls (175.73,116.53) and (164.54,105.33) .. (150.73,105.33) .. controls (136.93,105.33) and (125.73,116.53) .. (125.73,130.33) .. controls (125.73,144.14) and (136.93,155.33) .. (150.73,155.33) .. controls (164.54,155.33) and (175.73,144.14) .. (175.73,130.33) -- cycle ;
\draw [color={rgb, 255:red, 74; green, 144; blue, 226 }  ,draw opacity=1 ]   (200,160) -- (100,160) ;
\draw [color={rgb, 255:red, 74; green, 144; blue, 226 }  ,draw opacity=1 ]   (200,170) -- (100,170) ;
\draw [color={rgb, 255:red, 74; green, 144; blue, 226 }  ,draw opacity=1 ]   (200,110) -- (100,110) ;
\draw [color={rgb, 255:red, 74; green, 144; blue, 226 }  ,draw opacity=1 ]   (200,100) -- (100,100) ;
\draw [color={rgb, 255:red, 74; green, 144; blue, 226 }  ,draw opacity=1 ]   (200,90) -- (100,90) ;
\draw  [fill={rgb, 255:red, 0; green, 0; blue, 0 }  ,fill opacity=1 ] (150,91.73) .. controls (150,90.78) and (149.22,90) .. (148.27,90) .. controls (147.31,90) and (146.53,90.78) .. (146.53,91.73) .. controls (146.53,92.69) and (147.31,93.47) .. (148.27,93.47) .. controls (149.22,93.47) and (150,92.69) .. (150,91.73) -- cycle ;
\draw    (100,210) .. controls (149.09,178.82) and (104.29,138.08) .. (169.02,61.16) ;
\draw [shift={(170,60)}, rotate = 130.43] [color={rgb, 255:red, 0; green, 0; blue, 0 }  ][line width=0.75]    (10.93,-3.29) .. controls (6.95,-1.4) and (3.31,-0.3) .. (0,0) .. controls (3.31,0.3) and (6.95,1.4) .. (10.93,3.29)   ;
\draw   (320.5,80.5) -- (420,80.5) -- (420,180) -- (320.5,180) -- cycle ;
\draw [color={rgb, 255:red, 74; green, 144; blue, 226 }  ,draw opacity=1 ]   (419.67,130) -- (319.67,130) ;
\draw  [fill={rgb, 255:red, 0; green, 0; blue, 0 }  ,fill opacity=1 ] (372.13,130.33) .. controls (372.13,129.38) and (371.36,128.6) .. (370.4,128.6) .. controls (369.44,128.6) and (368.67,129.38) .. (368.67,130.33) .. controls (368.67,131.29) and (369.44,132.07) .. (370.4,132.07) .. controls (371.36,132.07) and (372.13,131.29) .. (372.13,130.33) -- cycle ;
\draw [color={rgb, 255:red, 74; green, 144; blue, 226 }  ,draw opacity=1 ]   (419.67,120) -- (319.67,120) ;
\draw [color={rgb, 255:red, 74; green, 144; blue, 226 }  ,draw opacity=1 ]   (419.67,140) -- (319.67,140) ;
\draw [color={rgb, 255:red, 74; green, 144; blue, 226 }  ,draw opacity=1 ]   (419.67,150) -- (319.67,150) ;
\draw  [dash pattern={on 0.84pt off 2.51pt}] (395.4,130.33) .. controls (395.4,116.53) and (384.21,105.33) .. (370.4,105.33) .. controls (356.59,105.33) and (345.4,116.53) .. (345.4,130.33) .. controls (345.4,144.14) and (356.59,155.33) .. (370.4,155.33) .. controls (384.21,155.33) and (395.4,144.14) .. (395.4,130.33) -- cycle ;
\draw [color={rgb, 255:red, 74; green, 144; blue, 226 }  ,draw opacity=1 ]   (419.67,160) -- (319.67,160) ;
\draw [color={rgb, 255:red, 74; green, 144; blue, 226 }  ,draw opacity=1 ]   (419.67,170) -- (319.67,170) ;
\draw [color={rgb, 255:red, 74; green, 144; blue, 226 }  ,draw opacity=1 ]   (419.67,110) -- (319.67,110) ;
\draw [color={rgb, 255:red, 74; green, 144; blue, 226 }  ,draw opacity=1 ]   (419.67,100) -- (319.67,100) ;
\draw [color={rgb, 255:red, 74; green, 144; blue, 226 }  ,draw opacity=1 ]   (419.67,90) -- (319.67,90) ;
\draw  [fill={rgb, 255:red, 0; green, 0; blue, 0 }  ,fill opacity=1 ] (369.67,91.73) .. controls (369.67,90.78) and (368.89,90) .. (367.93,90) .. controls (366.98,90) and (366.2,90.78) .. (366.2,91.73) .. controls (366.2,92.69) and (366.98,93.47) .. (367.93,93.47) .. controls (368.89,93.47) and (369.67,92.69) .. (369.67,91.73) -- cycle ;
\draw    (319.67,210) .. controls (412.33,159.17) and (325.17,149.83) .. (330,130) .. controls (334.81,110.27) and (354.14,122.54) .. (389.14,60.93) ;
\draw [shift={(389.67,60)}, rotate = 119.42] [color={rgb, 255:red, 0; green, 0; blue, 0 }  ][line width=0.75]    (10.93,-3.29) .. controls (6.95,-1.4) and (3.31,-0.3) .. (0,0) .. controls (3.31,0.3) and (6.95,1.4) .. (10.93,3.29)   ;
\draw [color={rgb, 255:red, 74; green, 144; blue, 226 }  ,draw opacity=1 ]   (110,120) -- (110,95) ;
\draw [shift={(110,93)}, rotate = 90] [color={rgb, 255:red, 74; green, 144; blue, 226 }  ,draw opacity=1 ][line width=0.75]    (10.93,-3.29) .. controls (6.95,-1.4) and (3.31,-0.3) .. (0,0) .. controls (3.31,0.3) and (6.95,1.4) .. (10.93,3.29)   ;
\draw [color={rgb, 255:red, 74; green, 144; blue, 226 }  ,draw opacity=1 ]   (330,120) -- (330,95) ;
\draw [shift={(330,93)}, rotate = 90] [color={rgb, 255:red, 74; green, 144; blue, 226 }  ,draw opacity=1 ][line width=0.75]    (10.93,-3.29) .. controls (6.95,-1.4) and (3.31,-0.3) .. (0,0) .. controls (3.31,0.3) and (6.95,1.4) .. (10.93,3.29)   ;
\draw   (222.5,111) -- (264.5,111) -- (264.5,101) -- (292.5,121) -- (264.5,141) -- (264.5,131) -- (222.5,131) -- cycle ;

\draw (157,82.4) node [anchor=north west][inner sep=0.75pt]    {${\textstyle y}$};
\draw (137,52.4) node [anchor=north west][inner sep=0.75pt]    {$\gamma $};
\draw (161,112.4) node [anchor=north west][inner sep=0.75pt]    {$B^{i}$};
\draw (376.67,82.4) node [anchor=north west][inner sep=0.75pt]    {${\textstyle y}$};
\draw (356.67,52.4) node [anchor=north west][inner sep=0.75pt]    {$\gamma '$};
\draw (380.67,112.4) node [anchor=north west][inner sep=0.75pt]    {$B^{i}$};
\draw (80,92.4) node [anchor=north west][inner sep=0.75pt]    {$\textcolor[rgb]{0.29,0.56,0.89}{v_{i}}$};
\draw (300,92.4) node [anchor=north west][inner sep=0.75pt]    {$\textcolor[rgb]{0.29,0.56,0.89}{v}\textcolor[rgb]{0.29,0.56,0.89}{_{i}}$};
\draw (231,114) node [anchor=north west][inner sep=0.75pt]   [align=left] {\tiny{modify}};

\end{tikzpicture}

    \caption{Modify $\gamma$ to avoid $B^i$ while preserving positivity.}
    \label{fig:modify}
\end{figure}
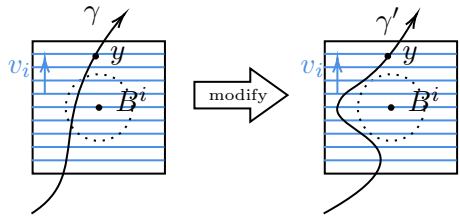

For any $x = (\sigma, t) \in S^{n-1} \times (-1, 1)$ that is not a critical point of $X_n \circ \Psi$,  
the leaf of $\Psi^* dX_n$ through $x$ intersects the boundary $S^{n-1}\times\{\pm1\}$.
Therefore, when $x$ is considered as a point of $M$, it lies in a leaf of $\ker v_M$, that extends into $U_i \setminus B^i \subset M_i$ for some $i=1,2$. There is some point $y \in U_i \setminus B^i \subset M_i$ that lies in the same leaf as $x$.  
By Lemma \ref{lem:modify_positive_loop_in_one_leaf}, we can modify the $v_i$-positive loop $\gamma \subset M_i \setminus B^i$ through $y$ to obtain a $v_M$-positive loop $\gamma'$ that passes through $x$.

Now, let $M_3,M_4$ be Morse neighborhoods of the two critical points of $X_n\circ\Psi$ on $S^{n-1}\times(-1,1)$. And let $M_5$ be $M\setminus \cup_{i=1}^4 M_i$. Thus $M_i~(1\le i\le 4)$ are the harmonic pieces and $M_5$ is the transition piece. 
Then the collection $\{M_i\}_{1\le i\le 5}$ and $v_M\big|_{M_i}$ satisfies the hypotheses of Theorem~\ref{thm:gluing}. And the existence of $g_M$ follows directly from that result.
\end{proof}

\subsection{A local replacement principle}\label{sec:changing_zero_sets}
This subsection proves a local replacement principle for a chosen connected component \(Y\) of the total zero set \(S\cup Z_v\). The main point is that positive loops in the gluing collar can be chosen uniformly away from \(Y\), so a sufficiently small replacement near \(Y\) preserves the transitivity condition in Theorem~\ref{thm:gluing}, see also \cite{Chen2025Perturbation} We will use this principle below to blow up isolated ordinary zeros and to split higher-order smooth branching components.

Let $(v,\Sl,\I)$ be a $\ZT$ harmonic $1$-form on a closed manifold $(M,g)$.
Let $Y$ be a connected component of the zero set $\Sl\cup\Zl_v$.
Let $U_Y$ be a precompact neighborhood of $Y$ such that
$v$ is nonvanishing on $U_Y\setminus Y$, $\p U_Y$ is a connected hypersurface
in $M$, and $M\setminus U_Y$ is connected.
Suppose $v$ is exact and $\star$-exact on $\p U_Y\times(-1,1)$.

Fix a tubular neighborhood of $\p U_Y$ and identify it with $\p U_Y\times(-1,1)$, so that $\p U_Y\times\{-1\}$ is the boundary component nearer to $Y$.
Define
$M_1 = M \setminus \Big(U_Y \cup \big(\p U_Y\times(-1,\tfrac{1}{2}]\big)\Big)$,
$M_2 = U_Y \setminus \big(\p U_Y\times[-\tfrac{1}{2},1)\big)$,
and $M_2'= M_2\cup\big(\p U_Y\times (-1,1)\big)$.
The boundary $\p U_Y$ plays the role of $X$ in Definition \ref{def:Gluing_compatible_pair}.

In this subsection, we prove a technical lemma that modifies $(v,\Sl,\I)$
on $M_2$ to produce a new $\ZT$ closed
$1$-form $(v',\Sl',\I')$ and a metric $g'$ making $v'$ harmonic.
The resulting $1$-form coincides with $v$ on $M_1$.
The main step is to verify the transitivity
condition (iii) of Theorem \ref{thm:gluing} on the gluing region
$\p U_Y\times(-1,1)$.

\begin{lemma}\label{lem:positive_curve_on_gluing_region}
    For each $x \in \p U_Y \times (-1,1)$, there exists a $v$-positive loop $\gamma_x$ in $M$ passing through $x$. These loops can be chosen so that, for all $t$ in the domain of $\gamma_x$ and all $x\in\p U_Y\times(-1,1)$, \begin{align}
        |\langle v(\gamma_x(t)),\dot{\gamma}_x(t)\rangle|\ge \frac12|v(\gamma_x(t))|\cdot|\dot{\gamma}_x(t)|.\label{eq:angle_condition}
    \end{align}
     Moreover, $\mathrm{dist}_g(\gamma_x, Y) \ge r_0 > 0$ for some $r_0>0$ independent of $x$.
\end{lemma}
\begin{proof}
   For each $x\in \p U_Y\times (-1,1)$, choose a flat chart $U_x$ of $\ker v$
centered at $x$. After shrinking $U_x$, we may
assume it takes the form
$\{(x_1,\dots,x_n):\sum_{i=1}^{n-1}|x_i|^2<\epsilon^2,\;|x_n|<\sqrt{\epsilon}\}$
for some sufficiently small $\epsilon>0$. In particular, each leaf-slice
$\{x_n=c\}$ has diameter at most $2\epsilon$, while the $x_n$-interval has length $2\sqrt{\epsilon}\gg 2\epsilon$. Let $v=\pm dx_n$ in these coordinates.

We cover the neck region $\p U_Y\times (-1,1)$ by finitely many such thin charts
$\{U_j\}_{1\le j\le J}$, where $U_j=U_{x_j}$ for some $x_j$.
For each $U_j$, choose a slightly larger flat chart $U_j'\Supset U_j$ that is identified with $\{(x_1,\dots,x_n):\sum_{i=1}^{n-1}|x_i|^2<2\epsilon^2,\;|x_n|<4\sqrt{\epsilon}\}$.
We can choose $\epsilon>0$ sufficiently small so that $\cup_j\overline{U_j'}$ is disjoint from $Y$.

    By the proof of Theorem \ref{thm:harmonic_implies_transitivity}, for each $U'_j$, there exists a $v$-positive loop $\gamma_j$ intersecting every local leaf in $U_j'$. In fact, outside $U'_j$, $\gamma_j$ is a flowline of the vector field dual to $v$. Hence
    $|\langle v(\gamma_j(t)),\dot{\gamma}_j(t)\rangle|=|v(\gamma_j(t))|\cdot|\dot{\gamma}_j(t)|$
    whenever $\gamma_j(t)\notin U_j$.

    Consequently, for any point $x\in U_j$, we modify $\gamma_j$ within $U_j'$ to obtain a $v$-positive loop $\gamma_x$ passing through $x$, as in Lemma \ref{lem:modify_positive_loop_in_one_leaf}. We may choose $\gamma_x$ so that for $\gamma_x(t)\in U_j'$,
    \[|\langle v(\gamma_x(t)),\dot{\gamma}_x(t)\rangle|\ge \frac12|v(\gamma_x(t))|\cdot|\dot{\gamma}_x(t)|,\]
    by our assumption that the diameter of each slice $\{x_n=c\}$ in $U'_j$ is much smaller than the distance between $\{x_n=4\sqrt{\epsilon}\}$ and $\{x_n=-4\sqrt{\epsilon}\}$.

\begin{figure}[!h]
    \centering

\tikzset{every picture/.style={line width=0.75pt}} 

\begin{tikzpicture}[x=0.5pt,y=0.5pt,yscale=-1,xscale=1]

\draw [color={rgb, 255:red, 74; green, 144; blue, 226 }  ,draw opacity=1 ]   (400,60) -- (400,110) ;
\draw [color={rgb, 255:red, 74; green, 144; blue, 226 }  ,draw opacity=1 ]   (130,60) -- (130,110) ;
\draw   (130,65) -- (400,65) -- (400,105) -- (130,105) -- cycle ;
\draw [color={rgb, 255:red, 74; green, 144; blue, 226 }  ,draw opacity=1 ]   (220,60) -- (220,110) ;
\draw    (86.5,95) -- (148,95) ;
\draw [shift={(150,95)}, rotate = 180] [color={rgb, 255:red, 0; green, 0; blue, 0 }  ][line width=0.75]    (10.93,-3.29) .. controls (6.95,-1.4) and (3.31,-0.3) .. (0,0) .. controls (3.31,0.3) and (6.95,1.4) .. (10.93,3.29)   ;
\draw    (380,75) -- (453,75) ;
\draw [shift={(455,75)}, rotate = 180] [color={rgb, 255:red, 0; green, 0; blue, 0 }  ][line width=0.75]    (10.93,-3.29) .. controls (6.95,-1.4) and (3.31,-0.3) .. (0,0) .. controls (3.31,0.3) and (6.95,1.4) .. (10.93,3.29)   ;
\draw  [dash pattern={on 4.5pt off 4.5pt}]  (56.5,95) -- (64.55,95) -- (86.5,95) ;
\draw  [dash pattern={on 4.5pt off 4.5pt}]  (455,75) -- (485,75) ;
\draw    (150,95) .. controls (188.67,95) and (317.17,75) .. (380,75) ;
\draw   (90.05,60) -- (440.55,60) -- (440.55,110) -- (90.05,110) -- cycle ;
\draw [color={rgb, 255:red, 74; green, 144; blue, 226 }  ,draw opacity=1 ]   (190,60) -- (190,110) ;
\draw [color={rgb, 255:red, 74; green, 144; blue, 226 }  ,draw opacity=1 ]   (160,60) -- (160,110) ;
\draw [color={rgb, 255:red, 74; green, 144; blue, 226 }  ,draw opacity=1 ]   (250,60) -- (250,110) ;
\draw [color={rgb, 255:red, 74; green, 144; blue, 226 }  ,draw opacity=1 ]   (280,60) -- (280,110) ;
\draw [color={rgb, 255:red, 74; green, 144; blue, 226 }  ,draw opacity=1 ]   (310,60) -- (310,110) ;
\draw [color={rgb, 255:red, 74; green, 144; blue, 226 }  ,draw opacity=1 ]   (340,60) -- (340,110) ;
\draw [color={rgb, 255:red, 74; green, 144; blue, 226 }  ,draw opacity=1 ]   (370,60) -- (370,110) ;
\draw [color={rgb, 255:red, 74; green, 144; blue, 226 }  ,draw opacity=1 ]   (370.5,99.5) -- (388.67,99.65) ;
\draw [shift={(390.67,99.67)}, rotate = 180.47] [color={rgb, 255:red, 74; green, 144; blue, 226 }  ,draw opacity=1 ][line width=0.75]    (10.93,-3.29) .. controls (6.95,-1.4) and (3.31,-0.3) .. (0,0) .. controls (3.31,0.3) and (6.95,1.4) .. (10.93,3.29)   ;
\draw  [fill={rgb, 255:red, 0; green, 0; blue, 0 }  ,fill opacity=1 ] (171.13,76.33) .. controls (171.13,75.38) and (170.36,74.6) .. (169.4,74.6) .. controls (168.44,74.6) and (167.67,75.38) .. (167.67,76.33) .. controls (167.67,77.29) and (168.44,78.07) .. (169.4,78.07) .. controls (170.36,78.07) and (171.13,77.29) .. (171.13,76.33) -- cycle ;
\draw [color={rgb, 255:red, 74; green, 144; blue, 226 }  ,draw opacity=1 ]   (100.5,60) -- (100.5,110) ;
\draw [color={rgb, 255:red, 74; green, 144; blue, 226 }  ,draw opacity=1 ]   (430.5,60) -- (430.5,110) ;
\draw [color={rgb, 255:red, 74; green, 144; blue, 226 }  ,draw opacity=1 ]   (400,150.5) -- (400,200.5) ;
\draw [color={rgb, 255:red, 74; green, 144; blue, 226 }  ,draw opacity=1 ]   (130,150.5) -- (130,200.5) ;
\draw   (130,155.5) -- (400,155.5) -- (400,195.5) -- (130,195.5) -- cycle ;
\draw [color={rgb, 255:red, 74; green, 144; blue, 226 }  ,draw opacity=1 ]   (220,150.5) -- (220,200.5) ;
\draw    (86.5,185.5) -- (106.55,185.56) ;
\draw [shift={(108.55,185.56)}, rotate = 180.17] [color={rgb, 255:red, 0; green, 0; blue, 0 }  ][line width=0.75]    (10.93,-3.29) .. controls (6.95,-1.4) and (3.31,-0.3) .. (0,0) .. controls (3.31,0.3) and (6.95,1.4) .. (10.93,3.29)   ;
\draw    (380,165.5) -- (453,165.5) ;
\draw [shift={(455,165.5)}, rotate = 180] [color={rgb, 255:red, 0; green, 0; blue, 0 }  ][line width=0.75]    (10.93,-3.29) .. controls (6.95,-1.4) and (3.31,-0.3) .. (0,0) .. controls (3.31,0.3) and (6.95,1.4) .. (10.93,3.29)   ;
\draw  [dash pattern={on 4.5pt off 4.5pt}]  (56.5,185.5) -- (64.55,185.5) -- (86.5,185.5) ;
\draw  [dash pattern={on 4.5pt off 4.5pt}]  (455,165.5) -- (485,165.5) ;
\draw    (108.55,185.56) .. controls (148.05,178.56) and (138.05,166.56) .. (199.05,163.56) .. controls (260.05,160.56) and (348.58,165.5) .. (380,165.5) ;
\draw   (90.05,150.5) -- (440.55,150.5) -- (440.55,200.5) -- (90.05,200.5) -- cycle ;
\draw [color={rgb, 255:red, 74; green, 144; blue, 226 }  ,draw opacity=1 ]   (190,150.5) -- (190,200.5) ;
\draw [color={rgb, 255:red, 74; green, 144; blue, 226 }  ,draw opacity=1 ]   (160,150.5) -- (160,200.5) ;
\draw [color={rgb, 255:red, 74; green, 144; blue, 226 }  ,draw opacity=1 ]   (250,150.5) -- (250,200.5) ;
\draw [color={rgb, 255:red, 74; green, 144; blue, 226 }  ,draw opacity=1 ]   (280,150.5) -- (280,200.5) ;
\draw [color={rgb, 255:red, 74; green, 144; blue, 226 }  ,draw opacity=1 ]   (310,150.5) -- (310,200.5) ;
\draw [color={rgb, 255:red, 74; green, 144; blue, 226 }  ,draw opacity=1 ]   (340,150.5) -- (340,200.5) ;
\draw [color={rgb, 255:red, 74; green, 144; blue, 226 }  ,draw opacity=1 ]   (370,150.5) -- (370,200.5) ;
\draw [color={rgb, 255:red, 74; green, 144; blue, 226 }  ,draw opacity=1 ]   (370.5,190) -- (388.67,190.15) ;
\draw [shift={(390.67,190.17)}, rotate = 180.47] [color={rgb, 255:red, 74; green, 144; blue, 226 }  ,draw opacity=1 ][line width=0.75]    (10.93,-3.29) .. controls (6.95,-1.4) and (3.31,-0.3) .. (0,0) .. controls (3.31,0.3) and (6.95,1.4) .. (10.93,3.29)   ;
\draw  [fill={rgb, 255:red, 0; green, 0; blue, 0 }  ,fill opacity=1 ] (171.13,166.83) .. controls (171.13,165.88) and (170.36,165.1) .. (169.4,165.1) .. controls (168.44,165.1) and (167.67,165.88) .. (167.67,166.83) .. controls (167.67,167.79) and (168.44,168.57) .. (169.4,168.57) .. controls (170.36,168.57) and (171.13,167.79) .. (171.13,166.83) -- cycle ;
\draw [color={rgb, 255:red, 74; green, 144; blue, 226 }  ,draw opacity=1 ]   (100.5,150.5) -- (100.5,200.5) ;
\draw [color={rgb, 255:red, 74; green, 144; blue, 226 }  ,draw opacity=1 ]   (430.5,150.5) -- (430.5,200.5) ;
\draw   (254.5,131.96) -- (259.64,131.96) -- (259.64,113.06) -- (269.91,113.06) -- (269.91,131.96) -- (275.05,131.96) -- (264.77,144.56) -- cycle ;

\draw (132,68.4) node [anchor=north west][inner sep=0.75pt]    {$U_{j}$};
\draw (99,62.4) node [anchor=north west][inner sep=0.75pt]    {$U'_{j}$};
\draw (378.82,78.9) node [anchor=north west][inner sep=0.75pt]    {$\textcolor[rgb]{0.29,0.56,0.89}{v}$};
\draw (171.82,72.9) node [anchor=north west][inner sep=0.75pt]    {$x$};
\draw (132,162) node [anchor=north west][inner sep=0.75pt]    {$U_{j}$};
\draw (99,152.9) node [anchor=north west][inner sep=0.75pt]    {$U'_{j}$};
\draw (378.82,169.4) node [anchor=north west][inner sep=0.75pt]    {$\textcolor[rgb]{0.29,0.56,0.89}{v}$};
\draw (171.82,165) node [anchor=north west][inner sep=0.75pt]    {$x$};
\draw (281.5,116.5) node [anchor=north west][inner sep=0.75pt]   [align=left] {\tiny{modify}};
\draw (451.5,77.9) node [anchor=north west][inner sep=0.75pt]    {$\gamma _{j}$};
\draw (453,168.4) node [anchor=north west][inner sep=0.75pt]    {$\gamma _{x}$};
\end{tikzpicture}

    \caption{Modify $\gamma_j$ to pass through $x$ while keeping the angle condition \eqref{eq:angle_condition}.}
    \label{fig:angle_keeping_modify}
\end{figure}
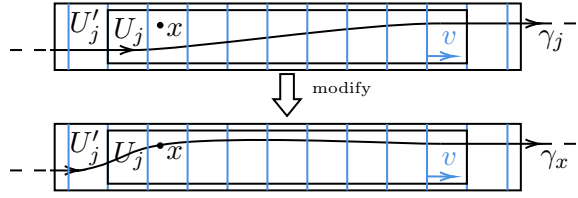
    
    Finally, by construction
    $\mathrm{dist}_g(\gamma_x,Y)\ge \min\{\mathrm{dist}_g(\gamma_j,Y),\mathrm{dist}_g\big(\cup_j\overline{U_j'},\,Y\big)\}$.
    Since $\{U_j\}_j$ is finite and covers $\p U_Y\times (-1,1)$, we can choose $r_0>0$ to be this minimum;
    then $\mathrm{dist}_g(\gamma_x,Y)\ge r_0$ for all $x$, which completes the proof.
\end{proof}

Let $B_{r_0}(Y)$ be the $r_0$-neighborhood of $Y$ under the metric $g$ and let $\mathfrak{d}=\mathrm{dist}_g(\p U_Y\times\{-1\},\p U_Y\times\{1\})$. Denote $M_2'\setminus B_{r_0}(Y)$ by $N_Y$.

\begin{lemma}\label{lem:changing_zero_sets}
    Let $g_2$ be a metric on $M_2$ and let $(v_2,\Sl_2,\I_2)$ be a $\ZT$ closed $1$-form on $M_2'$ such that
    $\Sl_2\Subset B_{r_0}(Y)$, $\I_2\cong \I$ on $N_Y$, and $v_2|_{M_2}$ is harmonic with respect to $g_2$.
    Suppose further that:
    \begin{enumerate}
        \item $v_2$ is exact on $N_Y$ and $v_2=df_2$ for some section of $\I$ on $N_Y$;
        \item $\star_{g_2}v_2$ is exact on a collar neighborhood of $\p M_2$;
        \item $v=df$ on $N_Y$ for a section $f$ of $\I$ and \[
    |f_2-f|<\frac{\mathfrak{d}}{10}|v|,\qquad
    |v_2-v|<\frac14|v|.
    \]
    \end{enumerate}
    Then there exists a $\ZT$ closed $1$-form $(v_3,\Sl_3,\I_3)$, harmonic with respect to a metric $g_3$ on $M$,
    such that $\Sl_3=(\Sl\setminus Y)\cup \Sl_2$.
\end{lemma}
\begin{proof}
Let $\chi(t)$ be a cut-off function such that $\chi(t)\equiv 1$ when $t<-\frac{1}{2}$ and $\chi(t)\equiv 0$ when $t>\frac{1}{2}$.  We can choose $\chi$ so that $|d\chi|\le 2/\mathfrak{d}$. Define $v_3 = d(\chi(t)f_2) + d((1-\chi(t))f)$ on $\partial U_Y \times (-1,1)$. We can extend $v_3$ to all of $M$ by setting $v_3=v$ on $M_1$ and $v_3=v_2$ on $M_2$.

To apply Theorem \ref{thm:gluing}, it suffices to verify the transitivity of $v_3$ on $\p U_Y\times (-1,1)$. By Lemma \ref{lem:positive_curve_on_gluing_region}, for each $x$ in this region $\p U_Y\times (-1,1)$, there exists a $v$-positive loop $\gamma_x$ based at $x$ lies in $M\setminus B_{r_0}(Y)$. We now verify that for each $x$, $\gamma_x$ is $v_3$-positive. When $\gamma_x(t)\notin N_Y$, $v_3(\gamma_x(t))=v(\gamma_x(t))$, hence $|\langle v_3(\gamma_x(t)),\dot{\gamma}_x(t) \rangle|>0$.

When $\gamma_x(t)\in N_Y$, we have\begin{align*}
    |\langle v_3(\gamma_x(t)),\dot{\gamma}_x(t) \rangle|\ge& |\langle v(\gamma_x(t)),\dot{\gamma}_x(t) \rangle|-|\langle (v-v_3)(\gamma_x(t)),\dot{\gamma}_x(t) \rangle|\\
            \ge &   \frac12|v|\cdot |\dot{\gamma}_x|-(|d\chi|\cdot |f_2-f|+|\chi|\cdot |v_2-v|)\cdot |\dot{\gamma}_x|\ge   \frac{1}{20}|v|\cdot |\dot{\gamma}_x|>0.
\end{align*}
Therefore, $\gamma_x$ is $v_3$-positive. Theorem~\ref{thm:gluing} now yields a metric $g_3$ on $M$ such that $v_3$ is harmonic with respect to $g_3$. By definition we have $\Sl_3 = (\Sl\setminus Y)\cup\Sl_2$.
\end{proof}


\subsection{Blowing up isolated ordinary zeros}\label{sec:polynomial}
We first apply the local replacement principle to an isolated ordinary zero. The result says that such a zero can be replaced by a compact \(\mathbb Z/2\) harmonic model on \(\mathbb{R}^n\), provided the model has the same leading homogeneous term at infinity and satisfies the required \(\star\)-exactness condition. We then show, in dimension three, how to create isolated zeros with prescribed zonal harmonic leading term at regular points. Combining these two steps with the models from Appendix~\ref{section:k-nondeg_on_Rn} produces \(k\)-nondegenerate branching circles.

Let $(v,\Sl,\I)$ be a $\ZT$ harmonic $1$-form on a closed manifold $(M,g)$. Suppose there is an isolated zero $p\in M\setminus \Sl$ of $v$. In geodesic coordinates $x_i$ near $p$, we could write
\begin{align}\label{eq:polynomial_v1}v(x)= d\left(P_p(x) + \mathcal{O}(|x|^{k+1})\right),\end{align}
where $P_p$ is a homogeneous harmonic polynomial of degree $k\ge 2$ in the variables $x_i$. The error term $v-dP_p$ is smooth. We assume further that $dP_p(x)\neq 0$ for $x\ne 0$, in particular, $P_{p}(x)$ only have one isolated critical point. Then we prove the following gluing result near isolated zero, which generalized \cite{yan2025construction}.

\begin{theorem}
\label{thm:blowup_isolated_zero}
Suppose there exists a $\ZT$ harmonic function $(u,\Sl',\I')$ on
$(\RR^n,g')$ such that $\Sl'$ is compact, $g'$ is Euclidean outside a
compact set,
\[
  u(y) = P_p(y) + \mathcal{O}(|y|^{k-1})
  \quad\text{as } |y|\to\infty,
\]
and $\star_{g'}\,du$ is exact near infinity.  Then, after replacing a
small neighborhood of $p$ by a suitably scaled copy of this model,
one obtains a metric $g_M$ on $M$ and a $\ZT$ harmonic $1$-form
$(v_M,\Sl_M,\I_M)$ on $(M,g_M)$, whose singular set $\Sl_M$ can be identified with $\Sl \cup \Sl'$.
\end{theorem}

Now we start the proof of Theorem \ref{thm:blowup_isolated_zero}. Let $B_{\epsilon}$ be the $\epsilon$-ball centered at $p$ contained in the geodesic neighborhood. In this subsection, we apply Lemma \ref{lem:changing_zero_sets} to modify $(v,\Sl,\I)$ near $p$ via a $\ZT$ harmonic
function on $\RR^n$ asymptotic to $P_p$, thereby establishing a blow-up result.

In this case, $Y$ in Section \ref{sec:changing_zero_sets} is the single point set $\{p\}$. Correspondingly, let $U_Y=B_{\epsilon}$, and let $\p U_Y\times(-1,1)= B_{2\epsilon}\setminus \overline{B}_{\epsilon/2}$. $M_1=M\setminus B_{3\epsilon/2}$, $M_2=B_{3\epsilon/4}$ and $M_2'=B_{2\epsilon}$. Set $\mathfrak{d}=\mathrm{dist}_g(\p U_Y\times\{-1\},\p U_Y\times\{1\})=\frac{3}{2}\epsilon$.

Apply Lemma \ref{lem:positive_curve_on_gluing_region} to this case, each point in the region $B_{2\epsilon}\setminus \overline{B}_{\epsilon/2}$ is passed by a $v$-positive loop. Moreover, there is an open ball $B_{r_0}$ of radius $r_0>0$ centered at $p$, so that every such $v$-positive loop stays away from $B_{r_0}$. 

Suppose we have the following local model, a $\ZT$ harmonic function $(u,\Sl',\I')$ on $(\RR^n,g')$, satisfying:
\begin{enumerate}
    \item $\Sl'$ is compact and $g'$ is flat outside a compact set;
    \item for $y\in\RR^n$ with $|y|$ sufficiently large, we have $u(y)=P_p(y)+ \mathcal{O}(|y|^{k-1})$;
    \item $\star_{g'}du$ is exact near infinity.
\end{enumerate}

These data provide the $(v_2,\Sl_2,\I_2)$ in Lemma \ref{lem:changing_zero_sets} as we explain now.

Consider an embedding $\phi_R:B_{2\epsilon}\to \RR^n:x\mapsto Rx/\epsilon$. Then for $R>0$ sufficiently large we have \begin{align}d\phi_R^\ast u(x)=\left(\frac{R}{\epsilon}\right)^k d\big(P_p(x)+\mathcal{O}(\frac{\epsilon}{R}|x|^{k-1})\big),\label{eq:polynomial_v2}\end{align}
for $|x|>r_0$. Set $s = 1/R$. Define the metric $g_{2,s} := s^2 \phi_R^\ast g'$, 
and set \[(v_{2,s},\Sl_{2,s},\I_{2,s})=\big((\epsilon s)^k d\phi_R^\ast u,\, \phi_R^{-1}(\Sl'),\, \phi_R^\ast \I'\big)\] 
on $B_{2\epsilon}$. 

\begin{lemma}\label{lem:verify_shrinking_blowup}
    For sufficiently small $\epsilon,s>0$, $(v_{2,s},\Sl_{2,s},\I_{2,s})$ satisfies conditions in Lemma \ref{lem:changing_zero_sets}.
\end{lemma}
\begin{proof}
    Conditions (i) and (ii) are by construction. It suffices to verify condition (iii). On $N_Y=B_{2\epsilon}\setminus B_{r_0}$ we have 
    $f_2=(\epsilon s)^k u\circ \phi_R=P_p(x)+\mathcal{O}(\epsilon s|x|^{k-1})$ and $f= P_p(x)+\mathcal{O}(|x|^{k+1})$. Thus 
        ${|f_2-f|}/{|v|}=C\epsilon s+\mathcal{O}(|x|^2)\le C'\epsilon(s+\epsilon)$,
    where $C,C'>0$ are constants. We can choose both $\epsilon$ and $s$ sufficiently small so that $C'\epsilon(s+\epsilon)<\frac{\mathfrak{d}}{10}=\frac{3}{20}\epsilon$.

    Similarly we have $\frac{|v_2-v|}{|v|}=C\epsilon s\frac{1}{|x|}+\mathcal{O}(|x|)\le C'(s+\epsilon).$
    For sufficiently small $\epsilon,s>0$ we have $|v_2-v|<\frac{1}{4}|v|$.
\end{proof}
Theorem \ref{thm:blowup_isolated_zero} follows straight forward from Lemma~\ref{lem:verify_shrinking_blowup} and 
Lemma~\ref{lem:changing_zero_sets}.

Now we breifly introduce the zonal harmonic polynomial on $\mathbb{R}^3$, which is homogeneous with extra symmetric, precisely disucess in Appendix \ref{section:k-nondeg_on_Rn} and \ref{sec:perturb_zonal_harmonics}. Let \(P_k\) denote the \(k\)-th Legendre polynomial and set $$Z_k(x,y,z)=r^kP_k(z/r),\qquad r=(x^2+y^2+z^2)^{1/2}.$$This is the degree \(k\) zonal harmonic polynomial on \(\mathbb{R}^3\). It is homogeneous and harmonic, invariant under rotations around the \(z\)-axis, and has an isolated critical point at the origin.

Theorem~\ref{thm:blowup_isolated_zero} requires an isolated zero $p$
whose leading homogeneous harmonic polynomial $P_p$ has a unique isolated critical point.
To produce such zeros, we now construct a procedure that adds a
degenerate zero at any point where $v$ is nonvanishing.

Our construction follows Calabi's strategy~\cite{intrinsic-Calabi},
where a pair of Morse zeros was added to a harmonic $1$-form.
In our case, we replace the local model $v=dx$ by a zonal harmonic perturbed by a higher
order term, producing an isolated degenerate zero whose leading term is
exactly the zonal harmonic.

\begin{theorem}\label{thm:adding_degenerate_zero}
		Let $(v,\Sl,\I)$ be a $\ZT$ harmonic $1$-form on a closed $3$-manifold
  $(M,g)$, and let $p\in M\setminus\Sl$ with $v(p)\neq 0$.

   For any $k\ge 3$, one can modify $v$ near $p$
   to obtain a $\ZT$ closed $1$-form $(v',\Sl,\I)$, harmonic with respect to
   a new metric $g'$, such that in sufficiently small geodesic coordinates centered at $p$, $v'(p)=0$, $v'=dZ_k$ on a neighborhood of $p$,
   where $Z_k(x,y,z)=r^k P_k(z/r)$ is the $k$-th zonal harmonic polynomial
   on $\RR^3$.  Moreover, the modification of $v$ may be supported in an arbitrarily
small open ball centered at $p$.
\end{theorem}
\begin{proof}
We construct \((v',\Sl,\I)\) by pasting the local model \(f_k\) from Theorem~\ref{thm:local_model_degenerate_zero} near \(p\), and then verify the three hypotheses of Theorem~\ref{thm:gluing}.

\medskip\noindent\emph{Step 1: Construction of $v'$.}
Since $v(p)\neq 0$ and $p\notin\Sl$, there exists a coordinate chart $U$
centred at $p$ with coordinates $(x,y,z)$ such that $v|_U=dx$ (after possibly
flipping the sign of the $x$-coordinate).  After shrinking $U$, we identify it
with a ball $B_{4\delta}(0)\subset\RR^3$.

Apply Theorem~\ref{thm:local_model_degenerate_zero} with the given $k\ge 3$.
Appendix~\ref{sec:perturb_zonal_harmonics} produces a smooth function
$f_k:\RR^3\to\RR$ satisfying:
\begin{enumerate}[label=(\roman*)]
  \item $f_k\equiv x$ on $\RR^3\setminus B_{3\delta}(0)$;
  \item all critical points of $f_k$ lie in $B_\delta(0)$;
  \item near the origin,
        $f_k(x,y,z)=Z_k(x,y,z)+a r^{m-1}x$,
        where $a>0$ is a constant depending on $\delta$, $r=\sqrt{x^2+y^2+z^2}$, and $Z_k(x,y,z)$
        is the $k$-th zonal harmonic polynomial on $\RR^3$;
  \item every nonzero critical point of $f_k$ is nondegenerate, with Morse index
        $1$ or $2$, and there are exactly $k-1$ such points;
        denote them by $q_1,\dots,q_{k-1}$.
\end{enumerate}
Set $h:=f_k-Z_k$.  By (iii), $h(x,y,z)=a r^{m-1}x$ is a homogeneous polynomial
of degree~$m$ (using that $m$ is odd).

The polynomial $Z_k$ has an isolated critical point at the origin.  By
homogeneity,
\begin{equation}\label{eq:lojasiewicz-Zk}
  |dZ_k(x)|\ge c_0 r^{k-1}\quad\text{for all }x\in\RR^3,
\end{equation}
where $c_0:=\min_{S^2}|d Z_k|>0$.
Moreover $|h(x)|\le a r^m$ and $|dh(x)|\le C_1 r^{m-1}$ for all $r\le 1$,
for some constant $C_1>0$.

Choose a sufficiently small $\varepsilon>0$ such that $\varepsilon<\frac\delta2$ and $0<\varepsilon<\min_{1\le j\le k-1}|q_j|$.

Fix a smooth cutoff $\chi:[0,\infty)\to[0,1]$ with $\chi(t)=1$ for
$t\le\frac12$, $\chi(t)=0$ for $t\ge 1$, and $|\chi'(t)|\le 4$.  Set
$\chi_\varepsilon(r):=\chi(r/\varepsilon)$ and define
\begin{equation}\label{eq:tilde-fk}
  \tilde f_k(x):=Z_k(x)+\bigl(1-\chi_\varepsilon(r)\bigr)h(x).
\end{equation}
Then $\tilde f_k$ is smooth and satisfies:
\begin{itemize}
  \item $\tilde f_k=Z_k$ on $B_{\varepsilon/2}(0)$ and $\tilde f_k=f_k$ on $\RR^3\setminus B_\varepsilon(0)$;
  \item on the annulus $A_\varepsilon:=B_\varepsilon(0)\setminus
        \overline{B_{\varepsilon/2}(0)}$,
        \begin{equation}\label{eq:annulus-grad}
          d\tilde f_k=dZ_k+(1-\chi_\varepsilon)dh-\chi_\varepsilon' h\,dr.
        \end{equation}
\end{itemize}
On $A_\varepsilon$ we have $r\ge\varepsilon/2$, so by \eqref{eq:lojasiewicz-Zk}
$|dZ_k|\ge c_0(\varepsilon/2)^{k-1}$.  The remaining terms in
\eqref{eq:annulus-grad} are bounded by
$|(1-\chi_\varepsilon)dh|+| \chi_\varepsilon' h|
  \le C_1\varepsilon^{m-1}+\frac{4}{\varepsilon}\cdot a\varepsilon^m$. We can choose $\varepsilon$ sufficiently small so that $(C_1+4a)\varepsilon^{m-k}<\frac{c_0}{2^{k-1}}$. 

Then we obtain on $A_\varepsilon$ that
$|d\tilde f_k|\ge c_0(\varepsilon/2)^{k-1}-(C_1+4a)\varepsilon^{m-1}
  >0$.
Hence $\tilde f_k$ has no critical points on $A_\varepsilon$. Since
$\tilde f_k=Z_k$ inside $B_{\varepsilon/2}(0)$, the origin is
the only critical point of $\tilde f_k$ there, and the remaining critical points of $\tilde f_k$ are
exactly $q_1,\dots,q_{k-1}$ in $B_\delta(0)\setminus B_\varepsilon(0)$.

Now define a $\ZT$ closed $1$-form $(v',\Sl,\I)$ on $M$ by
\[
v'=
\begin{cases}
  d\tilde f_k, & \text{on } B_{\varepsilon}(p), \\[2pt]
  df_k,        & \text{on } B_{3\delta}(p)\setminus B_\varepsilon(p), \\[2pt]
  v,           & \text{on } M\setminus B_{3\delta}(p).
\end{cases}
\]
Because $\tilde f_k=f_k$ on $B_{3\delta}\setminus B_\varepsilon(p)$
and $f_k\equiv x$ on $\RR^3\setminus B_{3\delta}(0)$ while $v=dx$ on $U$,
the three expressions match on overlaps; thus $v'$ is a smooth $\I$-valued
closed $1$-form on $M$ with the same flat bundle $\I$ and singular set $\Sl$.
Its ordinary zero set is
$\Zl_{v'}=\Zl_v\cup\mathrm{Crit}(\tilde f_k)=\Zl_v\cup\{p,q_1,\dots,q_{k-1}\}$.

\medskip\noindent\emph{Step 2: Local intrinsic harmonicity and local $\star$-exactness.}
On $B_{\varepsilon/2}(p)$ we have $v'=dZ_k$, which is harmonic with
respect to the Euclidean metric $g_0$. Moreover $\star_{g_0}dZ_k$ is a
closed $2$-form on a $3$-ball, hence exact.  Set
$M_0:=\overline{B_{\varepsilon/2}(p)}$, equipped with
$(v',\varnothing,\I|_{M_0})$ and $g_0$.

Each $q_j$ is a Morse zero of index $1$ or $2$.  A closed
$1$-form with Morse zeros of index strictly between $0$ and $n$
is always locally intrinsically harmonic and locally $\star$-exact
(see the discussion after Definition~\ref{def:intrinsic_harmonic}).  Choose disjoint neighborhoods of
$q_j$ with metrics $g_j$ for which $v'$ is harmonic and $\star$-exact,
and denote their closures by $M_1,\dots,M_{k-1}$.

Now let $M_k:=\overline{M\setminus B_{3\delta}(p)}$ carry the original
triple $(v,\Sl,\I)$ and metric $g$, on which $v$ is already
harmonic with respect to $g$.  Set
$M_{k+1}:=\overline{B_{3\delta}(p)\setminus\bigcup_{j=0}^{k-1}M_j}$,
on which $v'$ is nowhere vanishing and $\Sl=\varnothing$.

We verify that the collection $\{M_i\}_{0\le i\le k+1}$ satisfies
hypotheses~(i)--(iii) of Theorem~\ref{thm:gluing}.
For $0\le j\le k$, each $M_j$ carries a metric making $v'$ harmonic.
Condition~(i) holds because each $M_j$ ($0\le j\le k-1$) contains
exactly one component of $\Zl_{v'}$ in its interior, away from its
collar. $M_k$ contains $\Sl\cup\Zl_v$, which lies outside
$B_{3\delta}(p)$ and thus away from any collar of $\partial M_k$
placed inside $B_{3\delta}(p)$.  

For condition~(ii), on each
$M_j$ ($0\le j\le k-1$) the form $\star_{g_j}v'$ is exact on the whole
piece and hence on its collar; on $M_k$ we have $v'=dx$ in the chart
$U$, and $\star_g dx$ is exact on a collar of $\partial B_{3\delta}(p)$.
Condition~(iii) is verified in Step~3.

\medskip\noindent\emph{Step 3: Transitivity.}

Now we verify condition (iii) of Theorem~\ref{thm:gluing}. It suffices to verify transitivity of $v'$. Let $x\in M$ with $v'(x)\neq 0$. We construct a $v'$-positive loop through $x$.

\smallskip\noindent\textit{Case 1: $x\in M\setminus B_{3\delta}(p)$.}
  Here $v'(x)=v(x)\neq 0$. By Theorem~\ref{thm:harmonic_implies_transitivity}, there exists a $v$-positive loop
  $\gamma$ through $x$. Using Lemma~\ref{lem:modify_positive_loop_in_one_leaf} to push $\gamma$ along slices in $U$, we can modify $\gamma$ to be a $v$-positive loop that avoids $B_{3\delta}(p)$, hence also $v'$-positive.

\smallskip\noindent\textit{Case 2: $x\in B_{3\delta}(p)$.}
On $B_{3\delta}(p)$, the form $v'$ is exact. Indeed, $v'=d\tilde{f}_k$.  The critical points of
$\tilde{f}_k$ consist of $p$ and the Morse zeros
$q_1,\dots,q_{k-1}$, each of Morse index $1$ or $2$.  In particular, $\tilde{f}_k$
has no local maxima or minima.

Let $\ell_x$ be the leaf of $\ker v'$ through $x$.
We claim that $\ell_x\cap\partial B_{3\delta}(p)\neq\varnothing$.

Suppose $\ell_x\cap\partial B_{3\delta}(p)=\varnothing$. Then $\bar{\ell}_x\subset B_{3\delta}(p)$ is compact, and $\bar{\ell}_x\setminus\ell_x$ is contained in $\{p,q_1,\dots,q_{k-1}\}$. Near each critical point of $\tilde{f}_k$, $\ell_x$ is locally a finite union of punctured disks. Fill each punctured point with a distinct new point to obtain a smooth closed oriented surface $\Sigma$ (possibly disconnected). The original closure $\bar{\ell}_x$ is obtained from $\Sigma$ by identifying, for each critical point, the finitely many filled points that correspond to the sheets meeting there.

Let $V\subset \Sigma$ be the set of filled points and $V'\subset\bar{\ell}_x$ their image under the identification. The long exact sequences of the pairs $(\Sigma,V)$ and $(\bar{\ell}_x,V')$ give $H^2(\bar{\ell}_x;\RR)\cong H^2(\Sigma;\RR)\neq0$. By the generalized Jordan curve theorem, $B_{3\delta}(p)\setminus\bar{\ell}_x$ has more than one connected component. Let $\Omega$ be a component with $\partial\Omega\subset\bar{\ell}_x$. Since $\tilde{f}_k$ is constant on $\ell_x$, it is constant on $\partial\Omega$. By compactness it attains an extremum in $\Omega$. This leads to a contradiction, as $\tilde{f}_k$ has no local extremum.

Now we conclude that $\ell_x\cap\partial B_{3\delta}(p)\neq\varnothing$. Since $v'=dx$ on $U\setminus B_{3\delta}(p)$, in the chart $U$, leaves of $\ker v'$ are level sets of the coordinate $x$. Thus $\ell_x$ exits $U$ and continues into
$M\setminus U$ as a regular leaf of $\ker v'$. Let $y\in\ell_x\setminus U$.  By
Case~1, there exists a $v'$-positive loop through $y$.  Pulling this loop
back along the leaf from $y$ to $x$ via
Lemma~\ref{lem:modify_positive_loop_in_one_leaf} yields a $v'$-positive loop
through $x$. This concludes Step 3.

\vspace{1ex}
Now we can apply Theorem~\ref{thm:gluing} to the collection $\{M_i\}_{0\le i\le k+1}$ and $(v',\Sl,\I)$ and produce a Riemannian metric $g'$
on $M$ such that $v'$ is harmonic with respect to $g'$.
On $B_{\varepsilon/2}(p)$ we have $v'=dZ_k$ as required.
\end{proof}

Combining Theorems~\ref{thm:blowup_isolated_zero}
and~\ref{thm:adding_degenerate_zero} with the local models constructed
in Appendix~\ref{section:k-nondeg_on_Rn}, we obtain $\ZT$ harmonic
$1$-forms on closed $3$-manifolds with a prescribed
$\vec{k}$-nondegenerate singular component.

\begin{proposition}
\label{prop:construct-k-nondegenerate}
Let $(M^3,g)$ be a closed Riemannian manifold, and let $k_C\geq 1$.
Suppose that $(M,g)$ admits a nonzero $\ZT$ harmonic $1$-form
$(v,\Sl,\I)$. Then there exist a ball $B\Subset M\setminus\Sl$,
a Riemannian metric $g'$ on $M$, and a $\ZT$ harmonic $1$-form
$(v',\Sl',\I')$ on $(M,g')$ such that
$v'=v$ on $M\setminus B$,
and $\Sl'=\Sl\sqcup C$,
where $C\subset B$ is a smoothly embedded circle. Moreover, $v'$ is
$k_C$-nondegenerate along $C$.
\end{proposition}
\begin{proof}
Let $p\in M\setminus\Sl$ be a regular point of $v$, i.e.,
$v(p)\neq 0$. Choose a sufficiently small ball $B\subset M\setminus\Sl$ centered at $p$ such that $v|_B$ is nonvanishing. Set $k=2k_C+1$.

Applying Theorem~\ref{thm:adding_degenerate_zero} at $p$ with degree
$k$, we obtain a Riemannian metric $g_1$ on $M$ and a $\ZT$ harmonic
$1$-form $(v_1,\Sl,\I)$ on $(M,g_1)$ such that $v_1=v$
on $M\setminus B$, while in sufficiently small geodesic coordinates
centered at $p$ one has $v_1=dZ_k$, where $Z_k$ is the degree-$k$
zonal harmonic polynomial on $\RR^3$. In particular, $p$ is an
isolated zero of $v_1$.

By Proposition~\ref{prop:k_degeneratecoexact}, there exists, up to
multiplication by a nonzero constant, an $SO(2)$-invariant $\ZT$
harmonic function $(\tilde{h}_k,S^1,\tilde\I)$ on $\RR^3$,
where $S^1$ is the unit circle, such that $\tilde{h}_k$ is
asymptotic to $Z_k$ at infinity and $\star d\tilde{h}_k$ is exact
outside a sufficiently large ball. Moreover, $\tilde{h}_k$ is
$\lfloor (k-1)/2\rfloor=k_C$-nondegenerate along $S^1$.

Thus $\tilde h_k$ satisfies the hypotheses of
Theorem~\ref{thm:blowup_isolated_zero} for the isolated zero $p$ of
$v_1$. Applying that theorem, we obtain a Riemannian metric $g'$ on
$M$ and a $\ZT$ harmonic $1$-form $(v',\Sl',\I')$ on $(M,g')$ such
that $\Sl'=\Sl\sqcup C$,
where $C\subset B$ is the image of $S^1$. The construction is supported in $B$, so
$v'=v$ on $M\setminus B$. Since the local model near $C$ is obtained
from $\tilde h_k$ by scaling and pullback, $v'$ is
$k_C$-nondegenerate along $C$.
\end{proof}

\begin{proof}[Proof of Corollary \ref{cor:prescribed-knondegenerate-unknots}]
Since $b_1(M)>0$, after choosing any Riemannian metric on $M$ there
exists a nonzero ordinary harmonic $1$-form $v$. Choose regular
points $p_1,\ldots,p_m$ of $v$ and pairwise disjoint balls
$B_i\ni p_i$ on which $v$ is nonvanishing. Applying
Proposition~\ref{prop:construct-k-nondegenerate} successively in the
balls $B_i$, with $k_C=k_i$ at the $i$-th step, produces a metric
$g'$ and a smooth $\vec k$-nondegenerate $\ZT$ harmonic $1$-form
$(v',\Sl,\I')$ such that
$\Sl=C_1\sqcup\cdots\sqcup C_m$,
where each $C_i\subset B_i$ is an unknot and $v'$ is
$k_i$-nondegenerate along $C_i$.
\end{proof}

\subsection{\texorpdfstring{
Splitting of $\vec{k}$-nondegenerate $\ZT$ harmonic $1$-forms
}{
Splitting of k-nondegenerate $\mathbb Z/2$ harmonic 1-forms
}}
\label{subsection:desingularization_k-nondeg}

Recall from Definition~\ref{def:smooth_k_nondeg} that a smooth
$\ZT$ harmonic $1$-form $(v,\Sl,\I)$ is $\vec{k}$-nondegenerate if,
for every connected component $S_i\subset\Sl$, its local expansion
\eqref{eq:asymptotic_expansion} takes the form $v
=
d\Re\bigl(A_i(t)z^{k_i+\frac12}\bigr)
+\mathcal{O}\bigl(|z|^{k_i+\frac12}\bigr)$, with $A_i(t)$ nonvanishing.
In this subsection, we show that each such higher-order branching
component can be split into $2k_i-1$ nearby nondegenerate components, provided the \emph{winding class} (defined below) of each $A_i$ is zero.

The argument has two steps.  We first remove the remainder term in the
above expansion while leaving the singular set unchanged.  We then
replace the resulting homogeneous model by a product model with
$2k_i-1$ simple branch components, using the local replacement
Lemma~\ref{lem:changing_zero_sets}.

To carry out the fiberwise construction globally, we require a
trivialization of the normal bundle of each singular component.
As observed in \cite[Section 2.3]{HeSalm2026Metric}, the existence of a $\vec{k}$-nondegenerate form implies that $(2k_i+1)c_1\bigl(N(S_i)\bigr)=0$.
Throughout this subsection, we make the simplifying assumption that
$c_1(N(S_i))=0$ for every $i$, so that each normal bundle $N(S_i)$ is
trivial.  This assumption is automatic when $\dim M=3$ or $4$.

We now define the \emph{winding class} associated with the
leading coefficients.  Fix a component $S_i$, write $m_i:=2k_i+1$,
and use the chosen trivialization of $N(S_i)$.  Let
$\{U_\alpha\}_{\alpha}$ be a good cover of $S_i$. Choose local functions
$a_\alpha$ with $a_\alpha^{m_i}=A_i$ on $U_\alpha$.

On every connected overlap $U_{\alpha}\cap U_{\beta}$ one has $a_\alpha=\omega_{\alpha\beta}a_\beta$, where
$\omega_{\alpha\beta}\in\{\exp({2\pi\sqrt{-1}j/m_i}):0\le j\le m_i-1 \}\cong\ZZ/m_i$.
The resulting cocycle defines a class $\mathfrak{m}(A_i)
\in H^1\bigl(S_i;\ZZ/{m_i}\bigr)$,
which we call the \emph{winding class} of the leading coefficient.  It vanishes
if and only if $A_i$ admits a global $m_i$-th root.  Equivalently,
$\mathfrak m(A_i)=0$ if and only if there is a globally defined normal
coordinate $w=a_i^2z$ with $a_i^{m_i}=A_i$.

\begin{lemma}\label{lem:kill_remainder_terms}
Let $(v,\Sl,\I)$ be a $\vec{k}$-nondegenerate $\ZT$ harmonic
$1$-form on $(M,g)$. Then $v$ can be modified, without changing
$\Sl$ or $\I$, to an intrinsically harmonic $\ZT$ closed $1$-form
$(v',\Sl,\I)$ on $M$ such that, in a smaller tubular neighborhood
$U_i'\Subset U_i$ and using the same local coordinates,
\[
v'(z,t)
=
d\Re\bigl(A_i(t)z^{k_i+\frac12}\bigr).
\]
In particular, the remainder term vanishes near every component
$S_i\subset\Sl$.
\end{lemma}

\begin{proof}
Since the components of $\Sl$ have pairwise disjoint tubular
neighborhoods, it suffices to modify $v$ near one component $S_i$.
Write $k=k_i$. After shrinking $U_i$ and
normalizing $E_i(0,t)=0$, the local expansion gives
$v=d\Re\bigl(A_i(t)z^{k+\frac12}+E_i(z,t)\bigr)$,
where $|E_i|\le Cr^{k+\frac32}$ and
$|dE_i|\le Cr^{k+\frac12}$, with $r=|z|$. Since $A_i$ is nowhere
vanishing, we may also assume that
$|v|\ge cr^{k-\frac12}$ on $U_i\setminus S_i$.

Choose $\epsilon>0$ small and apply
Lemma~\ref{lem:changing_zero_sets} with $Y=S_i$ and
$U_Y=B_\epsilon(S_i)\Subset U_i$. Set
$M_1=M\setminus B_{3\epsilon/2}(Y)$,
$M_2=\overline{B_{3\epsilon/4}(S_i)}$, and
$M_2'=B_{2\epsilon}(S_i)$, so that
$\mathfrak d=3\epsilon/2$. By
Lemma~\ref{lem:positive_curve_on_gluing_region}, there is
$r_0\in(0,\epsilon/2)$ with the required positive loop property.
Set $N_Y=B_{2\epsilon}(S_i)\setminus B_{r_0}(S_i)$.

On $M_2'$, define $v_2=d\Re\bigl(A_i(t)z^{k+\frac12}\bigr)$ with $\Sl_2=S_i$ and
$\I_2=\I|_{M_2'\setminus S_i}$.
We next construct a metric for which $v_2$ is harmonic. Put
$m=2k+1$, choose a good cover $\{U_\alpha\}$ of $S_i$ and smooth functions $a_\alpha$ as above. Set $w_\alpha=a_\alpha^2z$. On each connected overlap $U_\alpha\cap U_\beta$ we have
$w_\alpha=\omega_{\alpha\beta}^2w_\beta$ and hence
$|dw_\alpha|^2=|dw_\beta|^2$. The local product metrics $g_{2,\alpha}=g_{S_i}+dw_\alpha\,d\overline{w}_\alpha$
therefore patch to a smooth metric $g_2$ on $M_2'$.

In each such chart, with compatible choices of the half-powers, $A_i(t)z^{k+\frac12}=w_\alpha^{k+\frac12}$.
Thus $v_2=d\Re(w_\alpha^{k+\frac12})$ locally and is harmonic with
respect to $g_2$. Notice that this construction does not require a
globally defined coordinate $w_\alpha$. Morevoer, the form $v_2$ is exact on $N_Y$, with global $\I$-valued potential
$\Re(A_i(t)z^{k+\frac12})$. Moreover, the local product calculation
gives
$\star_{g_2}v_2
=
\pm d\Bigl(
\Im\bigl(A_i(t)z^{k+\frac12}\bigr)
\,\mathrm{vol}_{g_{S_i}}
\Bigr)$
near $\partial M_2$. Hence the exactness conditions in
Lemma~\ref{lem:changing_zero_sets} are satisfied. Moreover,
$\I_2\cong\I$ on $N_Y$.

It remains to check the smallness conditions. On $N_Y$, since
$r\le2\epsilon$, the preceding estimates and the lower bound for
$|v|$ give $$|E_i|\le Cr^2|v|\le C\epsilon^2|v|\text{ and }
|v_2-v|\le|dE_i|\le Cr|v|\le C\epsilon|v|.$$ Thus, for $\epsilon$ sufficiently small,
$|\Re E_i|<\frac{\mathfrak d}{10}|v|$ and
$|v_2-v|<\frac14|v|$. Lemma~\ref{lem:changing_zero_sets} therefore
replaces $v$ near $S_i$ by
$d\Re(A_i(t)z^{k+\frac12})$ without changing either the singular set
or the flat line bundle. Repeating the construction in the pairwise disjoint tubular
neighborhoods of all components of $\Sl$ yields an intrinsically
harmonic $\ZT$ closed $1$-form $(v',\Sl,\I)$ such that
$v'
=
d\Re\bigl(A_i(t)z^{k_i+\frac12}\bigr)$
on some $U_i'\Subset U_i$ for every $i$.
\end{proof}

Assuming the winding class $\mathfrak{m}(A_i)=0$, the local coordinates $w_\alpha$ patch together to be a global transverse complex coordinate $w$ on $N(S_i)$. The preceding lemma then reduces the problem to the homogeneous local model
\[
d\Re\bigl(w^{k_i+\frac12}\bigr)
=
\Re\Bigl(\bigl(k_i+\tfrac12\bigr)
w^{k_i-\frac12}\,dw\Bigr)
\]
near each component $S_i$. We now split the higher-order branch point
in each normal fiber by replacing $w^{k_i-\frac12}$ with a product
having $2k_i-1$ distinct square-root factors.  This perturbation can be confined to an arbitrarily small neighborhood of $S_i$, and Lemma~\ref{lem:changing_zero_sets} then globalizes the resulting local splitting.

\begin{theorem}\label{thm:desingularize_k-nondegenerate}
Let $(v,\Sl,\I)$ be a smooth, $\vec{k}$-nondegenerate $\ZT$ harmonic $1$-form on $(M,g)$, and write $\Sl=\sqcup_{i=1}^m S_i$.
Assume that $\mathfrak{m}(A_i)=0$ for $1\le i\le m$, where $A_i$ is the leading coefficient along $S_i$.

Then there exist a Riemannian metric $g'$ on $M$ and a smooth, nondegenerate $\ZT$ harmonic $1$-form $(v',\Sl',\I')$ on $(M,g')$. Moreover, the splitting can be performed arbitrarily close to $\Sl$, and near each component $S_i$, the singular set $\Sl'$ consists of $2k_i-1$ pairwise disjoint copies of $S_i$.
\end{theorem}
\begin{proof}
We treat the components of $\Sl$ successively.  Fix an untreated
component $S_i$ and start with the harmonic $\ZT$ closed $1$-form
obtained at the preceding stage, taking the original form at the first
stage.  By Lemma~\ref{lem:kill_remainder_terms}, retaining the local
product metric used in its proof, we may replace it, without changing
its singular set or flat line bundle, by an intrinsically harmonic form
$(\tilde v,\Sl,\I)$ with harmonic metric $\tilde g$ such that near
$S_i$,
$\tilde v=d\Re(w^{m_i/2})$ and
$\tilde g=g_{S_i}+dw\,d\bar w$.

Choose $r>0$ so that $B_{3r}(S_i)$ is contained in this product
neighborhood, and set $U_Y=B_r(S_i)$.  By
Lemma~\ref{lem:positive_curve_on_gluing_region}, after possibly
decreasing $r$, there is $r_0\in(0,r/2)$ such that the positive loops
through the gluing region may all be chosen to avoid $B_{r_0}(S_i)$.

Let $\xi_{i,j}~(1\le j\le m_i-2)$, be distinct complex numbers with
$|\xi_{i,j}|\le\varepsilon\ll r_0$.  Let $\I_{(i)}$ be the flat real
line bundle on
$B_{2r}(S_i)\setminus\bigcup_{j=1}^{m_i-2}\{w=\xi_{i,j}\}$ whose
transition functions in the $S_i$-directions agree with those of $\I$
and whose monodromy around each component $\{w=\xi_{i,j}\}$ is $-1$.
Define the $\I_{(i)}$-valued $\ZT$ closed $1$-form
\[
v_{(i)}
=
\Re\left(
\frac{m_i}{2}
\prod_{j=1}^{m_i-2}(w-\xi_{i,j})^{1/2}\,dw
\right).
\]
Its singular set is $S_i'=\bigsqcup_{j=1}^{m_i-2}S_i^j$, where
$S_i^j:=\{w=\xi_{i,j}\}$.  Since $\varepsilon<r_0$, this new singular
set is contained in $B_{r_0}(S_i)$.  Near $S_i^j$ we have
$\prod_{\ell=1}^{m_i-2}(w-\xi_{i,\ell})^{1/2}
=
C_{i,j}(w-\xi_{i,j})^{1/2}
+\mathcal O(|w-\xi_{i,j}|^{3/2})$, where $C_{i,j}\ne0$.  Thus
\[
v_{(i)}
=
d\Re\left(
\frac{m_i}{3}C_{i,j}(w-\xi_{i,j})^{3/2}
+\mathcal O(|w-\xi_{i,j}|^{5/2})
\right),
\]
and each $S_i^j$ is a smooth nondegenerate branching component.

We now check the hypotheses of
Lemma~\ref{lem:changing_zero_sets}.  Put
$M_2=\overline{B_{3r/4}(S_i)}$, $M_2'=B_{2r}(S_i)$, and
$N_Y=B_{2r}(S_i)\setminus B_{r_0}(S_i)$.  On $M_2'$ we use the product
metric $g_{(i)}=g_{S_i}+dw\,d\bar w$.  With respect to this metric,
$v_{(i)}$ is harmonic, since on the associated branched cover it is the
real part of a holomorphic $1$-form in the normal variable and is
independent of the $t$-variables.

The flat bundle $\I_{(i)}$ also agrees with $\I$ on $N_Y$.  Indeed, their
transition functions in the $S_i$-directions agree by construction,
while a small loop in a normal fiber enclosing all the new branch
points has monodromy $(-1)^{m_i-2}=-1$, because $m_i-2$ is odd.  This
is the same normal monodromy as that of the original bundle around
$S_i$.

It remains to verify exactness on $N_Y$.  Write $v_{(i)}=\Re\eta_i$,
where
$\eta_i=\frac{m_i}{2}
\prod_{j=1}^{m_i-2}(w-\xi_{i,j})^{1/2}\,dw$.  Let
$\pi:\hat N_Y\to N_Y$ be the meridional double cover given fiberwise by
$w=\zeta^2$.  In each local parallel trivialization in the
$S_i$-directions, $\pi^*\eta_i$ is a single-valued complex $1$-form
with no $dt$ component.  Moreover,
$\pi^*\eta_i
=
m_i\zeta^{m_i-1}
\left(
\prod_{j=1}^{m_i-2}
(1-\xi_{i,j}\zeta^{-2})
\right)^{1/2}\,d\zeta$.

Since $|\zeta|^2\ge r_0>\varepsilon$ on $\hat N_Y$, the product has a
convergent Laurent expansion involving only powers $\zeta^{-2\ell}$.
Consequently, the Laurent expansion of $\pi^*\eta_i$ contains only
terms of the form $\zeta^{m_i-1-2\ell}\,d\zeta$.  Since $m_i$ is odd,
all these exponents are even and hence none is equal to $-1$.  Thus the
normal-circle period vanishes, so locally over $S_i$ one has
$\pi^*\eta_i=d\hat F_i$.

Since the deck transformation acts by $-1$ on $\pi^*\eta_i$, we may
normalize $\hat F_i$ so that it is anti-invariant.  This normalization
is unique.  Hence, on overlaps of parallel trivializations, the
normalized primitives transform by the same signs as the corresponding
trivializations of $\I$.  They therefore patch, and
$f_{(i)}:=\Re\hat F_i$ descends to an $\I$-valued primitive of
$v_{(i)}$ on $N_Y$.

The same primitive gives the required exactness of the Hodge dual near
$\partial M_2$.  Namely, for the product metric,
$\star_{g_{(i)}}v_{(i)}
=
\pm d(\Im\hat F_i)\wedge\mathrm{vol}_{g_{S_i}}$ on the double
cover, and hence
$\star_{g_{(i)}}v_{(i)}
=
\pm d\bigl((\Im\hat F_i)\mathrm{vol}_{g_{S_i}}\bigr)$
after descending to $N_Y$.

Finally, we check the smallness assumptions.  Let
$\tilde f=\Re(w^{m_i/2})$, so that $\tilde v=d\tilde f$ on $N_Y$.
Choosing the branches on $N_Y$ compatibly with the branch defining
$w^{(m_i-2)/2}$, Taylor expansion gives
$\prod_{j=1}^{m_i-2}(w-\xi_{i,j})^{1/2}
\to w^{(m_i-2)/2}$ in $C^\infty(N_Y)$ as $\varepsilon\to0$.
Consequently, $v_{(i)}\to\tilde v$ in $C^\infty(N_Y)$.  With the above
anti-invariant normalization, the corresponding primitives also satisfy
$f_{(i)}\to\tilde f$ in $C^\infty(N_Y)$.

Since $\tilde v$ is nowhere zero on $N_Y$, taking $\varepsilon$
sufficiently small gives
$|f_{(i)}-\tilde f|<\frac{\mathfrak d}{10}|\tilde v|$ and 
$|v_{(i)}-\tilde v|<\frac14|\tilde v|$.
Thus all the hypotheses of Lemma~\ref{lem:changing_zero_sets} are
satisfied.

Since the chosen tubular neighborhoods are pairwise disjoint, we may
apply the above construction successively to the components of $\Sl$.
 After finitely many steps, we obtain
a metric $g'$ and a smooth nondegenerate $\ZT$ harmonic $1$-form
$(v',\Sl',\I')$, where each $S_i$ is replaced by
$m_i-2=2k_i-1$ nearby copies.
\end{proof}

\begin{remark}
    The above argument also applies to $\vec{k}$-nondegenerate $\ZT$ harmonic $1$-forms on $\mathbb{R}^{n}$, $n\geq 3$, whose branching sets are compact and which are co-exact outside a sufficiently large compact ball, for examples the $\ZT$ harmonic $1$-forms constructed in Appendix \ref{section:k-nondeg_on_Rn}. In particular, the perturbed metric may be chosen to be asymptotic to the Euclidean metric at infinity.
\end{remark}

\section{Desingularization of Graphic Singularities}
\label{section:desingularize_graphical}
In this section we construct $\ZT$ harmonic 1-forms with prescribed local models near the vertices of a graphic singular set. The main step is to replace, in a controlled way, the conical tangent model at each vertex by a scaled resolution model, while preserving the foliation transitivity needed for the Calabi surgery theorem. We then glue these local modifications together and verify that the resulting global form becomes harmonic after a suitable change of metric.
\subsection{Resolution Models and Statement of the Theorem}
\label{subsec:5_resolution_models}

Let $(v,\Sl,\I)$ be a \emph{strongly nondegenerate} graphic $\ZT$ harmonic $1$-form
on $(M,g)$.  For each vertex $x_i$, Section~\ref{subsec:graphic_Z2} associates
the \emph{strongly nondegenerate} tangent cone $(d(\rho^{\mu_i}f_i), C(\vp_i), \I(\vp_i))$,
where $f_i$ is a critical $\ZT$ eigensection on
$\I_{\vp_i}\to S^{n-1}\setminus\vp_i$ with eigenvalue $\lambda_i$ and
$\mu_i=\mu_n(\lambda_i)$.

\begin{definition}[Resolution model]\label{def:resolution_model}
	Let $\bigl(d(\rho^\mu f),\,C(\vp),\,\I(\vp)\bigr)$ be a strongly nondegenerate homogeneous graphic $\ZT$ harmonic $1$-form on
	$\RR^n$, $n\in\{3,4\}$, where $f$ is a critical $\ZT$ eigensection on $\I_\vp\to S^{n-1}\setminus\vp$.
	
	A \emph{resolution model} for $\bigl(d(\rho^\mu f),\,C(\vp),\,\I(\vp)\bigr)$ is a nondegenerate $\ZT$ harmonic potential $(u,\Sl,\I)$ on $(\RR^n,g)$
	satisfying the following conditions
	
	\begin{enumerate}
		\item 
		The metric $g$ equals the Euclidean metric outside a compact set.
		
		\item 
		The singular set $\Sl$ is a smooth codimension-$2$ submanifold of $\RR^n$.
		
		\item 
		Outside a sufficiently large ball $B_{R_0}$, the singular set $\Sl$ is a
		normal graph over $C(\vp)$:
		\begin{equation*}
			\Sl\setminus B_{R_0}
			=
			\{\,y+\mathfrak s(y):y\in C(\vp)\setminus B_{R_0}\,\},
		\end{equation*}
		where $\mathfrak{s}$ is a section of the normal bundle $N(C(\vp))$ defined outside $B_{R_0}$.
		In conical coordinates $y=\rho p$, the graph satisfies, for some
		$\alpha>0$, any $k\ge0$ and some constants $C_k>0$, that $|\widetilde\nabla^k\mathfrak s(\rho,p)|
		\le C_k\rho^{-\alpha-k}$. Here, $\widetilde{\nabla}$ and $|\cdot|$ are taken with respect to the
		conical metric on $C(\vp)$.
		
		Equivalently, there is a straightening diffeomorphism $\Phi:\RR^n\setminus B_{R_0}\longrightarrow
		\RR^n\setminus B_{R_0}$
		mapping $C(\vp)\setminus B_{R_0}$ to $\Sl\setminus B_{R_0}$, satisfying
		$|\widetilde\nabla^k(\Phi-\mathrm{id})|(\rho,\omega)
		\le C_k\rho^{-\alpha-k}$, for any $k\ge 0$ and $(\rho,\omega)\in \{\rho\}\times S^{n-1}$.

		\item Under the diffeomorphism $\Phi$, $\I$ is identified with the model bundle
		$\I(\vp)$.
		
		\item 
		Let $U:=\Phi^*u$.  Write $U=\rho^\mu f+E$. There exists $0<\tau<\mu$ and constants $C_{K,a,\beta}$ such that 
		\begin{equation*}
			|(\rho\partial_\rho)^a\nabla_\omega^\beta E|_K
			\le C_{K,a,\beta}\rho^{\mu-\tau},
			\qquad a+|\beta|\le 2,
		\end{equation*}
		for every $K\Subset S^{n-1}\setminus\vp$. 
		
		\item 
		Let $P$ be a connected component of $\vp$.  Choose local coordinates
		$(t,z)$ on $S^{n-1}$ near $P$, where $z$ is a transverse complex coordinate
		and $P=\{z=0\}$. Write
		\begin{equation*}
			f(t,z)
			=
			\Re\bigl(a_P(t)z^{3/2}\bigr)
			+
			\mathcal{O}(|z|^{5/2}),
		\end{equation*}
		with $a_P(t)$ nonvanishing.
		Let $\zeta=\rho z$ be the physical transverse coordinate near $C(P)$.  Then, near $C(P)$ and
		for $\rho$ sufficiently large,
		\begin{equation}\label{eq:resolution_model_expansion}
			U
			=
			\Re\left(
			\bigl(a_P(t)\rho^{\mu-\frac32}
			+b_P(\rho,t)\bigr)\zeta^{3/2}
			\right)
			+
			E_{P},
		\end{equation}
		where $|E_P|\le C\rho^{\mu-\frac52}|\zeta|^{5/2}$, $|dE_P|\le C(\rho^{\mu-\frac52}|\zeta|^{3/2}+\rho^{\mu-\frac72}|\zeta|^{5/2})$ for some constant $C$ and
		\begin{equation}\label{eq:resolution_error_estimate}
			|\widetilde\nabla^k b_P(\rho,t)|
			\le
			C_k\rho^{\mu-\frac32-\tau-k},
			\qquad k=0,1,2.
		\end{equation}
	\end{enumerate}
\end{definition}

\begin{remark}\label{rem:resolution_model_derived}
	The asymptotic assumptions \eqref{eq:resolution_model_expansion}--\eqref{eq:resolution_error_estimate} in condition
	(vi) of Definition~\ref{def:resolution_model}
	are essentially a corollary of other conditions. They can be derived from elliptic edge theory together with a
	finite growth assumption on $u$.
\end{remark}

Similar to the proof of Lemma~\ref{prop:local_exact_graphic}, we can show that:
\begin{lemma}\label{lem:star-exactness-resmod}
    Let $(u,\Sl,\I)$ be a resolution model for $(d(\rho^\mu f),C(\vp),\I(\vp))$. Then on $\RR^n\setminus B_R$, $\star_{g}du$ is exact, where $B_R$ is an open ball with sufficiently large radius $R$. 
\end{lemma}

\begin{example}\label{ex:zw-c-resolution-model}
Let $c>0$ and set
$q(\zeta,\eta):=\zeta\eta-c$ on $\RR^4=\CC^2$.  Let
$\Sl_c=q^{-1}(0)$, and let $\I_c\to\CC^2\setminus\Sl_c$ be the flat real
line bundle with monodromy $-1$ around a meridian of $\Sl_c$.  Then $u:=\Re(q^{3/2})$
is a nondegenerate $\ZT$ harmonic function with respect to the Euclidean metric.
Indeed, $q$ is holomorphic and
$dq=\eta\,d\zeta+\zeta\,d\eta$ is nowhere zero on $\Sl_c$, so $q$ is a
transverse complex coordinate along the smooth codimension-$2$
submanifold $\Sl_c$.

Write $(\zeta,\eta)=\rho(z,w)$ with $(z,w)\in S^3$, and set
$C(\vp)=\{\zeta\eta=0\}$ and $\vp=C(\vp)\cap S^3$, the Hopf link.
The function
$f(z,w):=\Re\bigl((zw)^{3/2}\bigr)$ is a strongly nondegenerate critical
$\ZT$ eigensection on $\I_\vp$, and
$\rho^3f=\Re\bigl((\zeta\eta)^{3/2}\bigr)$.  Thus $\mu=3$.

We verify that $(u,\Sl_c,\I_c)$ is a resolution model for
$\bigl(d(\rho^3f),C(\vp),\I(\vp)\bigr)$.  The two components
$C_\zeta=\{\zeta=0\}$ and $C_\eta=\{\eta=0\}$ of $C(\vp)$ are disjoint
outside the origin.  Over their ends, $\Sl_c$ is given respectively by
the normal graphs
$\mathfrak s_\zeta(\eta)=c/\eta$ and
$\mathfrak s_\eta(\zeta)=c/\zeta$.  Hence
$|\widetilde\nabla^k\mathfrak s_\zeta|,
|\widetilde\nabla^k\mathfrak s_\eta|
\le C_k\rho^{-1-k}$ for every $k\ge0$, so condition~(iii) holds with
$\alpha=1$.

For $R_0$ sufficiently large, choose a straightening diffeomorphism
$\Phi:\RR^4\setminus B_{R_0}\to\RR^4\setminus B_{R_0}$ which, in conic
neighborhoods of $C_\zeta$ and $C_\eta$, is respectively
$\Phi(\zeta,\eta)
=
\left(\zeta+\frac{c}{\eta},\eta\right)$ and 
$\Phi(\zeta,\eta)
=
\left(\zeta,\eta+\frac{c}{\zeta}\right)$.
Patching these maps with the identity away from the two ends gives
$|\widetilde\nabla^k(\Phi-\mathrm{id})|\le C_k\rho^{-1-k}$ for all
$k\ge0$.  Moreover, $\Phi$ identifies $\I_c$ with $\I(\vp)$.

Set $U:=\Phi^*u$.  On every $K\Subset S^3\setminus\vp$, the function
$zw$ is bounded away from zero, while
$q\circ\Phi=\zeta\eta+\mathcal O(1)$ with the corresponding derivative
estimates.  Thus $U=\rho^3f+E$, where
\[
\bigl|(\rho\partial_\rho)^a\nabla_\omega^\beta E\bigr|_K
\le C_{K,a,\beta}\rho,
\qquad a+|\beta|\le2.
\]
This is condition~(v) with $\tau=2$.

Finally, near either component $P$ of $\vp$, the function
$z_P:=zw$ is a transverse complex coordinate on $S^3$, and
$\zeta_P:=\rho z_P$ is the corresponding physical transverse
coordinate.  By the local definition of $\Phi$, one has
$q\circ\Phi=\zeta\eta$ there, and hence
$U=\Re\bigl(\rho^{3/2}\zeta_P^{3/2}\bigr)$.  Condition~(vi) therefore
holds with $a_P=1$, $b_P=0$, and $E_P=0$.

Consequently, $(u,\Sl_c,\I_c)$ is a resolution model with
$\mu=3$, $\alpha=1$, and $\tau=2$.
\end{example}

Now we state the main theorem of this section.

\begin{theorem}[Desingularization of graphic singularities]\label{thm:desingularize_graphic}
	Let $(M^n,g)$ be a closed oriented Riemannian manifold with $n\in\{3,4\}$.
	Let $(v,\Sl,\I)$ be a strongly nondegenerate graphic $\ZT$ harmonic
	$1$-form on $(M,g)$.
	
	Let $V=\{x_1,\dots,x_m\}$ be the vertex set of $\Sl$. Let $\bigl(d(\rho^{\mu_i}f_i),\,C(\vp_i),\,\I(\vp_i)\bigr)$ be the tangent cone of $(v,\Sl,\I)$ at vertex $x_i$. Suppose that for
	each tangent cone there exists a resolution model $(u_i,\Sl_i,\I_i)$ in the sense of
	Definition~\ref{def:resolution_model}.
	
	Then there exist a Riemannian metric $g'$ on $M$ and a $\ZT$ harmonic
	$1$-form $(v',\Sl',\I')$ on $(M,g')$ such that:
	\begin{enumerate}
		\item $\Sl'$ is a smooth codimension-$2$ submanifold of $M$;
		\item $(v',\Sl',\I')$ is nondegenerate;
		\item the construction is local near the original singular set: for every
		open neighborhood $\mathcal U$ of $\Sl$, the choices can be made so that
		$\Sl'\subset\mathcal U$ and, after identifying the line bundles $\I$ and $\I'$ on
		$M\setminus\mathcal U$, one has $v'=v$ on $M\setminus\mathcal U$.
	\end{enumerate}
\end{theorem}

\subsection{Local and Scaled Estimates}
\label{subsec:5_local_prep}

Fix a vertex $x_i$.  Let $B_i\subset M$ be a sufficiently small geodesic ball around $x_i$.  By Proposition~\ref{prop:local_exact_graphic}, after shrinking $B_i$ there is an $\I$-valued harmonic function $h_i$ on $B_i\setminus\Sl$ with $v=dh_i$.  We pull $h_i$ back by the straightening map $\Gamma_i$ from Definition~\ref{def:graphic_singular_set} and use polar coordinates $(\rho,\omega)\in(0,\varepsilon)\times S_i^{n-1}$.  Thus, in the straightened coordinates,
\begin{equation}\label{eq:local_prep_expansion}
	H_i(\rho,\omega):=\Gamma_i^*h_i(\rho,\omega)
	=\rho^{\mu_i}f_i(\omega)+E_i^{\mathrm{old}}(\rho,\omega).
\end{equation}
Set $F_i=\rho^{\mu_i}f_i$.  The estimates \eqref{eq:relative_estimate_error_vertex} imply, on every compact set $K\Subset S_i^{n-1}\setminus\vp_i$,
\begin{equation}\label{eq:old_error_relative_away}
	|dE_i^{\mathrm{old}}|+\rho^{-1}|E_i^{\mathrm{old}}|=o(|dF_i|)
	\qquad\text{as }\rho\to0.
\end{equation}
Here the norms are computed in the conic metric in the straightened coordinates.  Since the tangent cone is strongly nondegenerate, $dF_i$ has no ordinary zero on the punctured cone away from $C(\vp_i)$.

Let $P\subset\vp_i$ be a connected component corresponding to an edge incident to $x_i$.  Choose local coordinates $(t,z)$ on $S_i^{n-1}$ near $P$, where $z$ is the complex normal coordinate and $P=\{z=0\}$.  The physical transverse coordinate in the ambient manifold is
$\zeta=\rho z$.  

By the edge expansion \eqref{eq:vertex_edge_expansion},\begin{equation*}
	H_i
	=
	\Re\left(
	\bigl(a_{P}(t)\rho^{\mu_i-\frac32}
	+b_{P}^{\mathrm{old}}(\rho,t)\bigr)\zeta^{3/2}
	\right)
	+
	\mathcal{O}\bigl(\rho^{\mu_i-\frac52}|\zeta|^{5/2}\bigr),\end{equation*}
where $b_{P}^{\mathrm{old}}$ is the corresponding contribution of $E_i^{\mathrm{old}}$ to the $\zeta^{3/2}$-coefficient. From \eqref{eq:vertex_edge_error_estimate}, it satisfies
$|\widetilde\nabla^k b_{P}^{\mathrm{old}}(\rho,t)|
\le
C_k\rho^{\mu_i-\frac32+\epsilon_i-\epsilon-k}$,
for $k=0,1,2$ and any $0<\epsilon<\epsilon_i$.

We also record how a resolution model behaves under scaling. Let
$(u_i^{\mathrm{res}},\Sl_i^{\mathrm{res}},\I_i^{\mathrm{res}})$ be the
resolution model at $x_i$, with decay rate $\tau_i>0$ as in
Definition~\ref{def:resolution_model} (v)--(vi). Let $g_{i}^{\mathrm{res}}$ be the corresponding Riemannian metric on $\RR^n$. Fix $R_i\gg1$ and
$0<\delta_i\ll1$, and define
$T_{R_i,\delta_i}(x)=R_ix/\delta_i$. Since $g^{\mathrm{res}}_i$ is
Euclidean outside a compact set, so is $g^{\mathrm{sc}}_i
:=
(\delta_i/R_i)^2
T_{R_i,\delta_i}^\ast g^{\mathrm{res}}_i$.
Set $\Sl_i^{\mathrm{sc}}
:=T_{R_i,\delta_i}^{-1}(\Sl_i^{\mathrm{res}})$ and
$\I_i^{\mathrm{sc}}:=T_{R_i,\delta_i}^*\I_i^{\mathrm{res}}$. Then
$u_i^{\mathrm{sc}}
:=
(\delta_i/R_i)^{\mu_i}
u_i^{\mathrm{res}}\circ T_{R_i,\delta_i}$ is harmonic with respect to $g^{\mathrm{sc}}_i$, and
$(u_i^{\mathrm{sc}},\Sl_i^{\mathrm{sc}},\I_i^{\mathrm{sc}})$ is again
a resolution model of
$(d(\rho^{\mu_i}f_i),C(\vp_i),\I(\vp_i))$.

The corresponding straightening map is
$\Phi_i^{\mathrm{sc}}
:=T_{R_i,\delta_i}^{-1}\circ\Phi_i\circ T_{R_i,\delta_i}$. Set
$U_i^{\mathrm{sc}}
:=(\Phi_i^{\mathrm{sc}})^*u_i^{\mathrm{sc}}$. Writing
$U_i:=\Phi_i^*u_i^{\mathrm{res}}
=\rho^{\mu_i}f_i+E_i$, we obtain
$U_i^{\mathrm{sc}}=\rho^{\mu_i}f_i+E_i^{\mathrm{sc}}$, where
$E_i^{\mathrm{sc}}(\rho,\omega)
=
(\delta_i/R_i)^{\mu_i}
E_i\left(R_i\rho/\delta_i,\omega\right)$.
For every compact set
$K\Subset S_i^{n-1}\setminus\vp_i$,
Definition~\ref{def:resolution_model}(v) gives
\begin{equation}\label{eq:E_i_sc_estimates}
	\sup_{\omega\in K}
	\left|
	(\rho\partial_\rho)^a\nabla_\omega^\beta
	E_i^{\mathrm{sc}}(\rho,\omega)
	\right|
	\le
	C_{K,a,\beta}
	\left(\frac{\delta_i}{R_i}\right)^{\tau_i}
	\rho^{\mu_i-\tau_i},
	\qquad a+|\beta|\le2.
\end{equation}

Let $P$ be a connected component of $\vp_i$. In the corresponding
edge coordinates, Definition~\ref{def:resolution_model}(vi) gives
\[
U_i^{\mathrm{sc}}
=
\Re\left(
\bigl(
a_P(t)\rho^{\mu_i-\frac32}
+b_P^{\mathrm{sc}}(\rho,t)
\bigr)\zeta^{3/2}
\right)
+E_P^{\mathrm{sc}},
\]
where $b_P^{\mathrm{sc}}(\rho,t)
=
(\delta_i/R_i)^{\mu_i-\frac32}
b_P (R_i\rho/\delta_i,t)$
and
\[
E_P^{\mathrm{sc}}(\rho,t,\zeta)
=
\left(\frac{\delta_i}{R_i}\right)^{\mu_i}
E_P\left(
\frac{R_i\rho}{\delta_i},
t,
\frac{R_i\zeta}{\delta_i}
\right).
\]
Consequently,
$|E_P^{\mathrm{sc}}|
\le C\rho^{\mu_i-\frac52}|\zeta|^{5/2}$ and
$|dE_P^{\mathrm{sc}}|
\le
C\left(
\rho^{\mu_i-\frac52}|\zeta|^{3/2}
+
\rho^{\mu_i-\frac72}|\zeta|^{5/2}
\right)$,
while
\begin{equation}\label{eq:b_scP_estimates}
	|\widetilde\nabla^k b_P^{\mathrm{sc}}(\rho,t)|
	\le
	C_k
	\left(\frac{\delta_i}{R_i}\right)^{\tau_i}
	\rho^{\mu_i-\frac32-\tau_i-k},
	\qquad k=0,1,2.
\end{equation}

\subsection{Replacement near the Graphic Vertices}
\label{subsec:5_replace}

We now replace the conical vertex model near each $x_i$ by a scaled model.  The order of choices is fixed throughout this subsection. We first choose $0<\delta_i\ll1$, and then choose $R_i\gg1$.

	Let $(u_i^{\mathrm{res}},\Sl_i^{\mathrm{res}},\I_i^{\mathrm{res}})$ be the resolution model at $x_i$, and let $\Phi_i$ be the straightening map from Definition~\ref{def:resolution_model}.  For every sufficiently large $R_i$, we can always interpolate $\Phi_i$ and $\mathrm{id}$ to be a diffeomorphism $\Psi_{i,R_i}:\RR^n\to\RR^n$ such that
	\begin{enumerate}
		\item $\Psi_{i,R_i}=\mathrm{id}$ on $B_{R_i/2}$;
		\item $\Psi_{i,R_i}=\Phi_i$ on $\RR^n\setminus B_{3R_i/4}$;
		\item on $B_{3R_i/4}\setminus B_{R_i/2}$ one has $|\widetilde\nabla^k(\Psi_{i,R_i}-\mathrm{id})|\le C_kR_i^{-\alpha_i-k}$ for $k=0,1,2$.
	\end{enumerate}
	In particular, $\Psi_{i,R_i}^{-1}(\Sl_i^{\mathrm{res}})$ agrees with $\Sl_i^{\mathrm{res}}$ on $B_{R_i/2}$ and with $C(\vp_i)$ outside $B_{3R_i/4}$.

Let $\chi_i$ be a radial cut-off with $\chi_i=1$ for $\rho\ge2\delta_i$, $\chi_i=0$ for $\rho\le\delta_i$, and $|d\chi_i|\le C\delta_i^{-1}$. The first step is to remove the error term near $x_i$.  In the straightened coordinates near $x_i$ define
\begin{equation}\label{eq:def_v1}
	v_1=d(F_i+\chi_iE_i^{\mathrm{old}}),
\end{equation}
and set $v_1=v$ outside the chosen vertex balls $B_i$. Thus $v_1=v$ for $\rho\ge2\delta_i$ and $v_1=dF_i$ for $\rho\le\delta_i$. The singular set and line bundle are not changed.

Next define $\widetilde\Phi_i^{\mathrm{sc}}:=T_{R_i,\delta_i}^{-1}\circ\Psi_{i,R_i}\circ T_{R_i,\delta_i}$ and $\widetilde U_i^{\mathrm{sc}}=\left(\widetilde\Phi_i^{\mathrm{sc}}\right)^\ast u_i^{\mathrm{sc}}$.
Then $\widetilde U_i^{\mathrm{sc}}=u_i^{\mathrm{sc}}$ for $\rho\le\delta_i/2$, while $\widetilde U_i^{\mathrm{sc}}=U_i^{\mathrm{sc}}$ for $\rho\ge3\delta_i/4$.  On this straightened end we write $E_i^{\mathrm{sc}}:=U_i^{\mathrm{sc}}-F_i$.

Let $\beta_i$ be a radial cut-off with $\beta_i=1$ for $\rho\le 3\delta_i/4$, $\beta_i=0$ for $\rho\ge \delta_i$, and $|d\beta_i|\le C\delta_i^{-1}$.  Define the local replacement potential by
\begin{equation}\label{eq:def_local_replacement_potential}
	U^{\mathrm{new}}_i=
	\begin{cases}
		\widetilde U_i^{\mathrm{sc}}, & 0\le\rho\le 3\delta_i/4,\\
		F_i+\beta_iE_i^{\mathrm{sc}}, & 3\delta_i/4\le\rho\le\delta_i.
	\end{cases}
\end{equation}
This is smooth because the two formulas agree near $\rho=3\delta_i/4$. We now define $(v_2,\Sl_2,\I_2)$ inside the vertex balls. In the straightened spherical coordinates $(\rho,\omega)$ near $x_i$, set
\begin{equation}\label{eq:def_v2}
	v_2=
	\begin{cases}
		dU^{\mathrm{new}}_i, & 0\le\rho\le\delta_i,\\
		v_1, & \rho\ge\delta_i.
	\end{cases}
\end{equation}
Finally, set $v_2=v$ outside the union of the vertex balls. 
The singular set $\Sl_2$ and the line bundle $\I_2$ are defined by replacing the original data $(\Sl,\I)$ near each vertex $x_i$ with the scaled resolution data.

More precisely, in the straightened coordinates near $x_i$, we use $\left(\widetilde{\Phi}_i^{\mathrm{sc}}\right)^\ast(\Sl_i^{\mathrm{sc}},\I_i^{\mathrm{sc}})$ on the core region $\{\rho<3\delta_i/4\}$. On the region $\{\rho\ge 3\delta_i/4\}$ we use the cone data $(C(\vp_i),\I(\vp_i))$, which agrees with the original data $(\Sl,\I)$ after pulling back by the inverse of the straightening map. Thus the new singular set $\Sl_2$ differs from $\Sl$ only in the core regions $\{\rho< 3\delta_i/4\}$ around the vertices.

\begin{proposition}\label{prop:vertex_replacement}
	After choosing $\delta_i>0$ sufficiently small and then $R_i$ sufficiently large for every vertex, the triple $(v_2,\Sl_2,\I_2)$ is a globally defined $\ZT$ closed $1$-form, where $\Sl_2$ is a smooth codimension-$2$ submanifold. Moreover, near every point of $\Sl_2$ there are coordinates $(t,\zeta)$, with $\zeta$ the normal complex coordinate, such that
	\begin{equation}\label{eq:v2_expansion}
		v_2=d\Re\bigl(B(t)\zeta^{3/2}+E_2(\zeta,t)\bigr),
	\end{equation}
	where $B(t)$ is nonvanishing and $E_2=\mathcal{O}(|\zeta|^{5/2})$. 
\end{proposition}

\begin{proof}
	We verify smoothness and nondegeneracy in four regions.
	
	\smallskip
	\noindent\textit{Exterior region.}
	For $\rho\ge2\delta_i$ one has $v_2=v$. The conclusion follows from the edge nondegeneracy of the original strongly nondegenerate graphic form.
	
	\smallskip
	\noindent\textit{Outer annulus $\{\delta_i\le\rho\le2\delta_i\}$.}
	Here $v_2=d(F_i+\chi_iE_i^{\mathrm{old}})$. Let $P\subset\vp_i$ be an incident link component and use the coordinates $(t,z)$ on the link, with physical transverse coordinate $\zeta=\rho z$. The $\zeta^{3/2}$-coefficient of the potential is
	$ a_P(t)\rho^{\mu_i-3/2}+\chi_i b_P^{\mathrm{old}}(\rho,t)$.
	By \eqref{eq:vertex_edge_error_estimate}, $b_P^{\mathrm{old}}(\rho,t)/\rho^{\mu_i-3/2}=\mathcal{O}(\rho^{\epsilon_i-\epsilon})$ for any sufficiently small $\epsilon>0$, as $\rho\to0$.  
	
	Since $a_P(t)$ is nowhere zero, choosing $\delta_i$ sufficiently small makes the old-error contribution $|\chi_i b_P^{\mathrm{old}}|$ smaller than half of $|a_P(t)|\rho^{\mu_i-3/2}$ on the annulus $\{\delta_i\le\rho\le2\delta_i\}$. Thus the $\zeta^{3/2}$-coefficient remains nonzero.
	
	\smallskip
	\noindent\textit{Inner annulus $\{3\delta_i/4\le\rho\le\delta_i\}$.}
	Here $v_2=d\bigl(F_i+\beta_iE_i^{\mathrm{sc}}\bigr)$,
	and the singular set and line bundle are the tangent cone data
	$(C(\vp_i),\I(\vp_i))$. For the same incident link component $P$, the
	$\zeta^{3/2}$-coefficient is
	$a_P(t)\rho^{\mu_i-3/2}+\beta_i b_P^{\mathrm{sc}}(\rho,t)$.
	By \eqref{eq:b_scP_estimates}, on $3\delta_i/4\le\rho\le\delta_i$ one has
	$|b_P^{\mathrm{sc}}(\rho,t)|
	\le C (\delta_i/R_i)^{\tau_i}\rho^{\mu_i-3/2-\tau_i}$.
	
	After $\delta_i$ is fixed, choosing $R_i$ sufficiently large gives
	$|\beta_i b_P^{\mathrm{sc}}(\rho,t)|
	\le \frac12 |a_P(t)|\rho^{\mu_i-3/2}$.
	Thus the leading $\zeta^{3/2}$-coefficient is again nonzero.
	
	\smallskip
	\noindent\textit{Resolution core $\{0\le\rho\le3\delta_i/4\}$.}
	In this region, $v_2=d\widetilde U_i^{\mathrm{sc}}$.
	For $\rho\le\delta_i/2$, this is the scaled resolution model
	$(u_i^{\mathrm{sc}},\Sl_i^{\mathrm{sc}},\I_i^{\mathrm{sc}})$. For
	$\delta_i/2\le\rho\le3\delta_i/4$, it is obtained from the same scaled model
	by the scaled straightening diffeomorphism $\Psi_{i,R_i}^{\mathrm{sc}}$.
	Therefore the nonvanishing of the
	$\zeta^{3/2}$-coefficient is preserved by scaling and by diffeomorphism.
	
	Combining the four cases, the local expansion \eqref{eq:v2_expansion} follows from the nonvanishing of the leading $\zeta^{3/2}$-coefficient in each local model.
\end{proof}

As a result, $\Sl_2$ has trivial normal bundle in $M$.

\begin{remark}
	In Proposition~\ref{prop:vertex_replacement}, the normal coordinate
	$\zeta$ is not meant to be an arbitrary tubular coordinate. The construction
	of $(v_2,\Sl_2,\I_2)$ implicitly fixes an adapted \emph{normal structure} (see \cite[Section 3]{donaldsondeformation2019}) along
	$\Sl_2$. On the exterior region it is the one induced by the original
	metric $g$, on the resolution core it is the one induced by the
	metric of the resolution model, and on the outer and inner annuli these structures are identified
	by the matching maps used in the vertex replacement. 
\end{remark}

\subsection{Normalization along the Smooth Singular Set}
\label{subsec:5_normalize}

Let $V=\{x_1,\dots,x_m\}$ be the vertex set of $\Sl$, and let
$\mathscr{E}_V$ be the set of connected components $e$ of $\Sl\setminus V$
whose closure in $\Sl$ meets $V$.  

After the vertex replacement, $\Sl_2=\Sl$
outside open balls $B_{2\delta_i}(x_i)~(x_i\in V)$. Define $\Sl_2^{\mathrm{nor}}\subset\Sl_2$ to be
the union of those connected components $S$ of $\Sl_2$ for which
$S\cap(e\setminus\bigcup_i B_{2\delta_i}(x_i))\ne\varnothing$ for some
$e\in \mathscr{E}_V$. 
Thus $\Sl_2^{\mathrm{nor}}$ is precisely the part of
$\Sl_2$ corresponding to the old edges incident to vertices, with the relevant
scaled resolution pieces inserted.  If several incident edges are joined by a
resolution model, they simply determine the same component of
$\Sl_2^{\mathrm{nor}}$.  

Choose a $0<\delta_{\mathrm{nor}}\ll1$. For each
$S\in\pi_0(\Sl_2^{\mathrm{nor}})$, let $r_S$ be the distance to $S$ with respect to the original metric $g$. Define tubular neighborhoods $U_S'=\{r_S\le \delta_{\mathrm{nor}}/2\}$ and $U_S=\{r_S<2\delta_{\mathrm{nor}}\}$.
This radius $\delta_{\mathrm{nor}}$ is chosen after all $\delta_i$ and
$R_i$ have been fixed, and can be taken arbitrarily small.

Let $\eta_S$ be a
cut-off on $U_S$ with $\eta_S=0$ on $U_S'$ and $\eta_S=1$ near
$\partial U_S$. On $U_S$ define
\begin{equation}\label{eq:def_vprime}
	v'=d\Re\bigl(B(t)\zeta^{3/2}+\eta_S E_2(\zeta,t)\bigr),
\end{equation}
and set $v'=v_2$ outside $\bigcup_{S\in\pi_0(\Sl_2^{\mathrm{nor}})}U_S$.  Thus
$(v',\Sl_2,\I_2)$ is a globally defined $\ZT$ closed $1$-form.  On $U_S'$ one
has $v'=d\Re(B(t)\zeta^{3/2})$. This is the local remainder-killing modification of
Lemma~\ref{lem:kill_remainder_terms}.  


On the normalization collar
$\mathcal C_S^{\mathrm{nor}}:=U_S\setminus U_S'$, one has
$v'-v_2=d\Re\bigl((\eta_S-1)E_2\bigr)$.
Moreover, if $r=r_S$, then with respect to $g$,
\begin{equation}\label{eq:normalization_error_small}
	|dE_2|+r^{-1}|E_2|=\mathcal{O}(r^{3/2}),\qquad
	|d\Re(B(t)\zeta^{3/2})|\ge c r^{1/2}
\end{equation}
after shrinking $U_S$. Hence the normalization perturbation is lower order
than the leading model on $\mathcal C_S^{\mathrm{nor}}$.

The total modification of $(v,\Sl,\I)$ can be made arbitrarily close to the
original singular set.  Given an open neighborhood $\mathcal U$ of $\Sl$, first
choose the vertex balls and the parameters $\delta_i$ so small that
$\Sl_2\subset\mathcal U$, and then choose all $U_S\Subset\mathcal U$.  Since
$v_2=v$ outside the vertex balls and $v'=v_2$ outside
$\bigcup_{S\in\pi_0(\Sl_2^{\mathrm{nor}})}U_S$, the final form satisfies $v'=v$ on
$M\setminus\mathcal U$ after identifying the restrictions of line bundles.


\subsection{The Transition Region}
\label{subsec:5_transition_region}

For each vertex $x_i$, define the radial transition collar
\begin{equation}\label{eq:radial_transition_collar}
	\mathcal C_i^{\mathrm{rad}}=\{3\delta_i/4<\rho_i<2\delta_i\}.
\end{equation}
It contains both the inner collar $\{3\delta_i/4<\rho_i<\delta_i\}$
and the outer collar $\{\delta_i<\rho_i<2\delta_i\}$.
For each component $S\in\pi_0(\Sl_2^{\mathrm{nor}})$, recall that
$U_S'=\{r_S\le\delta_{\mathrm{nor}}/2\}$ and
$U_S=\{r_S<2\delta_{\mathrm{nor}}\}$.  We write $\mathcal C_S^{\mathrm{nor}}
	=U_S\setminus U_S'
	=\{\delta_{\mathrm{nor}}/2<r_S<2\delta_{\mathrm{nor}}\}$.

The transition region for the final form $v'$ is
\[\mathcal N_T=
	\Big(
	\big(\cup_i \mathcal C_i^{\mathrm{rad}}\big)
	\cup
	\big(\cup_{S\in\pi_0(\Sl_2^{\mathrm{nor}})} U_S\big)
	\Big)
	\setminus
	\left(
	\cup_{S\in\pi_0(\Sl_2^{\mathrm{nor}})} U_S'
	\right).\]
Write $\mathcal N_T^{\mathrm{rad}}
	=
	\left(\cup_i \mathcal C_i^{\mathrm{rad}}\right)
	\setminus
	\left(\cup_{S\in\pi_0(\Sl_2^{\mathrm{nor}})}U_S'\right)$ and
	$\mathcal N_T^{\mathrm{nor}}
	=
	\cup_{S\in\pi_0(\Sl_2^{\mathrm{nor}})}\mathcal C_S^{\mathrm{nor}}$,
then $\mathcal N_T=\mathcal N_T^{\mathrm{rad}}\cup
\mathcal N_T^{\mathrm{nor}}$.  

We either regard the resulting $\mathcal{N}_T$ as manifold with corners or round the
corners inside arbitrarily small neighborhoods.  All estimates below hold on
slightly larger collars, so this rounding does not change nonvanishing,
$\star$-exactness, or the transitivity condition.

 $M\setminus\mathrm{int}(\mathcal N_T)$ has the following
three types of connected components.

First, the \emph{exterior component} $P^{\mathrm{ext}}=\{\rho\ge 2\delta_i,r_S\ge 2\delta_{\mathrm{nor}}\}$.  On this component, $(v',\Sl_2,\I_2)\big|_{P^{\mathrm{ext}}}=(v,\Sl,\I)\big|_{P^{\mathrm{ext}}}$
under the natural identification of line bundles, and the metric is the
original metric $g$.  

Second, near each vertex $x_i$, the \emph{resolution component} $P_i^{\mathrm{res}}:=\{\rho<3\delta_i/4\}\setminus \cup_{S\in \pi_0(\Sl_2^{\rm nor})} U_S$. This is the replacement core left after
removing the normalization collars. On such a piece $P_i^{\mathrm{res}}$ one has $v'\big|_{P_i^{\mathrm{res}}}=\left(\widetilde{\Phi}^{\mathrm{sc}}_i\right)^\ast du_i^{\mathrm{sc}}$.  The singular set and line bundle are the
corresponding restrictions of $(\Sl_2,\I_2)$, and the metric is the scaled
pullback resolution metric $\left(\widetilde{\Phi}_i^{\mathrm{sc}}\right)^\ast g^{\mathrm{sc}}_i$. 

Third, for each $S\in\pi_0(\Sl_2^{\mathrm{nor}})$, we take the
\emph{normalized tube} $U_S'$, equipped with the metric constructed in
the proof of Lemma~\ref{lem:kill_remainder_terms}, applied to the leading
coefficient $B$.  Since $v'=d\Re(B(t)\zeta^{3/2})$
on $U_S'$, the form $v'$ is harmonic with respect to this metric.


These components of $M\setminus\mathrm{int}(\mathcal N_T)$ form the collection of harmonic pieces that will be used in the application of Theorem~\ref{thm:gluing}. We now show that components of $\mathcal{N}_T$ form the collection of transition pieces. First, we prove the following:

\begin{lemma}\label{lem:transition_nonvanishing}
	After choosing all $\delta_i>0$ sufficiently small and then all
	$R_i$ sufficiently large, one can choose
	$\delta_{\mathrm{nor}}>0$ sufficiently small so that
	$(\Sl_2\cup\Zl_{v'})\cap\mathcal N_T=\varnothing$.
	In particular, $v'$ is nowhere vanishing on $\mathcal N_T$.
\end{lemma}
\begin{proof}
For each vertex $x_i$, choose a sufficiently small neighborhood
$\mathcal U_i$ of $\vp_i$ in $S_i^{n-1}$, and set
$K_i:=S_i^{n-1}\setminus\mathcal U_i\Subset
S_i^{n-1}\setminus\vp_i$.  In the straightened coordinates, let
$\mathcal V_i:=\{3\delta_i/4<\rho_i<2\delta_i,\ \omega\in\mathcal U_i\}$.
The edge estimates in the proof of
Proposition~\ref{prop:vertex_replacement} show that $\mathcal U_i$ can be
chosen independently of $\delta_i$ and $R_i$ so that, after choosing
$\delta_i$ sufficiently small and then $R_i$ sufficiently large, $v_2$ is
nonvanishing on $\mathcal V_i\setminus\Sl_2$.

We next show that $v_2$ is nonvanishing on
$\mathcal C_i^{\mathrm{rad}}\setminus\mathcal V_i$, which is identified with
$(3\delta_i/4,2\delta_i)\times K_i$ in the straightened coordinates
$(\rho_i,\omega)$.  Write $F_i=\rho_i^{\mu_i}f_i$.  Since the tangent cone is
strongly nondegenerate, there is a constant $c_i>0$ such that
$|dF_i|\ge c_i\rho_i^{\mu_i-1}$ on
$(3\delta_i/4,2\delta_i)\times K_i$.  On the outer collar
$\{\delta_i\le\rho_i\le2\delta_i\}$, the local potential is
$F_i+\chi_iE_i^{\mathrm{old}}$.  Since $|d\chi_i|\le C\rho_i^{-1}$ there,
estimate~\eqref{eq:old_error_relative_away} gives
$|d(\chi_iE_i^{\mathrm{old}})|\le \frac12|dF_i|$ after choosing
$\delta_i$ sufficiently small.  Hence $d(F_i+\chi_iE_i^{\mathrm{old}})$ has
no zero on the outer collar.

On the inner collar $\{3\delta_i/4\le\rho_i\le\delta_i\}$, the local potential
is $F_i+\beta_iE_i^{\mathrm{sc}}$.  Since $\rho_i\simeq\delta_i$ on this
collar, estimate~\eqref{eq:E_i_sc_estimates} gives
$|dE_i^{\mathrm{sc}}|+\rho_i^{-1}|E_i^{\mathrm{sc}}|
\le C R_i^{-\tau_i}|dF_i|$ on $(3\delta_i/4,\delta_i)\times K_i$.  Together
with $|d\beta_i|\le C\delta_i^{-1}$, this implies
$|d(\beta_iE_i^{\mathrm{sc}})|\le\frac12|dF_i|$ once $R_i$ is sufficiently
large.  Thus $v_2$ has no zero on the inner collar either.  It follows that
$v_2$ is nonvanishing on $\mathcal C_i^{\mathrm{rad}}\setminus\Sl_2$ for every
$i$.  Since $v'=v_2$ on
$\mathcal N_T^{\mathrm{rad}}\setminus\mathcal N_T^{\mathrm{nor}}$, the form
$v'$ is nonvanishing on this region.

It remains to consider $\mathcal N_T^{\mathrm{nor}}$.  Let
$S\in\pi_0(\Sl_2^{\mathrm{nor}})$ and write $r=r_S$.  Choose $\eta_S$ so that
$|d\eta_S|\le C\delta_{\mathrm{nor}}^{-1}$.  On
$\mathcal C_S^{\mathrm{nor}}$, where
$\delta_{\mathrm{nor}}/2<r<2\delta_{\mathrm{nor}}$, formulas
\eqref{eq:def_vprime} and \eqref{eq:normalization_error_small} give
\[
|d(\eta_SE_2)|
\le |dE_2|+|d\eta_S|\,|E_2|
\le Cr^{3/2}+C\delta_{\mathrm{nor}}^{-1}r^{5/2}
\le Cr^{3/2}.
\]
On the other hand, the leading term satisfies
$|d\Re(B(t)\zeta^{3/2})|\ge cr^{1/2}$.  After decreasing
$\delta_{\mathrm{nor}}$, the error term is smaller than one half of the
leading term.  Hence $v'$ is nonvanishing on every
$\mathcal C_S^{\mathrm{nor}}$, and therefore on
$\mathcal N_T^{\mathrm{nor}}$.

Finally, every component of $\Sl_2$ meeting a radial transition collar belongs
to $\Sl_2^{\mathrm{nor}}$ and is contained in the corresponding inner tube
$U_S'$.  Since all such inner tubes are removed in the definition of
$\mathcal N_T$, one has $\Sl_2\cap\mathcal N_T=\varnothing$.  Combining this
with the nonvanishing established above gives
$(\Sl_2\cup\Zl_{v'})\cap\mathcal N_T=\varnothing$.
\end{proof}

Choose $S\in\pi_0(\Sl_2^{\mathrm{nor}})$. Choose $r_1>0$ sufficiently small,
and let $\widetilde{U}_S$ be the $r_1$-tubular neighborhood of $S$ with respect
to the original metric $g$. Then $\widetilde{U}_S$ is embedded and can be
identified with a disk subbundle of the normal bundle $N(S)$.

\begin{lemma}\label{lem:leaf_escaping}
	After decreasing $\delta_{\mathrm{nor}}$, the following holds. If $\ell$ is a
	regular leaf of $\ker v'$ such that
	$\ell\cap\mathcal{C}_S^{\mathrm{nor}}\ne\varnothing$, then $\ell$ meets the
	exterior $M\setminus\big((\cup_i B_i)\cup U_S\big)$.
\end{lemma}
\begin{proof}
Work in the tubular neighborhood $\widetilde U_S$, with $r_1$ chosen sufficiently small. Choose an auxiliary tubular neighborhood $U_S^{(1)}$ such that $U_S'\Subset U_S\Subset U_S^{(1)}\Subset\widetilde U_S$.

Set $F_0:=\Re(B(t)\zeta^{3/2})$. On $\widetilde U_S\setminus S$, write $v'=d(F_0+H)$, where $H$ is an $\I_2$-section given by $\Re(\eta_S E_2)$ on $U_S$ and by $\Re E_2$ outside $U_S$. Thus $H=0$ on $U_S'$, and $|H|=\mathcal O(r^{5/2})$, $|dH|=\mathcal O(r^{3/2})$.

Choose a cut-off function $\chi$ on $\widetilde U_S$ such that $\chi\equiv1$ on $U_S^{(1)}$ and $\chi\equiv0$ near $\partial\widetilde U_S$, and set $\alpha_s=d(F_0+s\chi H)$ for $0\le s\le1$. As in the proofs of Lemmas~\ref{lem:kill_remainder_terms} and \ref{lem:transition_nonvanishing}, after decreasing $r_1$ together with $\delta_{\mathrm{nor}}$, the forms $\alpha_s$ are nowhere zero on $\widetilde U_S\setminus S$.

We apply Moser's trick on $\widetilde U_S\setminus S$. We seek an isotopy
$\varphi_s$ generated by a time-dependent vector field $X_s$ such that
$\varphi_s^*\alpha_s=\alpha_0$. Differentiating with respect to $s$, it
suffices to solve
\begin{equation}\label{eq:Moser_eq}
\mathcal L_{X_s}\alpha_s+d(\chi H)=0.
\end{equation}
Here and below, the equation may be read in any local flat trivialization of
$\I_2$. Since each $\alpha_s$ is closed, Cartan's formula gives
$\mathcal L_{X_s}\alpha_s=d(\alpha_s(X_s))$. Hence
\eqref{eq:Moser_eq} is implied by $\alpha_s(X_s)=-\chi H$.
After fixing an auxiliary metric, this equation is solved locally by
$X_s=-\chi H {\alpha_s^\sharp}/{|\alpha_s|^2}$.

The resulting vector field is globally defined, since both $H$ and
$\alpha_s^\sharp$ change sign under a change of flat trivialization. Moreover,
$\chi$ vanishes near $\partial\widetilde U_S$ and $H$ vanishes near $S$, so
$X_s$ is compactly supported in $\widetilde U_S\setminus S$.

Let $\varphi_s$ be the flow of $X_s$. Then
$\varphi_s^*\alpha_s=\alpha_0$, and the time-one map
$\varphi:=\varphi_1$ satisfies $\varphi^*\alpha_1=\alpha_0=dF_0$.
Since $\chi\equiv1$ on $U_S^{(1)}$, one has $\alpha_1=v'$ there. It follows
that $\varphi^*v'=dF_0$ on $\varphi^{-1}(U_S^{(1)})\setminus S$.

Since $|H|=\mathcal O(r^{5/2})$ and $|\alpha_s|\ge cr^{1/2}$, one has $|X_s|=\mathcal O(r^2)$, uniformly for all sufficiently small $\delta_{\mathrm{nor}}$. Since $B$ is nowhere zero on the compact component $S$, the adapted radius
$\rho_B:=|B(t)|^{2/3}|\zeta|$ is uniformly equivalent to the original
fiber radius. Using $|X_s|=\mathcal O(r^2)$, after decreasing $r_1$ choose
a fixed constant $\rho_0>0$ such that
$U_S^{(2)}:={\rho_B<\rho_0}\Subset\varphi^{-1}(U_S^{(1)})$ for all
sufficiently small $\delta_{\mathrm{nor}}$. After further decreasing
$\delta_{\mathrm{nor}}$, arrange that
$\varphi^{-1}(U_S)\Subset U_S^{(2)}$.

Choose $q\in S$ with positive distance from $\cup_i B_i$, and let $D_q:=\pi^{-1}(q)\cap\overline{U_S^{(2)}}$ be the corresponding normal disk, where $\pi:\widetilde U_S\to S$ denotes the tubular projection. We may assume that $\varphi(\partial D_q)\cap(\cup_i B_i)=\varnothing$.

After decreasing $\delta_{\mathrm{nor}}$ if necessary, choose a small compact annular neighborhood $K_S$ of $\varphi(\partial D_q)$ in $\widetilde U_S\setminus U_S$. This set may be kept fixed under any subsequent decrease of $\delta_{\mathrm{nor}}$. Indeed, the estimate $|X_s|=\mathcal O(r^2)$ implies that the Moser trajectories starting from $\partial D_q$ remain in a fixed annulus disjoint from $U_S$. On this annulus, $H=\Re E_2$, so $X_s$, and hence $\varphi|_{\partial D_q}$, is independent of $\delta_{\mathrm{nor}}$. 

Let $\ell$ be a regular leaf of $\ker v'$ such that $\ell\cap\mathcal C_S^{\mathrm{nor}}\ne\varnothing$, and let $L$ be a connected component of $\ell\cap\widetilde U_S$ intersects with $\mathcal{C}_S^{\mathrm{nor}}$. Choose $x\in L\cap\mathcal C_S^{\mathrm{nor}}$ and set $p=\varphi^{-1}(x)$. Note that $p\in \varphi(U_S)\subset U_S^{(2)}$. Choose a smooth path $\gamma:[0,1]\to S$ from $\pi(p)$ to $q$. Since
$[0,1]$ is contractible and $B$ is nowhere zero, choose a smooth cube root
$a$ of $B\circ\gamma$ and set $w=a(s)^2\zeta$ on the restricted tubular
neighborhood over $\gamma$. Then
$F_0=\Re(w^{3/2})$ and $\rho_B=|w|$.

Writing $p=(\gamma(0),w_0)$ in $U_S^{(2)}$, the path $s\mapsto(\gamma(s),w_0)$ is a lift of $\gamma$ under the projection $\pi$. It lies in the level set of $F_0$ through $p$, and ends in $D_q$. On $D_q$, the
regular level curve of $\Re(w^{3/2})$ through this endpoint meets
$\partial D_q$. Hence the leaf of $\ker dF_0$ through $p$ meets $\partial D_q$.

Since $U_S^{(2)}\Subset\varphi^{-1}(U_S^{(1)})$ and
$\varphi^*v'=dF_0$ there, the image of the path $(\gamma(s),w_0)$ under $\varphi$ lies in
$L$. Hence $L\cap\varphi(\partial D_q)\ne\varnothing$.
Since $\varphi(\partial D_q)\subset K_S$, it follows that
$L\cap K_S\ne\varnothing$. Recall that $K_S$ has no intersections with $U_S$ or $B_i$'s, hence $\ell$ meets the exterior.
\end{proof}

Choose $\delta_i$'s sufficiently small and keep them fixed. The
radial collars $\mathcal C_i^{\mathrm{rad}}$ are then determined by the
original data $(v,\Sl,\I)$ and the original straightening maps near the
vertices, and are therefore independent of the subsequent choices of
$R_i$ and $\delta_{\mathrm{nor}}$.

Retain the angular neighborhoods $\mathcal U_i$ chosen in the proof of
Lemma~\ref{lem:transition_nonvanishing}, shrinking them if necessary, and
let
$\mathcal V_i:=(3\delta_i/4,2\delta_i)\times\mathcal U_i$ as in the proof of Lemma~\ref{lem:transition_nonvanishing}.
Thus $\mathcal V_i$ depends on the fixed parameter $\delta_i$, but not on
the later choices of $R_i$ and $\delta_{\mathrm{nor}}$.  The set
$K_{\mathrm{rad}}
	:=
	\cup_i
	\overline{\mathcal C_i^{\mathrm{rad}}\setminus\mathcal V_i}$
is a compact subset of $M\setminus(\Sl\cup\Zl_v)$.
\begin{lemma}\label{lem:positive_loop_avoid_rad_outside}
	For $\delta_i$ sufficiently small, $R_i$ sufficiently large, and
	$\delta_{\mathrm{nor}}$ sufficiently small, every point
	$x\in K_{\mathrm{rad}}$ is passed by a $v'$-positive loop.
\end{lemma}
\begin{proof}
Applying the argument of
Lemma~\ref{lem:positive_curve_on_gluing_region} to the fixed compact set
$K_{\mathrm{rad}}$, we obtain, for every
$z\in K_{\mathrm{rad}}$, a $v$-positive loop $\gamma_z$ passing
through $z$.  The loops may be chosen so that
\[
	|\langle v(\gamma_z(t)),\dot\gamma_z(t)\rangle|
	\ge \frac12|v(\gamma_z(t))|\,|\dot\gamma_z(t)|
\]
for all $t$.  Moreover, their images are contained in a fixed compact set
$K_\ast\Subset M\setminus\Sl$.  Notice that $K_\ast$, as well as all the
loops $\gamma_z$, has been fixed before choosing the $R_i$ and
$\delta_{\mathrm{nor}}$.

We now increase the $R_i$.  On $K_\ast\setminus\bigcup_iB_i$ one has
$v_2=v$.  On each compact set $K_\ast\cap B_i$, the original estimate
\eqref{eq:old_error_relative_away}, the scaled estimate
\eqref{eq:E_i_sc_estimates}, and the estimates for the interpolated
straightening map $\Psi_{i,R_i}$ show that, after increasing $R_i$, the
replacement data are canonically identified with the original data and
$|v_2-v|\le\frac14|v|$ on $K_\ast$.
We may assume that $K_\ast\cap\Sl_2=\varnothing$.

Finally, decrease $\delta_{\mathrm{nor}}$ so that all normalization tubes
are disjoint from $K_\ast$.  Then $v'=v_2$ on $K_\ast$, and every loop
$\gamma_z$ constructed above is $v'$-positive.
\end{proof}

\begin{proposition}\label{prop:positive_loops_transition}
	For $\delta_i$ sufficiently small, $R_i$ sufficiently large and $\delta_{\mathrm{nor}}$ sufficiently small, every point of
	$\mathcal{N}_T$ is passed by a $v'$-positive loop.
\end{proposition}
\begin{proof}
After making the choices in Lemma~\ref{lem:transition_nonvanishing}, the
transition region contains no point of $\Sl_2$ and no zero of $v'$.  Hence
every point of $\mathcal N_T$ lies on a regular leaf of $\ker v'$.  For points
in $K_{\mathrm{rad}}$, the assertion was already proved in
Lemma~\ref{lem:positive_loop_avoid_rad_outside}; we therefore turn to the case
$x\in\mathcal V_i\setminus\Sl_2$.

Apply Moser's trick in a slightly larger tubular neighborhood
$\mathcal V_i'\supset\mathcal V_i$.  More precisely, after shrinking
$\mathcal V_i$ in Lemma~\ref{lem:positive_loop_avoid_rad_outside} if
necessary, there are coordinates $(t,\zeta)$ on $\mathcal V_i'$ such that
$\Sl_2\cap\mathcal V_i'=\{\zeta=0\}$ and
$v'=d(F_0+H)$ on $\mathcal V_i'\setminus\Sl_2$, where
$F_0=\Re(B(t)\zeta^{3/2})$ and $H$ vanishes in a neighborhood of
$\Sl_2$.  As in the proof of Lemma~\ref{lem:leaf_escaping}, choose a cut-off
function $\chi$ supported in $\mathcal V_i'$ with $\chi\equiv 1$ on
$\mathcal V_i$, and set $\alpha_s=dF_0+s\,d(\chi H)$.  Moser's trick then
gives an isotopy $\bar\varphi_s$ of $\mathcal V_i'$ which is the identity near
$\Sl_2$ and near $\partial\mathcal V_i'$, and whose time-one map satisfies
$\bar\varphi_1^*v'=dF_0$ on
$\bar\varphi_1^{-1}(\mathcal V_i\setminus\Sl_2)$.

Locally $F_0=\Re w_\alpha^{3/2}$, as discussed in
Subsection~\ref{subsection:desingularization_k-nondeg}.  All regular level
sets of this local model are unbounded, and therefore every regular leaf of
$\ker v'$ meeting $\mathcal V_i\setminus\Sl_2$ also meets $K_{\mathrm{rad}}$.
The conclusion for points in $\mathcal V_i\setminus\Sl_2$ now follows from
Lemmas~\ref{lem:positive_loop_avoid_rad_outside} and
\ref{lem:modify_positive_loop_in_one_leaf}.

It remains to treat the normalization collars
$\mathcal N_T^{\mathrm{nor}}$.  Fix
$S\in\pi_0(\Sl_2^{\mathrm{nor}})$ and let
$x\in\mathcal C_S^{\mathrm{nor}}$.  By Lemma~\ref{lem:leaf_escaping}, the
regular leaf $\ell_x$ of $\ker v'$ through $x$ meets the exterior region
$M\setminus\bigl((\cup_iB_i)\cup U_S\bigr)$.  More precisely, the proof of that
lemma shows that the intersection point may be chosen in a fixed compact set
$K_S\Subset M\setminus\bigl((\cup_iB_i)\cup\Sl\cup\Zl_v\bigr)$, defined there.

Since there are only finitely many components $S$, the union
$K_1:=\cup_S K_S$ is compact.  Moreover, $K_1$ is independent of the later
choices of larger $R_i$ and smaller $\delta_{\mathrm{nor}}$.  Applying the thin
flat chart argument of Lemma~\ref{lem:positive_curve_on_gluing_region} to
$K_1$, we obtain, for every $y\in K_1$, a $v$-positive loop $\gamma_y$ through
$y$.  These loops may be chosen with a uniform angle estimate, and all their
images lie in a fixed compact set
$K_\ast\Subset M\setminus(\Sl\cup\Zl_v)$.

By the same compact stability argument as in
Lemma~\ref{lem:positive_loop_avoid_rad_outside}, after increasing the $R_i$,
all the loops $\gamma_y$ remain positive for $v_2$.  After decreasing
$\delta_{\mathrm{nor}}$, the normalization modifications are disjoint from
$K_\ast$, and hence the same loops are $v'$-positive.  Choose
$y\in\ell_x\cap K_S$.  Since $x$ and $y$ lie on the same regular leaf of
$\ker v'$, Lemma~\ref{lem:modify_positive_loop_in_one_leaf} modifies the
$v'$-positive loop $\gamma_y$ to a $v'$-positive loop passing through $x$.
\end{proof}

\begin{proposition}\label{prop:gluing_hypotheses_graphic}
For the choices made above, the decomposition determined by
$\mathcal N_T$ satisfies the hypotheses of Theorem~\ref{thm:gluing}.
\end{proposition}
\begin{proof}
The harmonic pieces are the connected components of
$M\setminus\mathrm{int}(\mathcal N_T)$, equipped with the harmonic
metrics constructed in Subsection~\ref{subsec:5_transition_region}; the
transition pieces are the connected components of $\mathcal N_T$.  By
Lemma~\ref{lem:transition_nonvanishing}, the transition region
$\mathcal N_T$, and hence all gluing collars, contains neither a point of
$\Sl_2$ nor a zero of $v'$.  This verifies the first hypothesis of
Theorem~\ref{thm:gluing}.

We next verify the $\star$-exactness hypothesis for the harmonic pieces, using
the three types listed before Lemma~\ref{lem:transition_nonvanishing}.  On the
exterior component $P^{\mathrm{ext}}$, one has $v'=v$ and the harmonic metric
is the original metric $g$, so the required $\star$-exactness follows directly
from Proposition~\ref{prop:local_exact_graphic}.  On each normalized tube
$U_S'$, with the metric described before
Lemma~\ref{lem:transition_nonvanishing}, the $\star$-exactness of $v'$ follows
from the construction in the proof of
Lemma~\ref{lem:kill_remainder_terms}.

It remains to consider a resolution piece $P_i^{\mathrm{res}}$.  On its radial
end, the required $\star$-exactness follows from
Lemma~\ref{lem:star-exactness-resmod}, and this property is preserved under
scaling and pullback by the matching diffeomorphism.  Along the normalization
tubes, one has exactly the same smooth codimension-two local pieces and
overlaps as in the proof of Proposition~\ref{prop:local_exact_graphic}.
Therefore the Mayer--Vietoris argument used there applies verbatim to a collar
of the entire rounded boundary of $P_i^{\mathrm{res}}$, and shows that the
Hodge dual of $v'$ with respect to the scaled resolution metric is exact on
this collar.

Thus $v'$ is $\star$-exact on every boundary collar of every harmonic piece, so
the second hypothesis of Theorem~\ref{thm:gluing} holds.  Finally,
Proposition~\ref{prop:positive_loops_transition} shows that every point of
every transition piece is passed through by a $v'$-positive loop.  This is
precisely the third hypothesis of Theorem~\ref{thm:gluing}.
\end{proof}

\begin{proof}[Proof of Theorem~\ref{thm:desingularize_graphic}]
Given an open neighborhood $\mathcal U$ of $\Sl$, choose all vertex balls,
transition regions, and normalization tubes inside $\mathcal U$, with the
parameters taken in the order specified above.  The construction then
produces a globally defined closed $\ZT$ $1$-form
$(v',\Sl_2,\I_2)$ such that $v'=v$ on $M\setminus\mathcal U$.

By Proposition~\ref{prop:vertex_replacement}, the definition of a resolution
model, and the normalization construction, the singular set $\Sl_2$ is smooth
and $v'$ is nondegenerate along $\Sl_2$.  Finally,
Proposition~\ref{prop:gluing_hypotheses_graphic} and
Theorem~\ref{thm:gluing} give a Riemannian metric $g'$ for which
$(v',\Sl_2,\I_2)$ is harmonic.  Taking $\Sl'=\Sl_2$ and $\I'=\I_2$ proves
the theorem.
\end{proof}

\appendix

\section{$k$-nondegenerate $\mathbb{Z}/{2}$ harmonic functions on $\mathbb{R}^{n}$}\label{section:k-nondeg_on_Rn}
In this appendix we construct the Euclidean $\ZT$ harmonic models used in Section~\ref{sec:polynomial}, namely \(SO(n-1)\)-invariant examples branching along the unit sphere and asymptotic to zonal harmonic polynomials at infinity.

    \begin{proposition}\label{prop:k_degenerate}
       Fix a decomposition of $\mathbb{R}^{n}$ into $\mathbb{R}^{n-1}\oplus \mathbb{R}$, let $\mathcal{S}$ be a unit sphere in $\mathbb{R}^{n-1}$, then for each $k\in \mathbb{N}_{\geq 2}$, there exists a $SO(n-1)$-invariant $\lfloor \frac{k}{2}\rfloor$-nondegenerate $\ZT$ harmonic function $h_{k}$ that is asymptotic to a degree $k$ polynomial $p_{k}$ at infinty. Moreover, the top degree part of $p_{k}$ is a degree $k$ zonal harmonic on $\mathbb{R}^{n}$, a harmonic function that is invariant under the $SO(n-1)$ action. 
    \end{proposition}
    We follow a similar idea as in \cite{yan2025construction} using separation in an orthogonal coordinate system in $\mathbb{R}^{n}$. Consider the oblate-spherical coordinate on $\mathbb{R}^{n}$ $ (x_{1},\cdots,x_{n-1})=\Theta\sqrt{(1+s^{2})(1-t^{2})}$, and $x_{n}=st$, where $(s,t,\Theta)\in \mathbb{R}\times [-1,1]\times S^{n-2}$. This coordinate defines an orthogonal coordinate on the double branched cover of $\mathbb{R}^{n}$, in which $(s,t,\Theta)$ and $(-s,-t,\Theta)$ map to a same point in $\mathbb{R}^{n}$. In this coordinate, the Laplacian operator takes the form $$\frac{1}{s^{2}+t^{2}}\bigg(\mathscr{L}_{s}+\mathscr{D}_{t}\bigg)+\frac{1}{(1+s^{2})(1-t^{2})}\Delta_{S^{n-2}},$$ where $$\mathscr{L}_{s}=(1+s^{2})\p_{s}^{2}+(n-1)s\p_{s},\;\mathscr{D}_{t}=(1-t^{2})\p_{t}^{2}-(n-1))t\p_{t}.$$ Consider separation of variables of the Laplace equation $ u_{k,d}(s,t,\Theta)=F_{k,d}(s)G_{k,d}(t)\Phi_{d}(\Theta), d\in \mathbb{Z}, k\in \mathbb{N}$. Here, $\Phi_{d}(\Theta)$ is an eigenfunction for $\Delta_{S^{n-2}}$ with eigenvalue $d(d+n-3)$ and $F_{k,d}(s), G_{k,d}(t)$ are solutions to
    {\small{\begin{equation}\label{app:eqn:separation}
            \big(\mathscr{D}_{t}+\frac{d(d+n-3)}{1-t^{2}}+k(k+n-2)\big)G_{k,d}(t)=0,\;\big(\mathscr{L}_{s}-\frac{d(d+n-3))}{1+s^{2}}-k(k+n-2)\big)F_{k,d}(s)=0.
        \end{equation}}\vspace{-0.2cm}}
        
    In this section, we will focus only on the solutions when $d=0$, which is equivalent to the solution being invariant under the $SO(n-1)$ action. When $d=0$, the solutions to the first equations are known as ultra-spherical functions, a generalization of the Legendre functions. More precisely, there are two linearly independent solutions to the equation. The first one is called the ultra-spherical polynomial and can be given explicitly by the Rodriguesâ formula $$P_{k}^{a}(t)=\frac{(-2)^{-k}\Gamma(a+1)\Gamma(k+2a+1)}{\Gamma(2a+1)\Gamma(k+1)\Gamma(k+a+1)(1-t^{2})^{a}}\frac{d^{k}}{d t^{k}}(1-t^{2})^{k+a}.$$ Here, $a=\frac{n-3}{2}$ is a constant depending on dimension. When $n=3$, the above formula gives Legendre polynomials, which for simplicity, we denote as $P_{k}(t)$. The second one is called ultra-spherical function of the second kind, and it is related to the ultra-spherical polynomial using the reduction of second order differential equations $Q_{k}^{a}(t)=P_{k}^{a}(t)\int_{0}^{t} \frac{du}{(1-u^{2})^{a+1}(P^{a}_{k}(u))^{2}}$.
    Comparing the power series expansion for $Q_{k}^{a}(t)$ at $+ i \infty$, one may conclude, when $t \in \mathbb{C}$ and $\Im{t}>0$, $$Q_{k}^{a}(t)-iP_{k}^{a}(t)\int_{0}^{+\infty}\frac{d u}{(1+u^{2})^{a+1}}=c_{k}\int_{-1}^{1}(1-u^{2})^{k+a}(t-u)^{-(k+2a+1)},$$ where $c_{k}=(-1)^{-a}\frac{2^{-(2a+1+k)}\Gamma(k+2a+1)}{\Gamma(a+1)\Gamma{(k+a+1)}}$. On the other hand, apply Wick's rotation $R_{k}^{a}(s)=(-i)^{k}P_{k}^{a}(is)$, $T_{k}^{a}(s)=(-i)^{k+1}Q_{k}^{a}(is)$ we obtain two linearly independent solutions to the second equation in \ref{app:eqn:separation}. We now state several facts about those four kinds of functions, reader should refer to \cite{A_course}.

        \begin{enumerate}
            \item Polynomials $\{P_{k}^{a}(t)\}, k \in \mathbb{N}$ are an orthogonal basis of the Hilbert space $L^{2}([-1,1])$ with weighted measure $(1-t^{2})^{a}d t$.
            \item The functions $P_{k}^{a}(t), Q_{k}^{a}(t)$ satisfy the following recursion formula $ P_{k}^{a}(t)=\frac{1}{k}\big((2k+2a-1)tP_{k-1}(t)-(k+2a-1)P_{k-2}(t)\big)$, $k\geq 2$ and $P_{0}^{a}(t)=1, P_{1}^{a}(t)=(2a+1)t$. The functions $Q_{k}^{a}(t)$ satisfy the same recursion formula but instead $Q_{0}^{a}(t)=H^{a}(t)$ and $Q_{1}^{a}(t)=(2a+1)tH^{a}(t)-(1-t^{2})^{-a}$, where $H^{a}(t)=\int_{0}^{t}(1-u^{2})^{-(a+1)}d u$.
 
            Similarly, using Wick's rotation one can also derive the recursion formula for the functions $R_{k}^{a}(s), Q_{k}^{a}(s)$. Moreover, $T_{0}^{a}(s)=L^{a}(s), T_{1}^{a}=(2a+1)sL^{a}(s)+(1+s^{2})^{-a}$, where $L^{a}(s)=\int_{0}^{s}(1+u^{2})^{-(a+1)}d u$. We conclude from the recursion formula that $P_{k}^{a}(t), R_{k}^{a}(s)$ (respectively, $Q_{k}^{a}(t),T_{k}^{a}(s)$) are even (respectively, odd) functions when $k$ is even and are odd (respectively, even) functions when $k$ is odd.
            \item The product $R_{k}^{a}(s)P_{k}^{a}(t)$ defines a harmonic polynomial of degree $k$ in $\mathbb{R}^{n}$, while $Q_{k}^{a}(t)T_{k}^{a}(s)$ defines a harmonic function on $\mathbb{R}^{n}$ blowing up along the $x_{n}$-axis.
            \item The space $span_{\mathbb{R}}\{R_{0}^{a}(s)P_{0}^{a}(t), \cdots, R_{k}^{a}(s)P_{k}^{a}({t})\}$ is the space of zonal harmonics whose degree is no greater than $k$.
            \item The functions $T_{k}^{a}(s)$ are transcendental and satisfy $T_{k}^{a}(s)=\pm L^{a}(+\infty)R_{k}^{a}(s)+\mathcal{O}(|s|^{-k-1-2a})$, $s\to \pm \infty$.
            \item For each $k$, $ f_{k}:=T_{k}^{a}(s)P_{k}^{a}(t)$ defines a $SO(n-1)$-invariant $\mathbb{Z}/{2}$ harmonic function on $\mathbb{R}^{n}$ that branches along the unit sphere $\mathcal{S}$. Moreover, on a single-valued branch $\{s>0\}$ $f_{k}=L^{a}(+\infty)R_{k}^{a}(s)P_{k}^{a}(t)+\mathcal{O}(|x|^{-k-1-2a})$, $|x|\to +\infty$. Combine with the uniqueness result of the $\mathbb{Z}/{2}$ harmonic function with a fixed branching set and fixed asymptotic behavior at infinity, we conclude that any $SO(n-1)$-invariant $\mathbb{Z}/{2}$ harmonic function that branches along $\mathcal{S}$ and grows no faster than $|x|^{k}$ at infinity has to be a linear combination of $\{f_{0},\cdots, f_{k}\}$.
        \end{enumerate}
    \begin{proof}[Proof of Proposition \ref{prop:k_degenerate}]
        Let $r^{2}=x_{1}^{2}+\cdots+x_{n-1}^{2}$ be the radius function on $\mathbb{R}^{n-1}$ and $ z=\frac{1}{2}(s+it)^{2}=(r-1)+ix_{n}+\mathcal{O}((s^{2}+t^{2})^{2})$ be the complex coordinate near the branching set that is compatible with the normal structure. Notice that
            \begin{equation*}
                \begin{split}
                    & \Re((2z)^{k+1/2})=s^{2k+1}-k(2k+1)s^{2k-1}t^{2}+\cdots +(2k+1)(-1)^{k}st^{2k}.,\\
                    & \mathrm{Im}((2z)^{k+1/2})=(2k+1)s^{2k}t-\frac{(2k+1)2k(2k-1)}{3\cdot 2\cdot 1}s^{2k-2}t^{3}+\cdots+(-1)^{k}t^{2k+1}.
                \end{split}
            \end{equation*}
        Use the first fact for $P_{k}^{a}$ and $R_{k}^{a}$, we know that $f_{k}(s,t)=(-1)^{k}f_{k}(s,-t)=(-1)^{k+1}f_{k}(-s,t)$. Let $g_{k}$ be a real linear combination of the form $g_{k}=\sum_{l\leq k, \;k\equiv l(mod \; 2)} b_{l} f_{l}$. Analysis of the Laplacian operator under this double branched cover shows that there exists a family of nonzero $c_{k}\in \mathbb{R}$ such that near $\mathcal{S}$, $g_{k}=\Re({c_{k}z^{\frac{2m-1}{2}}})+\mathcal{O}(|z|^{\frac{2m+1}{2}})$ when $k$ is even and $g_{k}=\mathrm{Im}({c_{k}z^{\frac{2m'-1}{2}}})+\mathcal{O}(|z|^{\frac{2m'+1}{2}})$ when $k$ is odd. Here, $m,m'$ are some integer dependent on $b_{l}$.

        To prove the Proposition \ref{prop:k_degenerate}, it suffices to show that for each $k$ there is a unique combination such that as $|z|\to 0$ $h_{k}=st^{k}+\cdots+\mathcal{O}(|z|^{\frac{k+3}{2}})$ when $k$ is even and $h_{k}=t^{k} +\cdots+\mathcal{O}(|z|^{\frac{k+2}{2}})$ when $k$ is odd.
 
        The induction formulas imply that $T_{k}^{a}(s)\neq 0$ when $k$ is odd and $T_{k}^{a}(s)=0$, $ (T_{k}^{a})'(s)\neq 0$. We first prove the cases when $k$ are even, and the cases when $k$ are odd are analogous. Consider the Taylor's expansion of $T^{a}_{l}(s), l\leq k$ at $s=0$ $T^{a}_{l}(s)=c_{l,1}s+\cdots +c_{l,k+1}s^{k+1}+\mathcal{O}(s^{k+3})$.

        The $\mathbb{Z}/2$ harmonic solution $f_{l}=T_{l}^{a}(s)P_{l}^{a}(t), l\leq k$ admits the following Taylor's expansion near the singular set $\mathcal{S}$, $f_{l}=c_{l,1}P^{a}_{l}(t)s+\cdots +c_{l,k+1}P_{l}^{a}(t)s^{k+1}+\mathcal{O}(|z|^{\frac{k+3}{2}})$. According to the local expansion of linear combination of $f_{l}$, it suffices to compare the coefficients of the term $s$. It reduces to show that the polynomial $t^{k}$ is a linear combination of $P_{l}^{a}(t), l\leq k$. More precisely, $t^{k}=\sum_{l\leq k, \; k\equiv l\; (mod \;2)} b_{l}P_{l}^{a}(t)$. This follows directly from the fact that the ultra-spherical polynomials $\{P_{l}^{a}(t)\}_{l\leq k}$ is an orthogonal basis for the vector space $span_{\mathbb{R}}\{t^{l}\}_{l\leq k}$ with $L^{2}([-1,1],(1-t^{2})^{a}dt)$-norm.
    \end{proof}

    An important ingredient in blowing up isolated zeros in Theorem \ref{thm:adding_degenerate_zero} is that the gluing models must be $\star$-exact in the gluing region. In this gluing model, we use Euclidean metric to define the $\star$-operator. However, the $\lfloor \frac{k}{2}\rfloor$-nondegenerate $\mathbb{Z}/2$ harmonic $1$-form $d h_{k}$ constructed above is not co-exact outside a large compact ball. To see this, recall the asymptotic expansion for $f_{k}$ in Properties (vi). The differential of the polynomial term $L^{a}(+\infty)R_{k}^{a}(s)P_{k}^{a}(t)$ is co-exact, so the only possible obstruction comes from the remaining term, which is bounded by $\mathcal{O}(|x|^{-k-n+2})$. When $k>0$, this decay rate is faster than $|x|^{-n+2}$. Hence, by the separation of variables, the integral of the remainder over any large round sphere vanishes. It follows that the differential of the remaining term is co-exact.

    On the other hand, $ T_{0}^{a}(s)=L^{a}(+\infty)-\frac{1}{n-2}s^{-n+2}+\mathcal{O}(s^{-n})$ when $s\to +\infty$. We obtain an asymptotic expansion of $f_{0}$ on the branch $\{s>0\}$ $f_{0}=L^{a}(+\infty)-\frac{1}{n-2}|x|^{-n+2}+\mathcal{O}(|x|^{-n})$, $|x|\to \infty$. The non co-exact term in $f_{0}$ is contributed solely by $-\frac{1}{n-2}|x|^{-n+2}$. To conclude, if a linear combination does not include the $f_{0}$ term, then it is coexact.
    \begin{proposition}\label{prop:k_degeneratecoexact}
        For each $k\in \mathbb{N}_{\geq 3}$, there exists a $SO(n-1)$-invariant  $\lfloor \frac{k-1}{2}\rfloor$-nondegenerate $\mathbb{Z}/2$ harmonic function $\tilde{h}_{k}$ that is asymptotic to a degree $k$ zonal harmonic at infinity and the corresponding form is co-exact outside a large ball. 
    \end{proposition}
    \begin{proof}
     When $k$ is odd, the linear combination of $h_{k}$ does not contain $f_{0}$ and therefore it is already co-exact outside a large ball. When $k$ is even, one may define $\tilde{h}_{k}=h_{k}-c_{k}h_{k-2}$. Here, $c_{k}$ is a chosen constant such that let $\tilde{h}_{k}=(c' t^{k}+c''t^{k-2})s+\cdots, \text{ near }\mathcal{S}$, the coefficient $c' t^{k}+c''t^{k-2}$ having vanishing integral on $[-1,1]$ with measure $(1-t^{2})^{a}d t$. By construction, the linear combination of $\tilde{h}_{k}$ does not contain $f_{0}$.
    \end{proof}

    \begin{proposition}
        Let $h_{k}, \tilde{h}_{k}$ be the $\mathbb{Z}/2$ harmonic function constructed in Proposition \ref{prop:k_degenerate} and \ref{prop:k_degeneratecoexact}, then the related $1$-forms are nonvanishing outside a sufficiently large ball in $\mathbb{R}^{n}$.
    \end{proposition}
    \begin{proof}
        The $\mathbb{Z}/2$ harmonic function $h_{k},\tilde{h}_{k}$ is asymptotic to a constant multiple of $R_{k}^{a}(s)P_{k}^{a}(t)$ up a term bounded by $\mathcal{O}(|x|^{k-2})$. On the other hand, $P_{k}^{a}(t)$ has $k$ distinct zeros in the interval $(-1,1)$, since $\{P_{k}^{a}(t)\}$ are orthogonal polynomials. We deduce that in Euclidean coordinate, there exists distinct positive constants $p_{k,l}$
            \begin{equation*}
                 R_{k}^{a}(s)P_{k}^{a}(t)=
                 \begin{cases}
                    c_{k}\prod^{\lfloor\frac{k}{2}\rfloor}_{l=1}\big((p_{k,l}^{2}-1)(p_{k,l}^{2}+x_{n}^{2})+p_{k,l}^{2}r^{2}\big), & k \text{ is even},\\  
                    c_{k}x_{n}\prod^{\lfloor\frac{k}{2}\rfloor}_{l=1}\big((p_{k,l}^{2}-1)(p_{k,l}^{2}+x_{n}^{2})+p_{k,l}^{2}r^{2}\big), & k \text{ is odd}.
                 \end{cases}
            \end{equation*}

        It is a routine computation to check $d (R_{k}(s)P_{k}(t))\neq 0$ when $|x|\gg 1$.
    \end{proof}

    \begin{remark}\label{rem:winding}
       It should be remarked that every $k$-nondegenerate $\ZT$ harmonic $1$-form considered in this appendix has vanishing winding class. Using $z=\frac{1}{2}(s+it)^{2}$ as the normal coordinate, if the function arises from an even combination of $P^{a}_{k}(t)R^{a}_{k}(s)$, then the leading coefficient $A$ may be chosen to be real-valued. If the function arises from an odd combination, then $A$ may be chosen to be purely imaginary. Finally, if the function arises as the sum of an even and an odd combination of different orders, then its leading behavior reduces to one of the previous two cases. If the even and odd parts have the same leading order, then the image of $A$ is contained in a fixed sector of angle at most $\pi/4$. In each case, one can choose a global $(2k+1)$-st root of $A$, and hence the winding class vanishes.

       On the other hand, not every $k$-nondegenerate $\ZT$ harmonic $1$-form has vanishing winding class. When $n=3$, suitable linear combinations of the differentials of the $\ZT$ harmonic functions 
        \[
        e^{i \theta}\sqrt{(1+s^{2})(1-t^{2})}\frac{d R_{k}(s)}{ds}\frac{d P_{k}(t)}{d t},\; k=1,2,3,4.
        \]
        give rise to two nondegenerate $\ZT$ harmonic $1$-forms whose winding class are nontrivial element $1$ and $2$ in $H^{1}(S^{1},\mathbb{Z}/3)$.
    \end{remark}

\section{Perturbation and interpolation of zonal harmonics}\label{sec:perturb_zonal_harmonics}

In this appendix, we construct smooth functions $f_k:\RR^3\to\RR$ which agree with the coordinate function
$x$ outside a compact region, have a unique degenerate critical point at the origin
whose leading term is a homogeneous zonal harmonic $Z_k$ of degree $k\geq3$,
and have exactly $k-1$ additional nondegenerate critical points.  All nondegenerate
critical points have Morse index $1$ or $2$.


Recall that $P_k$ denotes the $k$-th Legendre polynomial, normalized by $P_k(1)=1$ (following the same convention as in Appendix \ref{section:k-nondeg_on_Rn}). It has $k$ distinct roots in $(-1,1)$ and satisfies the differential equation
\[
(1-t^2)P_k''-2tP_k'+k(k+1)P_k=0.
\]
Note that $r$ denotes the full Euclidean radius, rather than the cylindrical
radius $(x_{1}^2+\cdots +x_{n-1}^2)^{1/2}$ used in Appendix~\ref{section:k-nondeg_on_Rn}. Set 
$Z_k(x,y,z)=r^kP_k(z/r)$. By construction, $Z_k$ is a harmonic and homogeneous polynomial on $\RR^3$ with degree $k$. The only critical point of $Z_k$ is the origin, since the Legendre polynomials $P_{k}$ only have simple zeros.
We now investigate the following higher-order perturbation of the zonal harmonic $Z_k~(k\ge 3)$:
\begin{align*}
    G_k(x,y,z)=Z_k(x,y,z)+ar^{m-1}x,\label{eq:higher_perturb}
\end{align*}
where $a>0$ is a constant to be determined and $m$ is an odd integer larger than $k$.


\subsection{Critical point equations}
The function $G_k$ has an isolated and degenerate critical point at origin, since $ar^{m-1}x=\mathcal{O}(r^m)$ and $m>k$.
In this subsection, we derive the critical point equations $d G_k=0$ away from the origin. We introduce $v=x/r, w=y/r, t=z/r$ and set $q=a r^{m-k}>0$, then $G_k=r^k(P_k(t)+qv)$. In addition, we define an auxiliary function
\[
 E_{k,m}(t):=ktP_k(t)+m(1-t^2)P_k'(t).
\]

\begin{lemma}\label{lemma:critcor}
	The nonzero critical points of $G_k$ are in \emph{one-to-one
correspondence} with the roots of $E_{k,m}$ in $(-1,1)$.
\end{lemma}

\begin{proof}
		We first show that the critical points of $G_{k}$ lie in $\{y=0,x\neq 0\}$. Taking the derivatives of $G_{k}$ with respect to $x$ and $y$ gives
\[
v\bigl(kP_k-tP_k'\bigr)
  +q\bigl(1+(m-1)v^2\bigr)=0,\; w\bigl(kP_k-tP_k'+(m-1)qv\bigr)=0,
\]
respectively. Substituting the second equation into the first one forces $q=0$, which contradicts to the assumption $r>0$. Hence at every nonzero critical points $w=0$ and we conclude that $t^2+v^2=1$ at any nonzero critical point of $G_k$. Moreover, $x\ne 0$, otherwise, we would have $v=0$, which implies $q=0$ by the first equation, which is a contradiction. Thus every critical point of $G_k$ satisfies $z/r=t\in (-1,1)$.

Taking the derivative of $z$, we have $ r^{k-1}\partial_z G_k=ktP_k+(1-t^2)P_k'+(m-1)qvt=0$. Assuming $t\neq 0$, we get $q= -(ktP_k+(1-t^2)P_k')/(m-1)vt$. Substituting to the equation $\p_{x}G=0$, we get the \emph{reduced critical point equation}
\begin{equation}\label{eq:E-main}
  E_{k,m}(t)=0.
\end{equation}
It is a routine check that this establish an $1$-$1$ correspondence of the critical points of $G_{k}$ and zeros of $E_{k,m}$ when $t\neq 0$.

Lastly, when $k$ is odd $P_{k}'(0)\neq 0$, $t=0$ is not a solution to (\ref{eq:E-main}) and $\p_{z}G_{k}=0$. On the other hand, when $k$ is even, $t=0$ is a solution to (\ref{eq:E-main}) and there is a unique nonzero solution to the critical equations satisfying $q=|kP_{k}(0)|/m$ and $(v,w,t)=(-\mathrm{sign}(P_{k}(0)),0,0)$.

\end{proof}

In addition, we have the following lemma for the zeros of $E_{k,m}$.
\begin{lemma}\label{lem:simpleness}
    The equation $E_{k,m}(t)=0$ has exactly $k-1$ simple roots in $(-1,1)$. More precisely, let $-1<\alpha_1<\alpha_2<\cdots<\alpha_k<1$ be the roots of $P_k(t)=0$. Then for $1\le j<k$, there exists a unique simple root $t_j$ of $E_{k,m}(t)=0$, in $(\alpha_j,\alpha_{j+1})$.
\end{lemma}
\begin{proof}
    Consider the function $Q_{k,m}(t) := (1-t^2)^{-k/(2m)} P_k(t)$. A direct calculation shows that
\[
Q'_{k,m}(t) = \frac{1}{m} (1-t^2)^{-\frac{k}{2m}-1} E_{k,m}(t).
\]
Thus, the roots of \eqref{eq:E-main} in $(-1,1)$ 
are in one-to-one correspondence with the critical points of $Q_{k,m}(t)$ in $(-1,1)$.
Since $Q_{k,m}(\alpha_j) = Q_{k,m}(\alpha_{j+1}) = 0$, Rolle's Theorem implies that there exists at least one critical point 
of $Q_{k,m}$ in each interval $(\alpha_j, \alpha_{j+1})$.
Consequently, \eqref{eq:E-main} has at least $k-1$ roots in $(-1,1)$.

On the other hand, $E_{k,m}(t)$ is a polynomial of degree $k+1$. Note that $E_{k,m}(1) = k > 0$ and $E_{k,m}(t) \to -\infty$ as $t \to +\infty$. Thus, there exists at least one root of \eqref{eq:E-main} in $(1, +\infty)$.
Similarly, there is at least one root in $(-\infty, -1)$. Therefore, \eqref{eq:E-main} has exactly $k+1$ real roots.
Moreover, since the total number of roots equals the degree, all roots are simple.
In particular, there is exactly one simple root in each interval $(\alpha_j, \alpha_{j+1})$.
\end{proof}


\subsection{Nondegeneracy of critical points}
In this subsection, we verify that every nonzero critical point of $G_k$ is nondegenerate. Moreover, we calculate their Morse indices.

Recall that all nonzero critical points lie in the $xz$-plane, and we label them as $p_{j}$. It is convenient to use the coordinates $(s,t,y) = (\log r, t, y)$ in the following discussion, where $r = \sqrt{x^2+y^2+z^2}$ and $t = z/r$. 
Near $p_j$, use local coordinates $(s,t,y)$ with $s=\log r$. On the $xz$-plane one has
$G_k=e^{ks}P_k(t)+ae^{ms}v(t)$ with $v(t)=\pm\sqrt{1-t^2}$. Since the function $G_{k}$ is even with respect to variable $y$, the mixed partial derivatives $\partial^2_{ys} G_k$ 
and $\partial^2_{yt} G_k$ vanish on the $xz$-plane. Thus the Hessian of $G_k$ at a nonzero critical point $p_j$ has the form 
\begin{align*}
\mathrm{Hess}_{(s,t,y)}G_k(p_j)=\left(\begin{array}{cc}
   \mathrm{Hess}_{(s,t)}G_k(p_j)  & 0 \\
     0 & \partial^2_y G_k(p_j) 
\end{array}\right).   
\end{align*}
For the $(s,t)$-block $\mathrm{Hess}_{(s,t)}G_k$, we compute:
\begin{align*}
    \mathrm{Hess}_{(s,t)}G_k=e^{ks}\left(\begin{array}{cc}
        k(k-m)P_k & kP_k'+mqv' \\ 
        kP_k'+mqv' & P_k''+qv''
    \end{array}\right).
\end{align*}

We now discuss the $(s,t)$-block at each nonzero critical point $p_j=r_j\cdot (v(t_j),0,t_j)$ of $G_k$. Use the relation $v(t)=\pm\sqrt{1-t^{2}}$, we may write $v'=-t/v$ and $v''=-1/v^{3}$. We now claim that at each critical point
\begin{equation}\label{eq:q-formula}
    q v = -\frac{k}{m} P_k(t).
\end{equation}
When $t=0$, this follows directly from $\p_{x}G_{k}=0$, while $t\neq 0$, this follows from $\p_{z}G_{k}=0$ and $E_{k,m}=0$. Using (\ref{eq:q-formula}), we obtain
\[qv'=qv\frac{v'}{v^2}=qv\frac{-t_j}{1-t_j^2}=\frac{kt_jP_k(t_j)}{m(1-t_j^2)}=-P_k'(t_j).\]
Furthermore, since $v''(t)=-1/v^3$, it follows from definition of Legendre polynomials, \eqref{eq:q-formula} and \eqref{eq:E-main}, that
\begin{align*}P_k''(t_j)+qv''(t_j)=&\frac{1}{v^2}\big(2t_jP_k'(t_j)-k(k+1)P_k(t_j)+{kP_k(t_j)}/(mv^2)\big)\\
=&\frac{kP_k(t_j)}{mv^4}\big((m(k+1)-2)t_j^2-m(k+1)+1\big).
\end{align*}
Consequently, the determinant of the $(s,t)$-block is given by
\begin{equation*}
    \det\mathrm{Hess}_{(s,t)}G_k(p_j)=(m-k)\left(r_j^{k}\frac{kP_k(t_j)}{mv^2}\right)^2\left(m^2(k+1)-m-\big(m^2(k+1)-m-k\big)t_j^2\right).
\end{equation*}
The polynomial $m^2(k+1)-m-\big(m^2(k+1)-m-k\big)t^2$ depends linearly on $t^2$. When $t^2=0$ and $t^2=1$, it takes values $m^2(k+1)-m>0$ and $k>0$ respectively. Thus for $0\le t_j^2<1$, $\det\mathrm{Hess}_{(s,t)}G_k(p_j)>0$. As a result, the $(s,t)$-block is positive definite if $P_k(t_j)<0$ and is negative definite if $P_k(t_j)>0$.

On the other hand, direct computation shows \begin{align*}
    \partial^2_y G_k(p_j)=r_j^{k-2}(kP_k-tP_k'+(m-1)qv)=kr_j^{k-2}P_k(t_j)/(m(1-t_j^2))\ne 0.
\end{align*}

We summarize the results in the following lemma:
\begin{lemma}\label{lem:morse_index}
Every nonzero critical point of $G_k$ is nondegenerate. Let $p_j = r_j(v_j, 0, t_j)$ be a nonzero critical point of $G_k$. 
If $P_k(t_j) > 0$, then the eigenvalues of $\mathrm{Hess}_{(s,t,y)} G_k(p_j)$ have signs $(-,-,+)$, 
so $p_j$ has Morse index $2$. If $P_k(t_j) < 0$, then the eigenvalues have signs $(+,+,-)$, 
so $p_j$ has Morse index $1$. 

Consequently:
\begin{itemize}
    \item If $k$ is odd, there are exactly $\frac{k-1}{2}$ nondegenerate critical points of Morse index $1$, 
          and $\frac{k-1}{2}$ of Morse index $2$.
    \item If $k$ is even, there are $\frac{k}{2}$ nondegenerate critical points of Morse index $1$, 
          and $\frac{k}{2} - 1$ of Morse index $2$.
\end{itemize}
\end{lemma}

\subsection{Interpolation with the coordinate function $x$}
We now construct the desired functions $f_k: \RR^3 \to \RR$ that coincide with $G_k$ on an open ball 
and with the coordinate function $x$ outside a larger ball.

Let $M$ be the constant satisfying $|\partial_xZ_k|\le M r^{k-1}$. And let $Q>2^{k-1}M$ be a constant sufficiently large. Next choose $0<\delta\ll1$ so that $3\delta<1$ and $Q\delta^{k-1}<1$. Choose a smooth radial function $\lambda:[0,\infty)\to\mathbb{R}$ satisfying:
\[
  \lambda(r)=\begin{cases}
      Q\delta^{k-m}r^{m-1}&,0\le r\le\delta\\
      1 &,r\ge 2\delta.
  \end{cases}
\]
We may assume that $\lambda'(r)\ge0$ for all $r> 0$.
We also define a cut-off function $\mu$ such that $\mu(r)=1$ for $0\le r\le2\delta$ and $\mu(r)=0$ for $r\ge3\delta$. We can assume $|\mu'(r)|\le 6/\delta$ on $(2\delta,3\delta)$. Define
\begin{equation}\label{eq:fk-def}
  f_k(x,y,z)=\mu(r)Z_k(x,y,z)+\lambda(r)x.
\end{equation}

For $0\le r\le\delta$, one has $f_k=Z_k+Q\delta^{k-m}r^{m-1}x$. Thus $f_k$ coincides with the $G_k$ considered previously, with the parameter choice $a = Q \delta^{k-m}$. 
With this choice of $a$, nonzero critical points of $G_k$ are therefore $p_j = \delta\left({q_j}/{Q}\right)^{1/(m-k)}(v_j,0,t_j),$ where $q_j$ is determined by \eqref{eq:q-formula} 
and is therefore independent of $a$. Consequently, for sufficiently large $Q$, 
all these critical points lie strictly inside the ball $B_\delta(0)$.

Now we show that when $r\ge\delta$, $f_k$ has no critical points. For $\delta\le r\le2\delta$, one has $\mu=1$, so $f_k=Z_k+\lambda(r)x$. Then $\partial_x f_k=\partial_xZ_k+\lambda(r)+\lambda'(r){x^2}/{r}$. Since $\lambda'\ge0$, the last term is nonnegative. Moreover, by assumption $|\partial_x Z_k|\le M r^{k-1}\le M(2\delta)^{k-1}$, while $\lambda(r)\ge\lambda(\delta)=Q\delta^{k-1}>M(2\delta)^{k-1}$.
Therefore $\partial_xf_k>0$ for $\delta\le r\le2\delta$, so no critical point occurs in this annulus.

Meanwhile, for $2\delta\le r\le3\delta$, one has \(\lambda=1\), so $f_k=x+\mu(r)Z_k$. Thus $\partial_x f_k= 1+\mu\partial_x Z_k+\mu'Z_kx/r$.
Note that $|Z_k|=\mathcal{O}(r^k)$, $|\nabla Z_k|=\mathcal{O}(r^{k-1})$ and $|\mu'|\le 6/\delta$, we have $|\partial_x f_k-1|=\mathcal{O}(\delta^{k-1})$. Taking $\delta>0$ sufficiently small then we have $\partial_x f_k>0$ for $2\delta\le r\le3\delta$.

Finally, for $r\ge3\delta$, we have $f_k\equiv x$. In summary, we have the following result:
\begin{theorem}\label{thm:local_model_degenerate_zero}
Let $k\ge3$ and let $m>k$ be odd. For all sufficiently small
$\delta>0$, there exists a smooth function
$f_k:\RR^3\to\RR$ such that:
\begin{enumerate}
    \item $f_k=x$ on $\RR^3\setminus B_{3\delta}(0)$;
    \item on $B_\delta(0)$, $f_k=Z_k+a r^{m-1}x$ for some $a>0$;
    \item $\mathrm{Crit}(f_k)=\{0,p_1,\ldots,p_{k-1}\} \subset B_\delta(0)$;
\end{enumerate}
More precisely, the origin is degenerate, while every $p_j$ is nondegenerate and has Morse index $2$ when $P_k(t_j)>0$ and index $1$ when $P_k(t_j)<0$.
\end{theorem}

	\bibliographystyle{alpha}
	\bibliography{references}

\end{document}